\documentclass[11pt]{article}

\newcommand{\ZA}[1]{{\color{black}{#1}}}      
\newcommand{\ZAA}[1]{{\color{black}{#1}}} 
\newcommand{\ZAB}[1]{{\color{black}{#1}}}   
\newcommand{\ZAP}[1]{{\color{black}{#1}}}   

\usepackage[utf8]{inputenc}

\usepackage{graphicx}
\usepackage{mathtools}
\usepackage{amsmath}
\RequirePackage{amssymb}
\RequirePackage{amsthm}
\usepackage{amsmath,scalefnt}
\usepackage{footnote}
\usepackage{fmtcount}
\usepackage{threeparttable}
\usepackage{caption}
\usepackage{subcaption}
\usepackage{color}
\usepackage{tikz}
\usetikzlibrary{positioning,chains,fit,shapes,calc}
\usetikzlibrary{arrows}
\tikzstyle{vertex}=[circle, draw, inner sep=0pt, minimum size=6pt]
 \usepackage{enumitem}
 \usepackage{multicol}
 \usepackage{wrapfig}
\usepackage{enumitem}

\newcommand{\abs}[1]{\left\vert #1 \right\vert}

\newcommand{\notforthesis}[1]{}

	\newtheorem{definition}{Definition}
	\newtheorem{claim}{Claim}
	\newtheorem{lemma}{Lemma}

	 \newtheorem{assumption}{Assumption}
	 \newtheorem{proof of proposition}{Proof of Proposition}
	\newtheorem{proof of claim}{Proof of Claim}
	\newtheorem{proposition}{Proposition} 
	\newtheorem{theorem}{Theorem}
	\newtheorem{notation}{Notation}
        \newtheorem{corollary}{Corollary}
        \newtheorem{remark}{Remark}

\newcommand{\sech}{\mbox{sech}\,}

\renewcommand{\r}{{\mathbb R}}
\renewcommand{\a}{{\alpha}}

\renewcommand{\o}{{\omega}}

\newcommand{\w}{{w}}   
\renewcommand{\c}{{c}}  
\newcommand{\x}{{\xi}} 
\newcommand{\et}{{\eta}}

\newcommand{\ee}{\end{equation}}
\newcommand{\bal}{\begin{aligned}}
\newcommand{\eal}{\end{aligned}}
\newcommand{\bi}{\begin{itemize}}
\newcommand{\ei}{\end{itemize}}
\newcommand{\ben}{\begin{enumerate}}
\newcommand{\een}{\end{enumerate}}
\newcommand{\beqn}{\begin{eqnarray*}}
\newcommand{\eeqn}{\end{eqnarray*}}

\newcommand{\be}[1]{\begin{equation}\label{#1}}
\newcommand{\bp}{\begin{proof}}
\newcommand{\ep}{\end{proof}}
\newcommand{\bremark}{\begin{remark}\rm } 
\newcommand{\eremark}{\end{remark}}
\newcommand{\blem}{\begin{lemma}}
\newcommand{\elem}{\end{lemma}}
\newcommand{\bclaim}{\begin{claim}}
\newcommand{\eclaim}{\end{claim}}
\newcommand{\bnote}{\begin{notation}}
\newcommand{\enote}{\end{notation}}

\newcommand{\bthm}{\begin{theorem}}
\newcommand{\ethm}{\end{theorem}}
\newcommand{\bprop}{\begin{proposition}}
\newcommand{\eprop}{\end{proposition}}
\newcommand{\bcor}{\begin{corollary}}
\newcommand{\ecor}{\end{corollary}}
\newcommand{\dis}{\displaystyle}
\newcommand{\lt}{\left}
\newcommand{\rt}{\right}

\newcommand{\FR}{A_{FR} }
\newcommand{\FL}{A_{FL} }
\newcommand{\BR}{A_{BR} }
\newcommand{\BL}{A_{BL} }
\newcommand{\T}{A_{Tri} }

\title{Gait transitions in a phase oscillator model \\ of an insect  central pattern generator}
\author{Zahra Aminzare%
  \thanks{The Program in Applied and Computational Mathematics and  Department of Mechanical and Aerospace Engineering,  Princeton University, NJ, USA.
  Email: aminzare@math.princeton.edu (Zahra Aminzare).}
  \and Vaibhav Srivastava%
  \thanks{Department of Electrical and Computer Engineering Michigan State University, Lansing, MI , USA.
Email: vaibhav@egr.msu.edu (Vaibhav Srivastava).}
 \and Philip Holmes%
  \thanks{The Program in Applied and Computational Mathematics,  Department of Mechanical and Aerospace Engineering,  and Princeton Neuroscience Institute,  Princeton University, NJ, USA.
 Email: pholmes@math.princeton.edu (Philip Holmes).}
 }
\date{}                                           

\begin{document}
\maketitle

%
%
%
%
%

\begin{abstract}
Legged locomotion involves various gaits. It has been observed that fast
running insects (cockroaches) employ a tripod gait with three legs lifted
off the ground simultaneously in swing, while slow walking insects (stick
insects) use a tetrapod gait with two legs lifted off the ground simultaneously. 
Fruit flies use both gaits and exhibit a transition from tetrapod to tripod
at intermediate speeds. Here we study the effect of stepping frequency
on gait transition in an ion-channel bursting neuron model in which each
cell represents a hemi-segmental thoracic circuit of the central pattern generator. 
Employing phase reduction, we collapse the network of bursting neurons
represented by 24 ordinary differential equations to 6 coupled nonlinear
phase oscillators, each corresponding to a sub-network of neurons controlling one leg.
Assuming that the left  and right legs maintain constant phase differences
(contralateral symmetry), we reduce from 6 equations to 3, allowing
analysis of a  dynamical system with 2 phase differences defined on a torus. 
We show that  bifurcations occur from multiple stable tetrapod gaits to a
unique stable tripod gait as speed increases. Finally, we consider gait
transitions in two sets of data fitted to freely walking fruit flies. 
\end{abstract}

{\bf Key words.} bifurcation, bursting neurons, coupling functions, insect gaits, 
phase reduction, phase response curves, stability

{\bf AMS subject classifications.} 34C16, 34C60, 37G10, 92B20, 92C20 

\section{Introduction: idealized insect gaits}
\label{s.introduction}

Legged locomotion involves alternating stance and swing phases in which
legs respectively provide thrust to move the body and are then raised
and repositioned for the next stance phase. Insects, having six legs,
are capable of complex walking gaits in which various combinations
of legs can be simultaneously in stance and swing. However, when
walking on level ground, their locomotive behavior can be characterized
by the following kinematic rules,  \cite{Wilson66, Graham85}. 

\begin{enumerate}

\item A wave of protractions (swing) runs from posterior to
anterior legs.

\item Contralateral legs of the same segment alternate approximately
in anti-phase.

In addition, in \cite{Wilson66}, Wilson assumed that: 

\item Swing duration remains approximately constant as speed
increases. 

\item Stance (retraction) duration decreases as speed increases.

\end{enumerate}
Rules 3 and 4  have been documented in fruit flies by Mendes et. al. 
\cite{Mendes_eLife_2013}. 

In the slow metachronal gait, the hind, middle and front legs on one
side swing in succession followed by those on the other side; at most
one leg is in swing at any time. As speed increases, in view of rules
3 and 4, the swing phases of contralateral pairs of legs begin to overlap,
so that two legs swing while four legs are in stance in a tetrapod gait,
as observed for fruit flies in  \cite{Mendes_eLife_2013}. 
At the highest speeds the hind and front legs on one side swing together 
with the contralateral middle leg while their contralateral partners provide
support in an alternating tripod gait which is typical for insects at high speeds. 

Motivated by observations and data from fruit flies, which use both tetrapod and tripod gaits, 
\ZA{and from cockroaches, which use tripod gaits \cite{Fuchs14}, and stick insects, which use tetrapod gaits \cite{AYALI20151}},
our goal is to understand the transition between these gaits and their stability properties, 
 analytically. 
 Our dynamical analysis provides a mechanism that supplements the kinematic description given above. 
  This will allow us to distinguish tetrapod, tripod, and transition gaits precisely and ultimately to 
  obtain rigorous results characterizing their existence and stability.  
  \ZAP{For gait transitions in vertebrate animals, 
  see e.g. \cite{Golubitsky1999, Collins1993}.}

In \cite{Fuchs14}, a $6$-oscillator model, first proposed in  \cite{SIAM2}, was used to fit data from freely
running cockroaches that use tripod gaits over much of their speed range
\cite{Delcomyn_1971}. Here, in addition to the tripod gait, we consider 
tetrapod gaits and study  the transitions among them and tripod gaits. 
We derive a $6$-oscillator model from a network of 6 bursting neurons 
\ZA{with inhibitory nearest neighbor coupling.}
After showing numerically that it can produce multiple tetrapod gaits
as well as a tripod gait, we appeal to the methods of phase reduction and bifurcation theory 
to study gait transitions. 
\ZAA{Our coupling assumption is supported by studies of freely running cockroaches in \cite{Fuchs14}, 
in which various architectures were compared and inhibitory nearest neighbor coupling provided the 
best fits to data according to Akaike and Bayesian Information Criteria (AIC and BIC). The inhibitory assumption is 
motivated by  the fact that neighboring oscillators' solutions are out of phase \cite{pearson_Iles_1973}.} 

Phase reduction is also used by Yeldesbay et. al.
in \cite{YeldHolTothDaun_2016, YeldTothDaun_2017} to model stick insect locomotion and
display gait transitions. \ZAA{Their reduced model contains 3 ipsilateral legs and has 
a cyclical coupling architecture, with a connection from hind to front segments. 
Here we show that the nearest neighbor architecture also produces such 
gait transitions.}

\ZA{Our main contributions are as follows. First, we confirm that speed changes in the bursting neuron 
model can be achieved by parameter variations (cf. \cite{SIAM2, SIAM1}) and we numerically illustrate 
that increasing speed leads to transition from tetrapod to tripod gaits. 
We then reduce the bursting neuron model from 24 ODEs to 2 phase difference equations and characterize 
coupling functions that produce these gait transitions. 
We illustrate them via analysis and  simulations of the 24 ODE model and the phase difference equations, using parameters
derived from fruit fly data, thereby showing biological feasibility of the mechanisms.}
 
 This paper is organized as follows. 
In Section~\ref{Bursting neuron model}, we review the ion-channel
model for bursting neurons which was developed in \cite{SIAM2,SIAM1}, 
study the influence of the parameters on speed and demonstrate gait transitions numerically.  
%
\ZA{In Section~\ref{phase_reduction}, we
describe the derivation of reduced phase equations, 
and define tetrapod, tripod and transition gaits. 
At any fixed speed, we  assume constant
phase differences between left- and right-hand oscillators, so that an
ipsilateral network of 3 oscillators determines the dynamics of all
6 legs. 
We further reduce to a pair of
phase-difference equations defined on a 2-dimensional torus. 
In Section~\ref{existence_and_stability} we prove the existence of
tetrapod, tripod and transition gaits under specific conditions on the
intersegmental coupling strengths, and establish their stability types. 

In Section \ref{application_BN} we apply the results of 
Section~\ref{existence_and_stability} to the bursting neuron model. 
We show that the form of the coupling functions, which depend upon speed, 
imply the existence of transition solutions connecting tetrapod gaits to the tripod gait.   
%
In Section \ref{application_general_H} we characterize a class of explicit coupling 
functions that exhibit transitions from tetrapod gaits 
to the tripod gait. 
As an example, we  analyze phase-difference equations, 
using coupling functions approximated by Fourier series 
and derive bifurcation diagrams via branch-following methods.
In Section~\ref{DATA} we describe gait transitions in a phase model
with  coupling strengths estimated by fitting data from
freely running fruit flies, and show that such transitions occur even 
when coupling strengths are far from the special cases
studied in Sections~\ref{existence_and_stability} and  \ref{application_BN}. We conclude in 
Section~\ref{conclusion}. 
}

\section{Bursting neuron model }
\label{Bursting neuron model}
In this section we define the bursting neuron model, describe its behavior,  
and illustrate the gait transitions in a system of 24 ODEs representing 6 coupled bursting neurons. 
\subsection{A single neuron}

CPGs in insects are networks of neurons 
in the thoracic and other ganglia that produce rhythmic motor patterns
such as walking, swimming, and flying. 
\ZAP{CPGs for rhythmic movements are reviewed in e.g. \cite{Marder_Bucher_CPG_review, Ijspeert_CPG_review}.}
In this work, we employ a bursting neuron model which was developed
in \cite{SIAM1} to model the local neural network driving each leg.
This system includes a fast nonlinear current, e.g., $I_{Ca}$, a slower
potassium current $I_K$, an additional very slow current $I_{KS}$, and
a linear leakage current $I_L$. The following system of ordinary 
differential equations (ODEs) describes the bursting neuron model 
and its synaptic output $s(t)$. 
\begin{subequations}\label{BN}
\begin{align}
      C\dot {v}&= -\left\{I_{Ca}(v)+I_K(v,m)+I_{KS}(v,\w)+I_L(v)\right\}  + I_{ext},\label{BN1}\\
         \dot {m}&= \dis\frac{\epsilon}{\tau_m(v)} [m_{\infty} (v)-m],\label{BN2}\\
         \dot {\w}&= \dis\frac{\delta}{\tau_{\w}(v)} [\w_{\infty} (v)-\w],\label{BN3}\\
         \dot {s}&= \dis\frac{1}{\tau_{s}}[s_{\infty} (v)(1-s)-s],
\label{BN4}
\end{align}
\end{subequations}
where the ionic currents are of the following forms
\be{currents}
\bal
      &I_{Ca} (v)\;=\; \bar g_{Ca} n_{\infty}(v) (v-E_{Ca}),\qquad
        I_{K} (v,m)\;=\; \bar g_{K} \;m \;(v-E_{K}),\\
      &I_{KS} (v,\w)\;=\; \bar g_{KS} \w \;(v-E_{KS}), \qquad
        I_L(v)\;\;=\;\;\bar g_L(v-E_L).
\eal
\ee
The steady state gating variables associated with ion channels and their
time scales take the forms
\be{steady_states}
\bal 
        &m_{\infty}(v) \;=\; \frac{1}{1+e^{-2k_{K}(v-v_K)}}\;,\qquad \qquad
          \w_{\infty}(v) \;=\; \frac{1}{1+e^{-2k_{KS}(v-v_{KS})}}\;,\\
        &n_{\infty} (v)\;=\;  \frac{1}{1+e^{-2k_{Ca}(v-v_{Ca})}}\;,\qquad\qquad
          s_{\infty} (v)\;=\;  \frac{a}{1+e^{-2k_s(v-E^{pre}_{s})}}\;,
\eal
\ee 
 and 
\be{time_scales}
\bal 
       \tau_{m} (v)= \sech( k_{K}(v-v_{K})),\quad
       \tau_{\w} (v)= \sech( k_{KS}(v-v_{KS})).
\eal
\ee
Here the variable $s$ represents neurotransmitter released at the synapse 
and the constant parameter $\tau_s$ specifies the synaptic time scale.  
The constant parameters are generally fixed as specified in 
Table \ref{table_of_parameters}. Most of the parameter values are taken 
from \cite{SIAM1}, but some of our notations are different.  
\begin{table}[h!]
 \begin{tabular}{|c | c c c c c  c c c c c  c c |} 
 \hline
  & $\delta$  
  & $I_{ext}$ 
  & $\bar g_{Ca}$ 
  & $\bar g_K$ 
  & $\bar g_{KS}$ 
  & $\bar g_{L}$ 
  & $\bar g_{syn}$ 
  & $E_{Ca}$ 
  & $ E_K$ 
  &$ E_{KS}$ 
  &  $ E_{L}$ 
  & $E^{post}_s$  
 \\ 
  [0.5ex] 
 \hline\hline
 $\delta$ control & varies 
 & 35.6
 &4.4 
 & 9.0 
 & 0.19 
 & 2.0 
 & 0.01
 &120 
 & -80 
 &-80 
 & -60 
 & -70 
  \\ 
 \hline
 $I_{ext}$ control 
 & 0.027 
 &varies 
 &4.4 
 &9.0 
 &0.5 
 &2.0
 &0.01
 &120
 &-80 
 &-80
 &-60
 &-70
 \\
  [1ex] 
 \hline
 \end{tabular}
 
 \begin{tabular}{|c | c c c  c c c c c c c c c |} 
 \hline
   & $k_{Ca}$ 
   & $k_K$
   & $k_{KS}$
   & $k_s$ 
   & $v_{Ca}$ 
   & $v_K$
   & $v_{KS}$
   &  $E^{pre}_s$ 
   & a 
   & $C$ 
   & $\epsilon$ 
   & $\tau_s$
   \\ 
  [0.5ex] 
 \hline\hline
 $\delta$ control 
 & 0.056 
 & 0.1 
 & 0.8 
 &0.11  
 &-1.2
 &2
 &-27
 &2 
 &55.56
 &1.2
 & 4.9
 &5.56
  \\ 
 \hline
 $I_{ext}$ control 
 & 0.056
 & 0.1 
 &0.8 
 & 0.11
 &-1.2
 &2
 &-26
 & 2
 &444.48
 &1.2
 &5.0
 &5.56
  \\
  [1ex] 
 \hline
 \end{tabular}
\caption{{The constant parameters in the bursting neuron model,
as $\delta$ (first row) and $I_{ext}$ (second row) vary.}}
\label{table_of_parameters}
\end{table}

Figure \ref{v_m_w_s_del_Iext} (first row) shows the solution of 
Equation~(\ref{BN}) for the parameters specified in the first row of
Table~\ref{table_of_parameters}, and  for $\delta=0.02$. 
Figure~\ref{v_m_w_s_del_Iext} (second row) shows the solution of 
Equation~(\ref{BN}) for the parameters specified in the second row
of Table~\ref{table_of_parameters}, and for $I_{ext}=36.5$. 
 We solved the equation using a fourth order \ZAB{explicit} Runge-Kutta method 
 \ZAB{in a custom-written code, with fixed time step, $0.001$ ms} and ran the 
 simulation for 1000 ms with initial conditions:
\[
v(0) = -70, \;
m(0)= -10,\;
\w(0)= -4,\;
s(0) = 2.
\]

 \begin{figure}[h!]
\begin{center}
 \includegraphics[scale=.4]{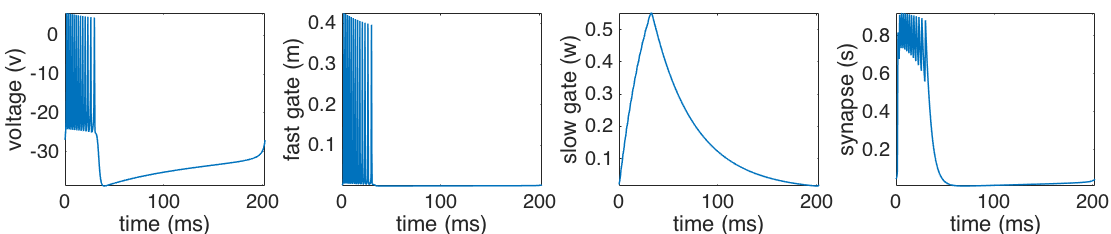}
 \includegraphics[scale=.4]{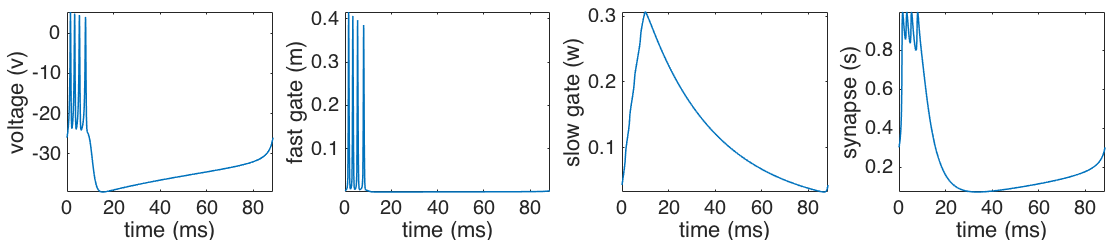}
\end{center}
\caption{First row: A solution of Equation (\ref{BN}) for the parameters
 in the first row of Table \ref{table_of_parameters}, and for $\delta=0.02$. 
 Second row: A solution of Equation (\ref{BN}) for the parameters in the
 second row of Table \ref{table_of_parameters}, and for $I_{ext}=36.5$. 
 Each case is shown for one period of the bursting process.}
\label{v_m_w_s_del_Iext}
\end{figure} 

 The periodic orbit in $(v,m,\w)$ space contains a sequence of spikes (a burst) followed by
 a quiescent phase, which correspond respectively to the swing and
 stance phases of one leg. The burst from the CPG inhibits depressor
 motoneurons, allowing the swing leg to lift from the ground
 \cite{SIAM2, pearson_Iles_1973} (see also  \cite{pearson_Iles_1970,pearson_1972}). 
We denote the period of the periodic orbit by $T$, i.e., it takes $T$
time units (ms here) for an insect to complete the cycle of each leg.  
The number of  steps completed by one leg per unit of time is the
stepping frequency and is equal to $\omega = 2\pi/T$. The period of
the limit cycles shown in Figure \ref{v_m_w_s_del_Iext} are
approximately 202 ms and  88.57 ms, and their  frequencies 
are approximately 4.95 Hz and 11.29 Hz, respectively. 
The swing phase ($SW$) is the duration of one burst and represents
the time when the leg is off the ground, and the stance phase ($ST$)
is the duration of the quiescence  in each periodic orbit and
represents the time when the leg is on the ground. Hence, $SW + ST
=T$. The swing duty cycle, denoted by $DC$, is equal to $SW/T$. 
 Note that an insect decreases its speed primarily by decreasing its
 stance phase duration (see the data in \cite{Mendes_eLife_2013},
 and the rules from \cite{Wilson66}, given in the introduction). 
  
In what follows, we show the effect of two parameters in the bursting
neuron model, $\delta$ and $I_{ext}$, on period, swing, stance and
duty cycle. We will see that these parameters have a major
effect on speed; i.e., when either $\delta$ and $I_{ext}$ increase, the
period of the periodic orbit decreases, primarily by decreasing stance
phase duration, and so the insect's speed increases. We consider
the effects of each parameter separately but in parallel. As we study
the effect of $\delta$ (resp. $I_{ext}$), we fix all other parameters
as in the first (resp. second) row of Table~\ref{table_of_parameters}.
We let $\delta$ vary in the range $\ZAA{[\delta_1, \delta_2]} = [0.0097,0.04]$
and $I_{ext}$ vary  in the range $\ZAA{[I_1, I_2]} = [35.65,37.7]$.

 \subsubsection{Effect of the slowest  time scale $\delta$ and external
 input $I_{ext}$ on stepping frequency}
 \label{range_frequency}

Figure~\ref{FDSS_delta} shows the frequency, duty cycle, stance, 
and swing as functions of $\delta$. We computed these quantities by
numerically solving the bursting neuron model (\ref{BN}) for a fixed set
of parameters (first row of Table \ref{table_of_parameters}) as $\delta$
varies. As the figure depicts, as $\delta$ increases from $0.0097$ to $0.04$,
stepping frequency increases from approximately $2.66$ Hz to $8.59$ Hz,
i.e., the speed of the animal increases. Also, note that the stance and swing
phase durations decrease, while the duty cycle remains approximately constant.  
\begin{figure}[h!]
\begin{center}
        \includegraphics[scale=.43]{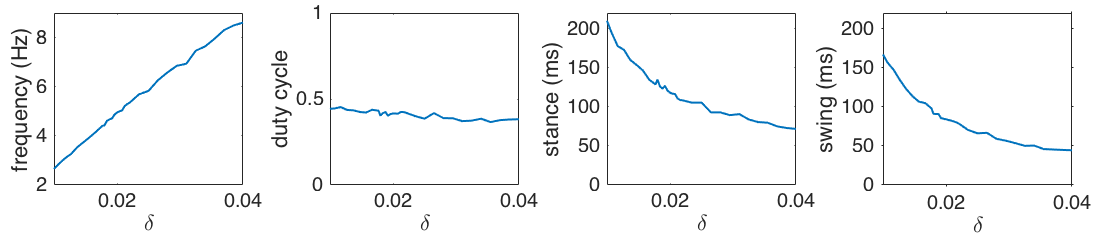}
\end{center}
\caption{The effect of $\delta$ on frequency, duty cycle, stance, and swing in 
a single uncoupled bursting neuron model. See the first row of 
Table~\ref{table_of_parameters} for parameters.}
\label{FDSS_delta}
\end{figure}

We repeat the scenario with fixed parameters in the second row of
Table~\ref{table_of_parameters}  and varying $I_{ext}$. 
Figure~\ref{FDSS_Iext} shows frequency, duty cycle, stance, and swing as
functions of  $I_{ext}$. As $I_{ext}$ increases from $35.65$ to $37.7$,
stepping frequency increases from approximately $6.9$ Hz to $14.9$ Hz.
Now, the duty cycle increases slightly, in contrast to Figure~\ref{FDSS_delta},
while the swing duration remains approximately constant. This is closer to 
the rules given in Section~\ref{s.introduction}. 
\begin{figure}[h!]
\begin{center}
          \includegraphics[scale=.43]{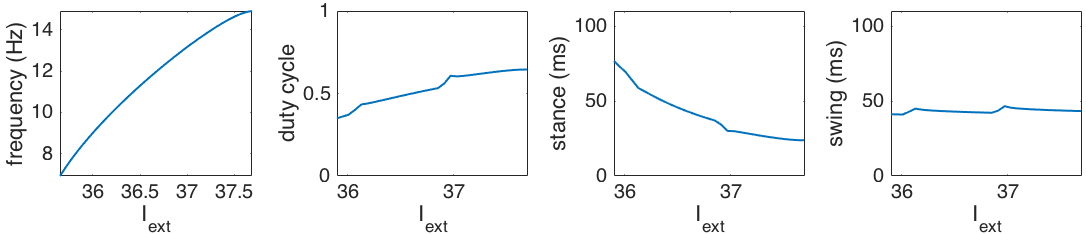}  
\end{center}
\caption{The effect of $I_{ext}$ on frequency, duty cycle, stance, and swing
in a single uncoupled bursting neuron model. See the second rows of
Table~\ref{table_of_parameters} for parameters.}
\label{FDSS_Iext}
\end{figure}

For the rest of the paper, we use the symbol $\x$ to denote the speed
parameter $\delta$ or $I_{ext}$. We note that it is more realistic
to use $I_{ext}$ as speed parameter, for the following three reasons.
\begin{enumerate}[leftmargin=*]
\item Input currents provide a more biologically relevant control mechanism, \cite{TothKnopsDaun-G-JNPhys12}. 
\item Swing duration remains approximately constant, as proposed in rule 3
of Section~\ref{s.introduction}, while $\delta$ affects burst duration
\cite{SIAM1}.
\item The frequency range obtained is closer to that seen in fruit fly
\cite{Mendes_eLife_2013} and cockroach data \cite{Fuchs14}. 
\end{enumerate}
 
\bremark
We have observed that $\bar{g}_{KS}$ has a similar influence on stepping
frequency as $\delta$ and $I_{ext}$, but we will not study the effects of this
parameter on gait transition in this paper.  
 \eremark 

\subsection{Weakly interconnected  neurons}

\begin{wrapfigure}[13]{r}{0.3\textwidth}
\centering 
         \includegraphics[width=.16\textwidth]{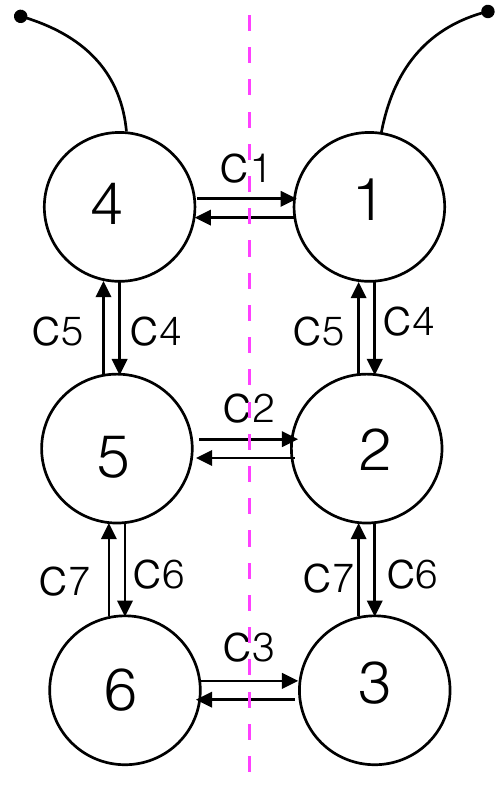}
\caption{Network of CPGs}
\label{6-oscillator}
\vspace{-110pt}
\end{wrapfigure}
We now consider a network of six mutually inhibiting units, representing
the hemi-segmental CPG networks contained in the insect's thorax, as 
shown in Figure~\ref{6-oscillator}.
We assume that inhibitory coupling is achieved via synapses that produce
negative postsynaptic currents. The synapse variable $s$ enters the 
postsynaptic cell in Equation~(\ref{BN1}) as an additional term, $I_{syn}$, 
\be{synapse_v}
\bal 
          C\dot {v}_i&= -\left\{I_{Ca}+I_K+I_{KS}+I_L\right\}  + I_{ext} +I_{syn}\;,
\eal
\ee 
where
\be{I_syn}
\bal 
          I_{syn}=\dis\sum_{j\in \mathcal{N}_i} I_{syn}(v_i,s_j) = \dis\sum_{j\in \mathcal{N}_i} - \bar\c_{ji} \bar g_{syn} s_j \lt(v_i- E_s^{post}\rt), 
 \eal
\ee  
$\bar g_{syn}$ denotes the synaptic strength, and $\mathcal{N}_i$
 denotes the set of the nodes adjacent to node $i$.   
The multiplicative factor $\bar\c_{ji}$ accounts for the fact that multiple
bursting neurons are interconnected in the real animals,  and 
$- \bar\c_{ji} \bar g_{syn}$ represents an overall coupling strength between 
hemi-segments. Following \cite{Fuchs14} we assume contralateral symmetry
and include only nearest neighbor coupling, so that there are three
contralateral coupling strengths $\c_1, \c_2, \c_3$ and four ipsilateral
coupling strengths $\c_4, \c_5, \c_6,$ and $\c_7$; \ZAB{see Figure \ref{6-oscillator}}. 
For example, $\bar\c_{21} = \c_5$, $\bar\c_{41} = \c_1,$ etc. Furthermore, 
we assume that all connections are inhibitory, i.e., $- \c_{i} \bar g_{syn}<0$, 
therefore, all the $\c_{i}$'s are positive.
 
A system of $24$ equations describes the dynamics of the $6$ coupled
cells in the network as shown in Figure~\ref{6-oscillator}. We assume that
each cell which is governed by Equation (\ref{BN}), represents one leg of
the insect. Cells $1,2$, and $3$ represent right front, middle, and hind legs,
and cells $4, 5$, and $6$ represent left front, middle, and hind legs, respectively. 
 For example, assuming that each cell is described by $(v_i,m_i,\w_i,s_i)^T$,
 $i=1,\ldots, 6$, the synapses from presynaptic cells $2$  and $4$, denoted
 by $s_2$ and $s_4$, respectively, enter the postsynaptic cell $1$. The
 following system of $4$ ODEs describe  the dynamics of cell $1$ when 
 connected to cells $2$ and $4$:
\be{cell1}
\bal 
         C\dot {v}_1&=  -\left\{I_{Ca}(v_1)+I_K(v_1,m_1)+I_{KS}(v_1,\w_1)+I_L(v_1)\right\} + I_{ext}\\
                           &\quad - \c_1 \bar g_{syn} s_4 (v_1-E_s^{post}) - \c_5 \bar g_{syn} s_2 (v_1-E_s^{post}),\\
           \dot {m}_1&= \dis\frac{\epsilon}{\tau_m(v_1)} [m_{\infty} (v_1)-m_1],\\
          \dot {\w}_1&= \dis\frac{\delta}{\tau_{\w}(v_1)} [\w_{\infty} (v_1)-\w_1],\\
            \dot {s}_1&= \dis\frac{1}{\tau_{s}}[s_{\infty} (v_1)(1-s_1)-s_1],
 \eal
\ee  
where $\c_1$ and $\c_5$ are the coupling strengths from  cell $4$ and cell
$2$ to cell $1$, respectively. Note that we assume contralateral symmetry,
so the coupling strength from cell $1$ to cell $4$ is  equal to the coupling
strength from cell $4$ to cell $1$, etc.  Five sets, each of analogous ODEs,
describe the dynamics of the other five legs. Moreover, unlike the front and
hind legs, the middle leg cells are connected to three neighbors; see 
Figure~\ref{6-oscillator}. Thus, the full model is described by 24 ODEs.

\subsection{Tetrapod and tripod gaits}
\label{Tetra+Tri_gaits}

In this section, we show numerically the gait transition from tetrapod to
tripod as the speed parameter $\x$ increases.
An insect is said to move in a tetrapod gait if at each step two legs
swing in synchrony while the remaining four are in stance. The following
four patterns are possible. 
\ZA{
 \begin{enumerate}[leftmargin=*]
\item Forward right tetrapod: $ (R2, L3), (R1, L2), (R3, L1)$.
\item Forward left tetrapod: $(R2, L1), (R1, L3), (R3, L2)$.
\item Backward right tetrapod: $(R2, L3), (R3, L1), (R1, L2)$.
\item Backward left tetrapod: $(R2, L1), (R3, L2), (R1, L3)$.
 \end{enumerate}}
Here $R1, R2, R3$ denote the right front, middle and hind legs, and 
$L1, L2, L3$ denote the left front, middle and hind legs, respectively.
The legs in each pair swing simultaneously, and touchdown of the legs
in each pair coincides with lift off of the next pair. For example, in
$(R1,L3)$, the right front leg and left hind legs are in synchrony, etc. 
Figure~\ref{Cartoon_tetrapod_tripod} (left) shows  cartoons of an insect executing
one cycle of the forward and backward tetrapod gaits, in which each
leg completes one swing and one stance phase.

\begin{figure}[h!]
\begin{center}
\includegraphics[scale=.25]{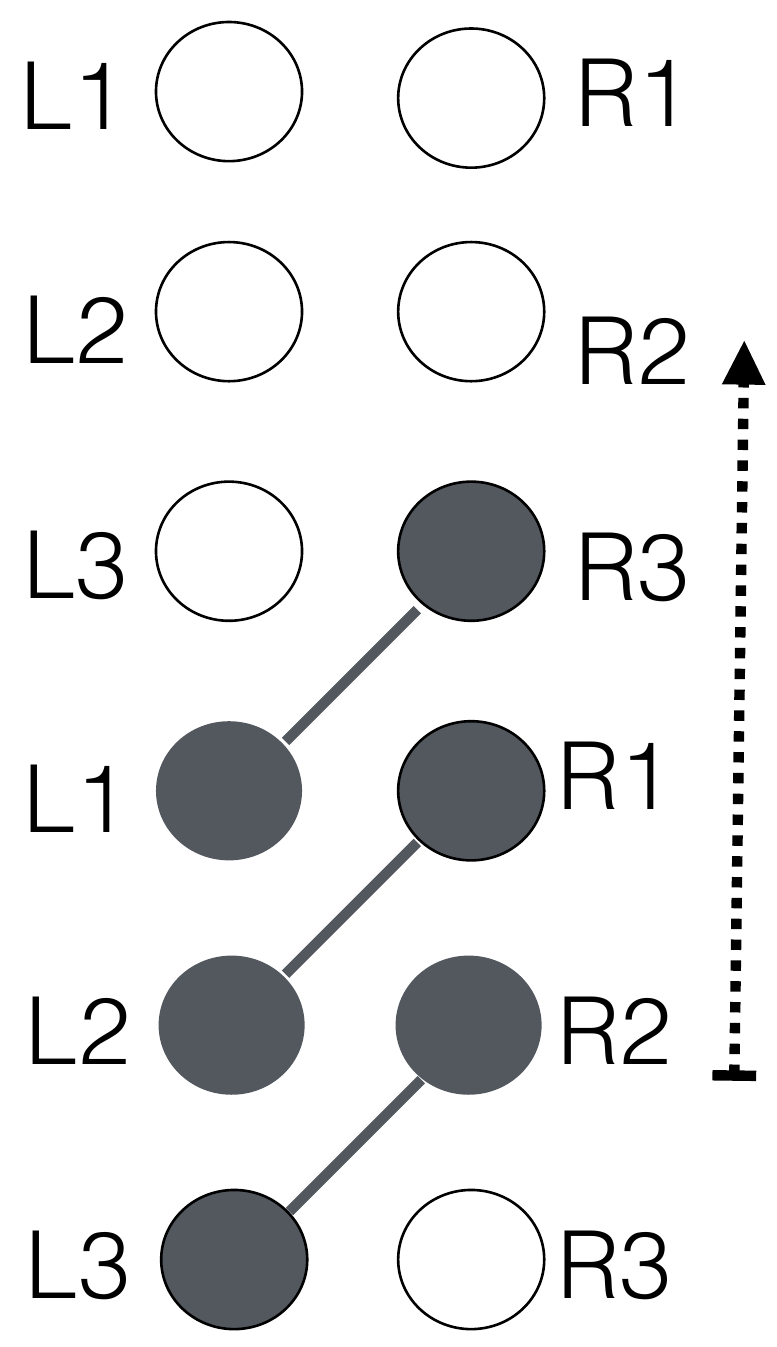}
\qquad
\includegraphics[scale=.25]{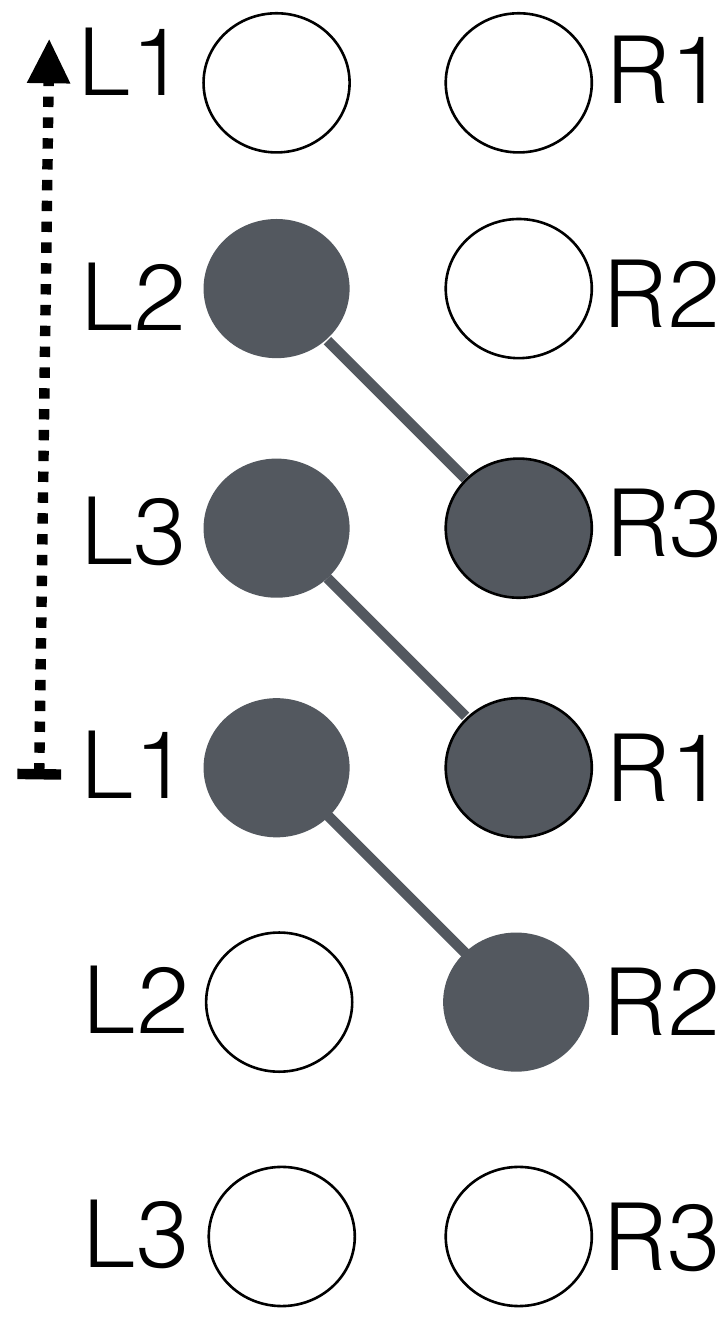}
\qquad
\includegraphics[scale=.25]{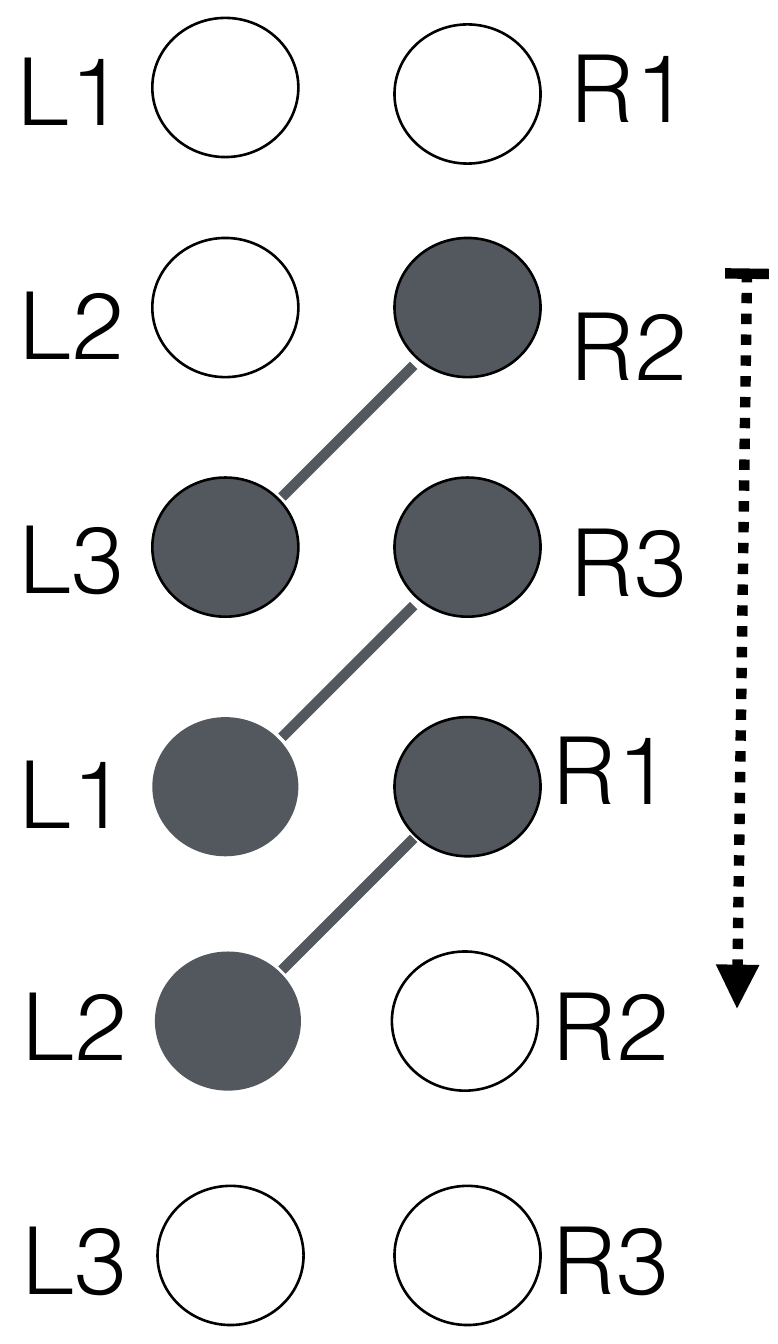}
\qquad
\includegraphics[scale=.25]{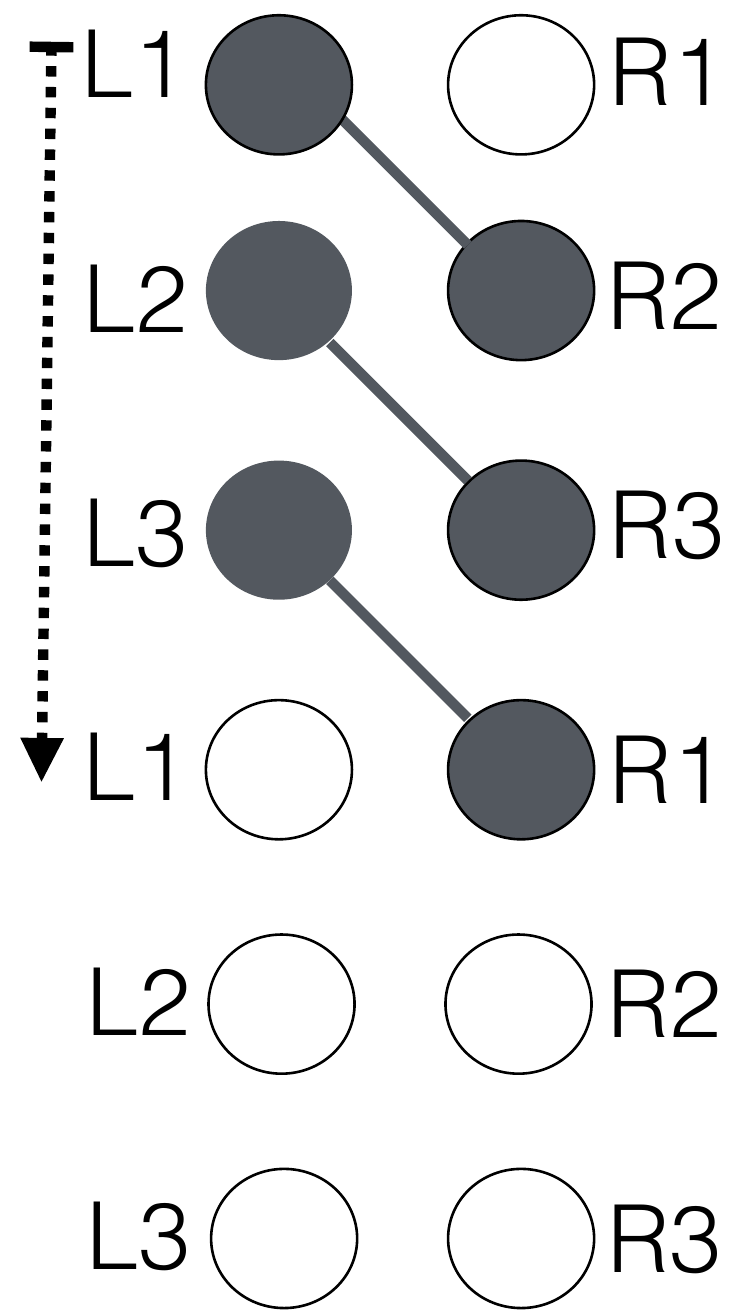}
\qquad\qquad\qquad
\includegraphics[scale=.25]{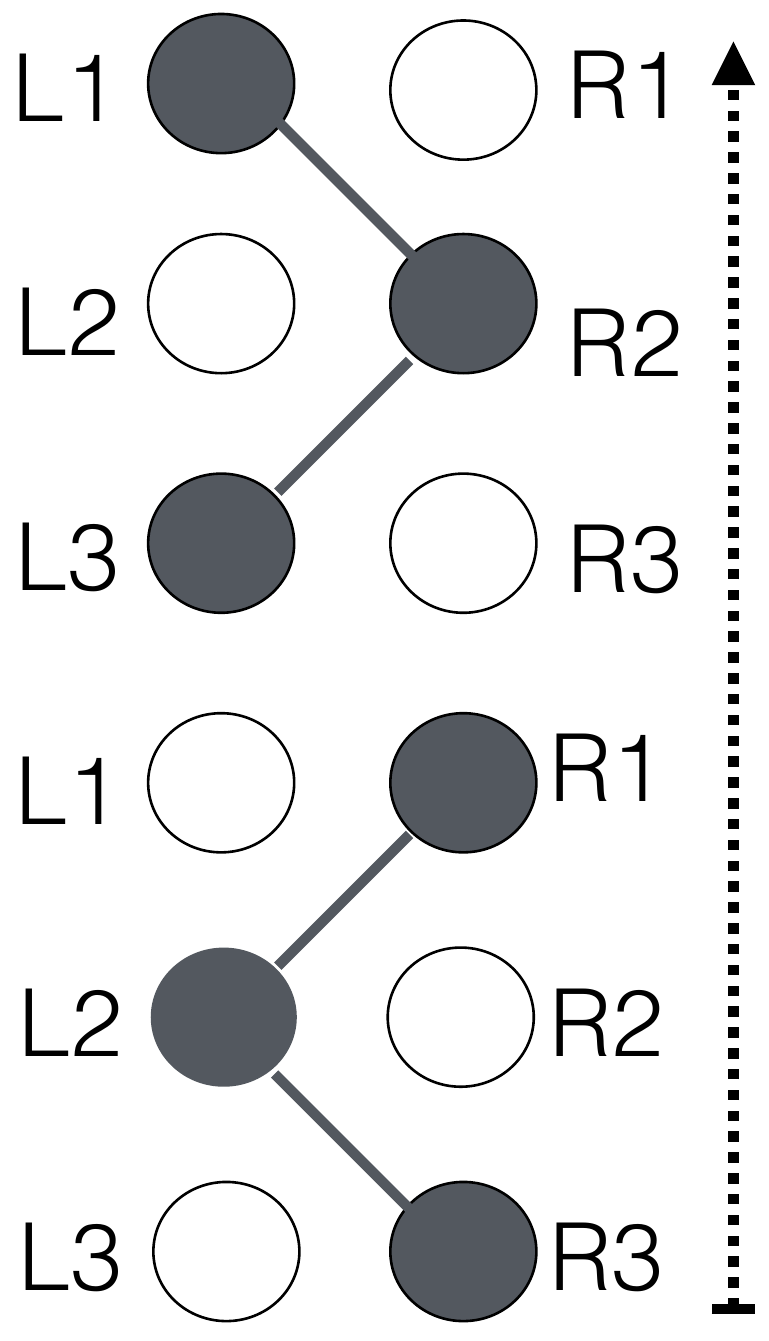}
\caption{ \ZA{(Left to right)  one cycle of  forward right,  forward left, backward
right, backward left tetrapod gaits, and a tripod gait are shown. The diagonal lines connect legs that
swing together; arrows indicate forward (resp. backward) waves.}}
\label{Cartoon_tetrapod_tripod}
\end{center}
\end{figure}


In forward gaits, a forward wave of swing phases from hind  to front legs
causes a movement, while in backward gaits, the swing phases pass
from front  to hind legs. In right gaits, \ZA{the right  legs lead  while in
left gaits the left  legs lead}. We will exhibit a gait transition from
forward right tetrapod to tripod as $I_{ext}$ varies, and a gait transition 
from forward left tetrapod to tripod  as $\delta$ varies. \ZA{Backward
gaits have not been observed in forward walking; however, see 
Figure~\ref{backward_gait} and the 
corresponding discussion in the text.}

An insect is said to move in a tripod gait (also called alternating tripod),
when the following triplets of legs swing simultaneously, and touchdown
of each triplet coincides with lift off of the other.  
\[(R1, L2, R3), (L1, R2, L3 ).\]
Figure~\ref{Cartoon_tetrapod_tripod} (right) shows a cartoon of an insect executing one
cycle of the tripod gait, in which each leg completes one swing and one
stance phase. 

Figure~\ref{transition_gait_delta_Iext} depicts a gait transition from a
forward \ZA{right} tetrapod  to a tripod  in the bursting neuron model as 
$I_{ext}$ increases (first column) and from a forward \ZA{left} tetrapod   to a 
tripod as $\delta$ increases (second column), and for a fixed set of parameters, 
initial conditions, and coupling strengths $\c_i$ as given below.
Figure~\ref{transition_v_delta_Iext}
shows the corresponding voltages. Coupling strengths $\c_i$ are fixed
at the following values for the simulations shown in 
Figures~\ref{transition_gait_delta_Iext} and \ref{transition_v_delta_Iext}:
 \be{coupling:numerical:tetra:tri}
  \c_ 1= 
   \c_ 2= 
  \c_3 = 
  \c_4 = 1,\; 
  \c_5 =
  \c_6 = 3,\; 
  \c_ 7= 2. 
  \ee
%
\begin{figure}[h!]
\begin{center}
\includegraphics[scale=.25]{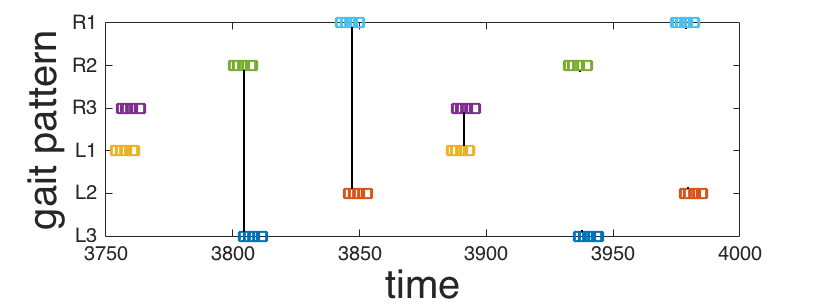} \quad \includegraphics[scale=.5]{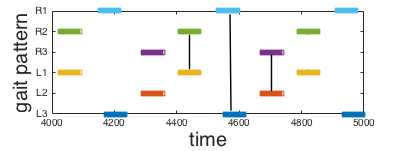}

\includegraphics[scale=.25]{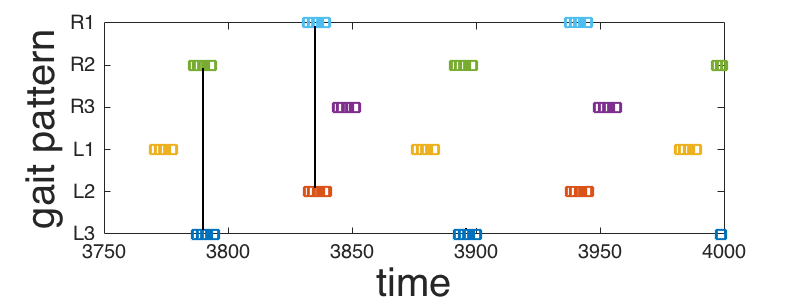} \quad \includegraphics[scale=.5]{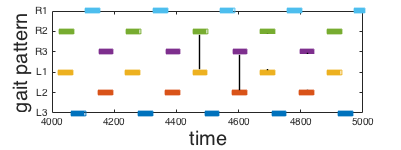}

\includegraphics[scale=.25]{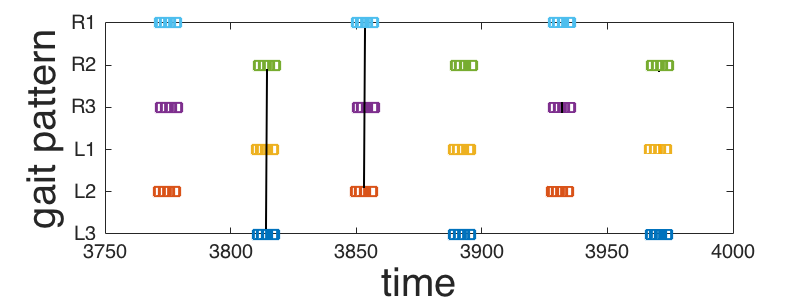} \quad \includegraphics[scale=.5]{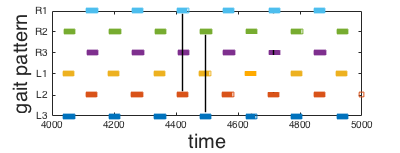}
\caption{Interconnected bursting neuron model: gait transitions from
forward {right} tetrapod  to tripod  as $I_{ext}$  increases, $I_{ext} = 35.9,
36.2, 37.0$ (left column, top to bottom), and
from forward {left} tetrapod  to tripod  as $\delta$ increases, $\delta =0.01, 0.019,
0.03$ (right column, top to bottom).
  Width of horizontal bars indicate  swing durations. 
Note the transitional gaits with
partial overlap of swing durations in the middle row. }
\label{transition_gait_delta_Iext}
\end{center}
\end{figure}
%
%
\begin{figure}[h!]
\begin{center}
\includegraphics[scale=.25]{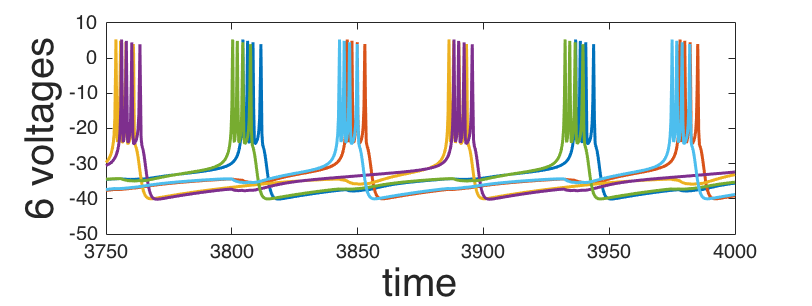} \quad \includegraphics[scale=.5]{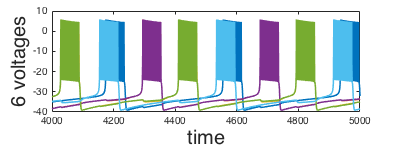}

\includegraphics[scale=.25]{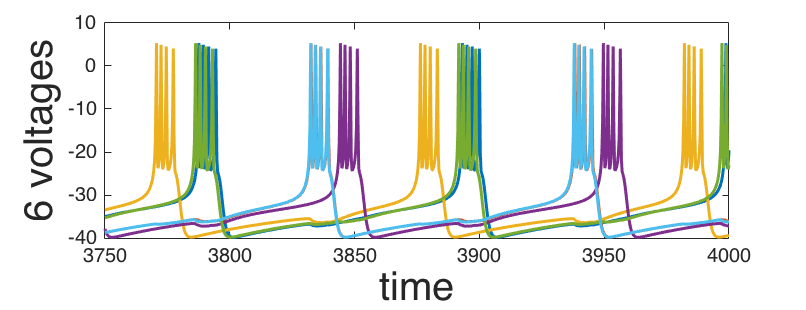}  \quad \includegraphics[scale=.5]{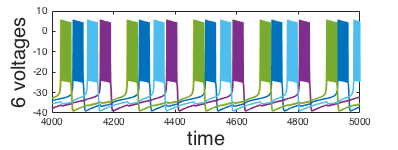}

\includegraphics[scale=.25]{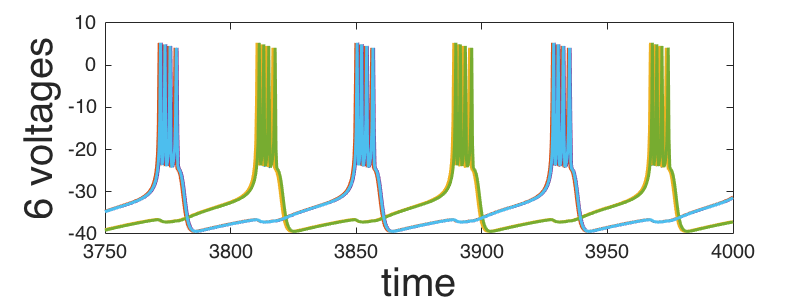} \quad \includegraphics[scale=.5]{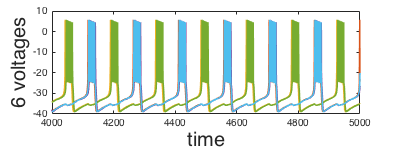}
\caption{The corresponding voltages for the gaits presented in 
Figure~\ref{transition_gait_delta_Iext}. \ZAB{Color code matches that
for legs in Figure~\ref{transition_gait_delta_Iext}. Note that some traces are 
hidden due to in-phase bursts.}}
\label{transition_v_delta_Iext}
\end{center}
\end{figure}
In the simulations shown in first column of  
Figure~\ref{transition_gait_delta_Iext} (as $I_{ext}$ varies), the 24 initial 
conditions for the 24 ODEs  are equal to 
 \be{IC:numerical:tetra:tri:v:Iext}
v_1(0)=-31.93, \;
v_2(0)=-38.55,  \;
v_3(0)=-23.83, \;
v_4(0)=-24.12, \; 
v_5(0)=-31.93, \;
v_6(0)=-38.55,
  \ee
  and for $i=1,\cdots,6$, $m_i$, $\w_i$, and $s_i$ take their steady
  state values as in Equation (\ref{IC:numerical:tetra:tri:m,w,s:delta}). 
  We computed the solutions up to time $t=4000$ ms but only
  show the  time window $[3750,  4000],$ after transients have died out. 
 In the simulations shown in second column of  Figure~\ref{transition_gait_delta_Iext}
(as $\delta$ varies), the 24 initial conditions for the 24 ODEs  are equal to 
 \be{IC:numerical:tetra:tri:v:delta}
v_1(0)= -10, \;
v_2(0)=-40,  \;
v_3(0)=-30, \;
v_4(0)=-40, \; 
v_5(0)=5, \;
v_6(0)=20,
  \ee
  and for $i=1,\cdots,6$, $m_i$, $\w_i$, and $s_i$ take their steady state values:
    \be{IC:numerical:tetra:tri:m,w,s:delta}
    m_i(0)= m_{\infty} (v_i(0)), \; 
     \w_i(0) = \w_{\infty} (v_i(0)), \; 
      s_i(0)= \frac{s_{\infty} (v_i(0))}{s_{\infty} (v_i(0))+1}.  
    \ee
We computed solutions up to time $t=5000$ ms but only show the 
time window $[4000,  5000],$ after transients have died out. 

Our goal is to show that, for the fixed set of parameters in
Table~\ref{table_of_parameters}, and appropriate coupling strengths
$\c_i$, as the speed parameter $\x$, $I_{ext}$ or $\delta$, increases,
a gait transition  from (forward) tetrapod  to tripod gait occurs. We will
provide appropriate  conditions on the $\c_i$'s  in
Section~\ref{existence_and_stability}.  
To reach our goal we first need to define the tetrapod and tripod gaits
mathematically. To this end, in the following section, we reduce the
interconnected  bursting neuron model to $6$ interconnected phase
oscillators, each describing one leg's cyclical movement.  

\section{A phase oscillator model}
\label{phase_reduction}

In this section, we apply  the theory of weakly coupled oscillators (see Section~\ref{Appendix})
 to the coupled bursting neuron models to reduce the 24 
ODEs to 6 phase oscillator equations. 
\ZAP{For a comprehensive review of oscillatory dynamics in neuroscience with many 
references see \cite{Ashwin-JMatNSci16}.}

\subsection{Phase equations for a pair of weakly coupled oscillators}
\label{pair_coupled}
Let the ODE
\be{single:neuron}
\dot X = f(X), \qquad X\in\r^n,
\ee
describe the dynamics of a single neuron. In our model, $X=(v,m,\w,s)^T$
and $f(X)$ is as the right hand side of Equations~(\ref{BN}).  Assume that
Equation~(\ref{single:neuron}) has an attracting hyperbolic limit cycle
{$\Gamma = \Gamma(t)$}, with period $T$ and frequency $\o=2\pi/T$. 

Now consider the system of weakly coupled identical neurons 
\be{coupled:neurons}
\bal
&\dot X_1= f(X_1) + \epsilon g(X_1,X_2),\\
&\dot X_2= f(X_2) + \epsilon g(X_2,X_1),
\eal
\ee
where $0<\epsilon\ll1$ is the coupling strength and $g$ is the coupling function.
 The phase of a neuron, denoted by $\phi$, is the time that has elapsed as its state
 moves around $\Gamma$, starting from an arbitrary reference point
 in the cycle.
  For each neuron,  the phase equation is: 
\be{take:average:3}
\dis\frac{d\phi_i}{dt}(t) =\omega + \epsilon H(\phi_j(t) -\phi_i(t)),
\ee
where 
\[H = H(\theta) =  \frac{1}{T}\int_0^T Z (\Gamma(\tilde t)) \cdot g(\Gamma(\tilde t),\Gamma(\tilde t +\theta )) \; d\tilde t,\]
 is the coupling function: the convolution  of the synaptic current  input to the neuron via coupling $g$ and the neuron's infinitesimal phase response curve (iPRC), $Z$. 
 For more details see the Appendix (Section~\ref{Appendix}).  

In the interconnected bursting neuron model, the coupling function $g$ 
is defined as follows. 
\be{g:interconnected:bursting:neuron}
g (x_i,x_j)= \lt(-  \bar g_{syn} s_j \lt(v_i-E_s^{post}\rt) , 0, 0, 0 \rt)^T, 
\ee
where $x_i=(v_i, m_i, \w_i, s_i)^T$ represents a single neuron (cf. 
Equations~(\ref{synapse_v})-(\ref{cell1})). Therefore, $Z\cdot g = -Z_v
\bar g_{syn} s_j \lt(v_i-E_s^{post}\rt)$,  where $Z_v$ is the iPRC in
the direction of voltage (Figures~\ref{H_PRC_versus_delta} and 
\ref{H_PRC_versus_Iext} (first rows)), and the coupling function, denoted by  $H_{BN}$, 
takes the following form:
\be{H:interconnected:bursting:neuron}
H_{BN}(\theta )=   -\frac{\bar g_{syn}}{T}\int_0^T Z_v (\Gamma(\tilde t))  \lt(v_i(\Gamma(\tilde t))  -E_s^{post}\rt) s_j\lt(\Gamma(\tilde t +\theta )\rt)  \; d\tilde t. 
\ee
In Figures~\ref{H_PRC_versus_delta} and~\ref{H_PRC_versus_Iext}
(second rows), we show the coupling functions $H_{BN}$ derived in 
Equation~(\ref{H:interconnected:bursting:neuron}) for two different
values of $\delta$ and $I_{ext}$, respectively. Note that $H_{BN}(\theta)<0$
over most of its range, and in particular over the interval $[T/3,2T/3]$
corresponding to tetrapod and tripod gaits. Here and for the remainder
of the paper,  coupling functions are plotted with domain $[0,1]$, although
we continue to specify the period $T$ and stepping frequency 
$\omega = 2\pi/T$ in referring to gaits in the text. 

\ZAA{Similar iPRCs to ours have been obtained for the non-spiking half center oscillator model 
used by Yeldesbay et. al. \cite{YeldTothDaun_2017},  apart from
in the region of the burst (personal communication).}
\begin{figure}[h!]
\begin{center}
 \includegraphics[scale=.17]{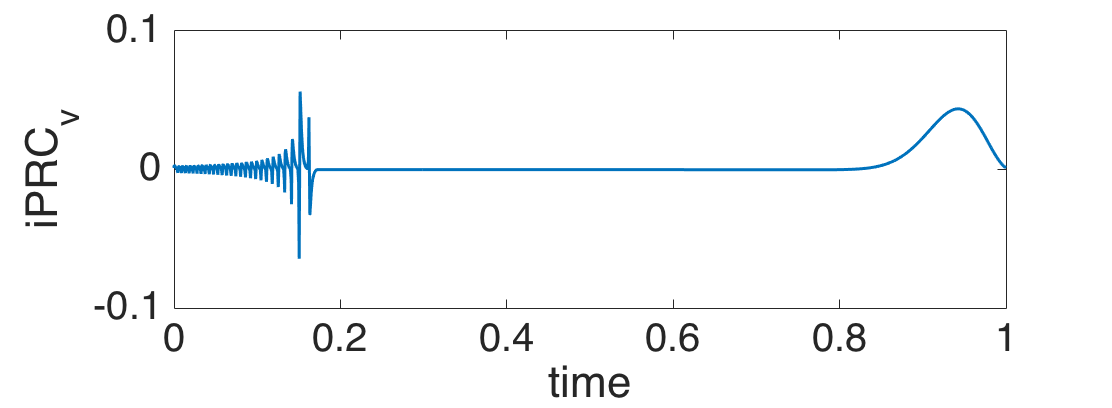}
  \includegraphics[scale=.17]{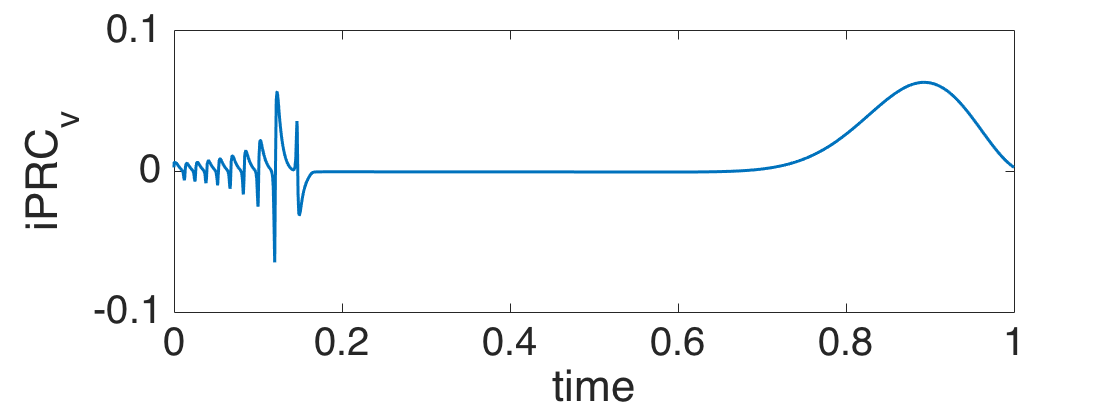}
  \includegraphics[scale=.17]{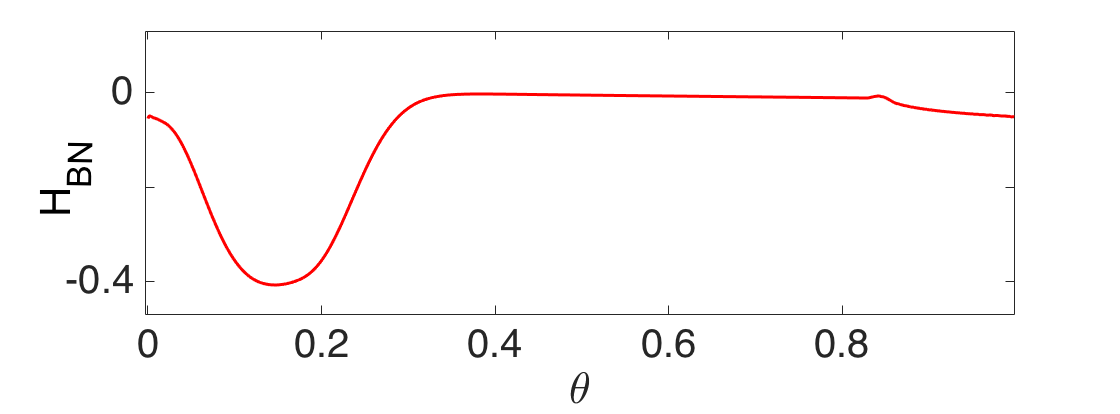}
    \includegraphics[scale=.17]{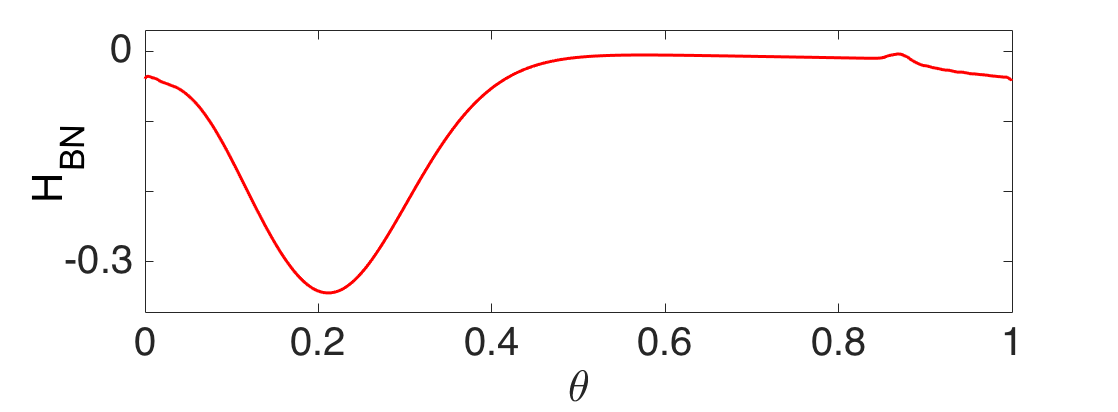}
\caption{First row:  iPRC (in the direction of $v$) for $\delta=0.0097$ (left), $\delta=0.03$ (right). 
Second row:  the coupling functions $H_{BN}(\theta)$ for  $\delta=0.0097$ (left), $\delta=0.03$ (right) }
\label{H_PRC_versus_delta}
\end{center}
\end{figure}
\begin{figure}[h!]
\begin{center}
 \includegraphics[scale=.35]{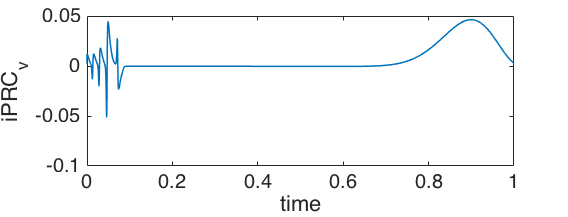}
  \includegraphics[scale=.37]{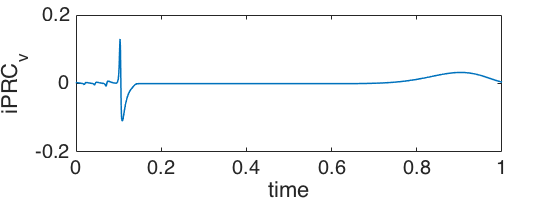}
  \includegraphics[scale=.35]{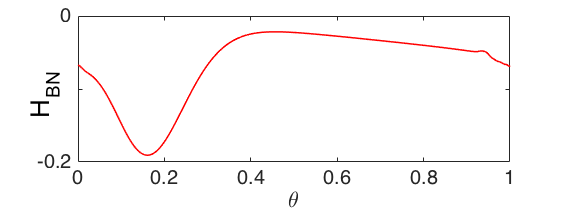}
    \includegraphics[scale=.37]{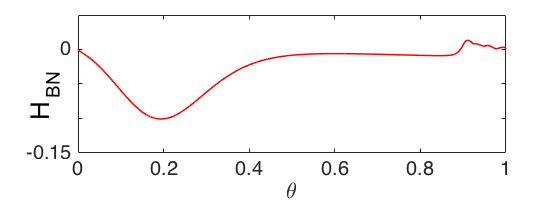}
\caption{First row:  iPRC (in the direction of $v$) for $I_{ext}=35.9$ (left), $I_{ext}=37.1$ (right). 
Second row:  the coupling functions $H_{BN}(\theta)$ for  $I_{ext}=35.9$ (left), $I_{ext}=37.1$ (right). }
\label{H_PRC_versus_Iext}
\end{center}
\end{figure}

\subsection{Phase equations for six weakly coupled neurons}
\label{6oscillators}

We now apply the techniques from Section \ref{pair_coupled} to
six coupled neurons and derive the 6-coupled phase oscillator model
via phase reduction. We assume  that all six hemi-segmental units have
the same intrinsic (uncoupled) frequency $\omega$ and that the coupling
functions $H_i$ are all identical ($H_i= H$) and T-periodic,
$T = 2\pi/\omega$. Recalling Equation~(\ref{take:average:3}) for a pair
of neurons, this leads to the following system of ODEs describing
the six legs' motions. 
\begin{equation} \label{eq.osc1}
\begin{array}{ll}
& \dot{\phi}_1 \;=\;\omega + \c_1H(\phi_4 - \phi_1) + \c_5H(\phi_2 - \phi_1), \\
& \dot{\phi}_2  \;=\; \omega + \c_2H(\phi_5 - \phi_2) + \c_4H(\phi_1 - \phi_2) + \c_7H(\phi_3 - \phi_2), \\
& \dot{\phi}_3  \;=\; \omega + \c_3H(\phi_6 - \phi_3) + \c_6H(\phi_2 - \phi_3), \\
& \dot{\phi}_4  \;=\;\omega + \c_1H(\phi_1 - \phi_4) + \c_5H(\phi_5 - \phi_4), \\
& \dot{\phi}_5  \;=\;\omega +  \c_2H(\phi_2 - \phi_5) + \c_4H(\phi_4 - \phi_5) + \c_7H(\phi_6 - \phi_5), \\
& \dot{\phi}_6  \;=\; \omega + \c_3H(\phi_3 - \phi_6) + \c_6H(\phi_5 - \phi_6) .
\end{array}
\end{equation}
Oscillators 1, 2, and 3 drive the front, middle, and hind legs on the right
with phases $\phi_1, \phi_2,$ and $\phi_3$, and oscillators 4, 5, and 6
drive the analogous  legs on the left with phases $\phi_4, \phi_5,$ and
$\phi_6$ {($\phi_i\in[0,T)$)}. Note that the derivation of the phase reduced
system in Section \ref{pair_coupled} assumes that the coupling strength
$\epsilon$ is small, implying that the product of the coefficients $\c_i$
and $H$ in Equations (\ref{eq.osc1}) should be small compared to the
uncoupled frequency $\omega$.  Since $H$ includes $\bar g_{syn}$, 
(Equation~(\ref{H:interconnected:bursting:neuron})) and $\bar g_{syn}=0.01$, 
(Table~\ref{table_of_parameters}), we have $H=\mathcal{O}(0.1)$ 
(Figures~\ref{H_PRC_versus_delta} and \ref{H_PRC_versus_Iext}). 
In the examples studied below we will take  $\c_i=\mathcal{O}(1)$. 

Next, we provide  sufficient conditions such that an insect employs a 
tetrapod gait at low speeds and a tripod gait at high speeds. We first
define idealized tetrapod and tripod gaits mathematically.    

\begin{definition}[Tetrapod and tripod gaits]  \label{defn1}
We define four versions of  tetrapod gaits as follows. Each gait corresponds
to a $T$-periodic solution of Equation~(\ref{eq.osc1}). In each version two
legs swing simultaneously in the sequences indicated in braces, and all six
oscillators share the common coupled stepping frequency $\hat\omega$. 
\begin{enumerate}[leftmargin=*]
\ZA{
\item Forward right tetrapod gait $\FR,  \{(R2, L3), (R1, L2), (R3, L1)\}$,  corresponds to
\[\FR:=\lt(\hat\omega t+\frac{2T}{3}, \;\;  \hat\omega t, \;\;   \hat\omega t+\frac{T}{3};  \;\;   \hat\omega t+\frac{T}{3}, \;\;   \hat\omega t+\frac{2T}{3}, \;\;   \hat\omega t \rt).\]

\item Forward left tetrapod gait $\FL, \{(R2, L1), (R1, L3), (R3, L2))\}$, corresponds to
\[\FL:=\lt(\hat\omega t+\frac{2T}{3}, \;\;  \hat\omega t, \;\;   \hat\omega t+\frac{T}{3};  \;\;   \hat\omega t, \;\;   \hat\omega t+\frac{T}{3}, \;\;   \hat\omega t +\frac{2T}{3} \rt).\]

\item Backward right tetrapod gait $\BR, \{(R2, L3), (R3, L1), (R1, L2)\}$, corresponds to
\[\BR:=\lt(\hat\omega t+\frac{T}{3}, \;\;  \hat\omega t, \;\;   \hat\omega t+\frac{2T}{3}; \;\;   \hat\omega t+\frac{2T}{3}, \;\;   \hat\omega t+\frac{T}{3}, \;\;   \hat\omega t\rt).\]

\item Backward left tetrapod gait $\BL, \{(R2, L1), (R3, L2), (R1, L3)\}$, corresponds to
\[\BL := \lt(\hat\omega t+\frac{T}{3}, \; \; \hat\omega t,\;\; \hat\omega t+\frac{2T}{3};\;\; \hat\omega t, \; \;\hat\omega t+\frac{2T}{3},\;\;   \hat\omega t+\frac{T}{3}\rt).\]
}
\end{enumerate}
Finally the tripod gait $\T, \{(R1, L2, R3), (R2, L1, L3)\}$,  corresponds to
\[\T:=\lt(\hat\omega t+\frac{T}{2}, \;\;  \hat\omega t, \;\;   \hat\omega t+\frac{T}{2}; \;\;   \hat\omega t, \;\;   \hat\omega t+\frac{T}{2}, \;\;   \hat\omega t \rt).\]
The frequency $\hat\omega$ will be determined later in Proposition~\ref{omega_hat}.
Depending on the sign of the coupling, $\c_iH$, $\hat\omega$ is either larger
or smaller than $\omega.$  Since we assumed  that all the connections are inhibitory, 
 $\c_i H(\phi)<0$ in the relevant range $[T/3, 2T/3]$, and therefore $\hat\omega<\omega.$
 \end{definition}

Note that in both tetrapod and tripod gaits, the phase difference between
the left and right legs in each segment is constant and is either equal to
$T/3$ or $2T/3$ (in tetrapod gaits) or $T/2$ (in tripod gaits).   

We would like to show that Equations~(\ref{eq.osc1}) admit a stable
solution at $\FR$ or $\FL$ corresponding to a  forward right or left tetrapod
gait, respectively,  when the speed parameter $\x$ (representing either
$\delta$ or $I_{ext}$) is ``small," and  a stable solution at $\T$ corresponding
to a tripod gait, when the speed parameter $\x$ is ``large."  
Since we are interested in studying the effect of the speed parameter
$\x$ on gait transition, we let the coupling function $H$ and the frequency
$\omega$ depend on $\x$ and write $H=H(\phi;\x)$ and $\omega=\omega(\x)$. 
 
\begin{definition} [Transition gaits] 
For any fixed number $\et\in[0,T/6]$,  the forward right and 
forward left transition gaits, $\FR(\et)$ and $\FL(\et)$ respectively, are as follows.  
\ZA{
\begin{subequations}\label{transition_gaits}
\begin{align}
\FR(\et)&:= \lt(\hat\omega t+\frac{2T}{3}-\et, \; \hat\omega t, \; \hat\omega t+\frac{T}{3}+\et; \;\; \hat\omega t+\frac{T}{3}-2\et, \; \hat\omega t+\frac{2T}{3}-\et, \; \hat\omega t\rt),\\
\FL(\et)& := \left(\hat\omega t+\frac{2T}{3} -\et, \; \hat\omega t, \; \hat\omega t+\frac{T}{3}+\et;\;  \;\hat\omega t, \; \hat\omega t+\frac{T}{3}+\et, \; \hat\omega t+\frac{2T}{3} +2\et\right).
\end{align}
\end{subequations}
}
\end{definition}

We call $\FR(\et)$ and $\FL(\et)$ ``transition gaits" since as $\et$ varies 
from $0$ to $T/6$, $\FR(\et)$ (resp. $\FL(\et)$) transits from the forward right
(resp. left) tetrapod gait to the tripod gait. 
For $\et=0$, $\FR(0) = \FR$  corresponds to the forward right tetrapod gait,
and $\FL(0) = \FL$  corresponds to the forward left tetrapod gait. Also for
$\et=T/6$, $\FR(T/6) = \FL(T/6) = \T$  corresponds to the tripod gait.
In addition, the phase differences between the left and  right legs
($\phi_4-\phi_1, \phi_5-\phi_2, \phi_6-\phi_3$), 
are constant and equal to $2T/3-\et$ in  $\FR(\et)$, and  $T/3+\et$ in  
$\FL(\et)$. This value is  equal to $2T/3$
(resp. $T/3$) when $\et=0$, as in the  forward right  (resp. left) tetrapod
gait, and is equal to $T/2$, when $\et=T/6$, as in the tripod gait.  
 
 \ZAA{We further assume that the phase differences between the left and right 
legs are equal to the  steady state phase differences in $\FR(\et)$ or  $\FL(\et)$ 
(later we will see that there are no differences between these two choices), 
i.e., we assume that for a
fixed $\et$, and for any $i=1,2,3$, 
\be{constant_contralateral}
 \phi_{i+3} = \phi_i + \bar\phi(\et), 
 \ee
where $\bar\phi(\et) = {2T}/{3} -\et$ or $\bar\phi(\et) = {T}/{3} + \et$.
For steady states, this assumption is supported by experiments for tripod gaits \cite{Fuchs14}, where $\bar\phi(\et)= T/2$, 
and by simulations for tripod and tetrapod gaits in the bursting neuron model, Figures~\ref{transition_gait_delta_Iext} and 
\ref{transition_v_delta_Iext}. We make a further simplifying assumption that the steady state contralateral phase differences remain constant 
for all $t$.}

Thus, assuming that the phase difference between the left and right legs
 $ \phi_{i+3} - \phi_i= \bar\phi(\et) = {2T}/{3} - \et$ or $\bar\phi(\et) = {T}/{3} + \et$, 
 and noting that since $H= H(\phi;\x)$ is $T$-periodic in its first argument, 
 $\phi_i-\phi_{i+3} = -\bar\phi(\et) = {T}/{3} +\et \; \mbox{or}\; {2T}/{3} -\et $  
(recall that $-2T/3=T/3$ mod $T$),
we can rewrite Equation~(\ref{eq.osc1}) for the forward right transition gait
$\FR(\et)$ as follows. 
\begin{subequations}\label{eq.osc3}
\begin{align}
& \dot{\phi}_1 \;=\;\omega(\x) + \c_1H\lt(\frac{2T}{3} - \et; \x\rt) + \c_5H(\phi_2 - \phi_1; \x),\label{eq.osc3.1} \\
& \dot{\phi}_2  \;=\; \omega (\x)+ \c_2H\lt(\frac{2T}{3} -\et; \x\rt) + \c_4H(\phi_1 - \phi_2; \x) + \c_7H(\phi_3 - \phi_2; \x),\label{eq.osc3.2} \\
& \dot{\phi}_3  \;=\; \omega(\x) + \c_3H\lt(\frac{2T}{3} -\et; \x\rt) + \c_6H(\phi_2 - \phi_3; \x),\label{eq.osc3.3} \\
& \dot{\phi}_4  \;=\;\omega(\x) + \c_1H\lt(\frac{T}{3} +\et; \x\rt) + \c_5H(\phi_5 - \phi_4; \x), \label{eq.osc3.4}\\
& \dot{\phi}_5  \;=\;\omega (\x)+  \c_2H\lt(\frac{T}{3} +\et; \x\rt) + \c_4H(\phi_4 - \phi_5; \x) + \c_7H(\phi_6 - \phi_5; \x),\label{eq.osc3.5} \\
& \dot{\phi}_6  \;=\; \omega (\x)+ \c_3H\lt(\frac{T}{3} + \et; \x\rt) + \c_6H(\phi_5 - \phi_6; \x)\label{eq.osc3.6} .
\end{align}
\end{subequations}

A similar equation is obtained for $\FL(\et)$ as follows. 
\ZA{
\begin{subequations}\label{eq.osc4}
\begin{align}
& \dot{\phi}_1 \;=\;\omega(\x) + \c_1H\lt(\frac{T}{3} + \et; \x\rt) + \c_5H(\phi_2 - \phi_1; \x),\label{eq.osc4.1} \\
& \dot{\phi}_2  \;=\; \omega (\x)+ \c_2H\lt(\frac{T}{3} +\et; \x\rt) + \c_4H(\phi_1 - \phi_2; \x) + \c_7H(\phi_3 - \phi_2; \x),\label{eq.osc4.2} \\
& \dot{\phi}_3  \;=\; \omega(\x) + \c_3H\lt(\frac{T}{3} +\et; \x\rt) + \c_6H(\phi_2 - \phi_3; \x),\label{eq.osc4.3} \\
& \dot{\phi}_4  \;=\;\omega(\x) + \c_1H\lt(\frac{2T}{3} -\et; \x\rt) + \c_5H(\phi_5 - \phi_4; \x), \label{eq.osc4.4}\\
& \dot{\phi}_5  \;=\;\omega (\x)+  \c_2H\lt(\frac{2T}{3} -\et; \x\rt) + \c_4H(\phi_4 - \phi_5; \x) + \c_7H(\phi_6 - \phi_5; \x),\label{eq.osc4.5} \\
& \dot{\phi}_6  \;=\; \omega (\x)+ \c_3H\lt(\frac{2T}{3} - \et; \x\rt) + \c_6H(\phi_5 - \phi_6; \x)\label{eq.osc4.6} .
\end{align}
\end{subequations}
}
\ZAA{Although we are interested in gait transitions in the bursting neuron model and 
in the phase reduction equations derived from the bursting neuron model, we prove our
results for more general $H$.}
Our goal is to provide sufficient conditions on the coupling function $H$ and 
the coupling strengths $\c_i$
that guarantee for any $\et\in[0,T/6]$, $\FR(\et)$ or $\FL(\et)$ is a stable
solution of Equations~(\ref{eq.osc3}) and (\ref{eq.osc4}). To this end, in the following section
we reduce the 6 equations (\ref{eq.osc3.1})-(\ref{eq.osc3.6}) \ZA{and  
the 6 equations (\ref{eq.osc4.1})-(\ref{eq.osc4.6})} to 2 equations
on a 2-torus. 
\ZA{The coupling strengths $\c_i$ may also depend on the speed
parameter $\x$ (see Section \ref{DATA} below). For the rest of the paper, we assume that 
$\c_i$ depends on $\x$, $\c_i=\c_i(\x)$, but for simplicity, we drop the argument $\x$.}

\subsection{Phase differences model}
\label{phase_diffs}

In this section, the goal is to reduce 
the 6 equations (\ref{eq.osc3.1})-(\ref{eq.osc3.6}) {and  
the 6 equations (\ref{eq.osc4.1})-(\ref{eq.osc4.6})} to 2 equations
on a 2-torus. 
To this end, we assume the following condition 
for the coupling function $H$. 

\begin{assumption}
\label{eta_assumption}
Assume that $H = H(\theta; \x)$ is a  differentiable function, defined on
$\r\times\ZAA{[\x_1,\x_2]}$ which is $T$-periodic on its first argument and has 
the following property. For any fixed $\x \in [\x_1,\x_2]$, 
\be{eta}
H\lt(\frac{2T}{3} - \et; \x\rt) = H\lt(\frac{T}{3}+\et; \x\rt),
\ee
has a unique solution $\eta(\x)$ such that $\eta = \eta(\x):[\x_1,\x_2] \to [0,T/6]$
is an onto and  non-decreasing function.  Note that Equation~(\ref{eta}) is also trivially satisfied by the constant 
solution $\et= T/6$. 
 \end{assumption}
\ZAA{For the rest of the paper, we assume that the coupling function $H$ satisfies Assumption \ref{eta_assumption}. 
In Proposition~\ref{H-properties}, Section \ref{application_general_H}, we characterize
a class of functions $H$, that guarantee 
solutions of Equation~(\ref{eta}). Also, we will show that the coupling functions $H_{BN}$ derived from the 
bursting neuron model satisfy  Assumption \ref{eta_assumption}, see Figures~\ref{H_PRC_versus_delta} and  \ref{H_PRC_versus_Iext}, and Section \ref{application_BN} below.}

Using Equations~(\ref{constant_contralateral}) and (\ref{eta}), 
Equations~(\ref{eq.osc3}) and (\ref{eq.osc4}) can be  reduced to the following 3 equations 
describing the right legs' motions:
\begin{subequations}\label{eq.osc5}
\begin{align}
& \dot{\phi}_1 \;=\;\omega(\x) + \c_1H\lt(\frac{2T}{3} - \et; \x\rt) + \c_5H(\phi_2 - \phi_1; \x), \\
& \dot{\phi}_2  \;=\; \omega(\x) + \c_2H\lt(\frac{2T}{3} -\et; \x\rt) + \c_4H(\phi_1 - \phi_2; \x) + \c_7H(\phi_3 - \phi_2; \x), \\
& \dot{\phi}_3  \;=\; \omega(\x) + \c_3H\lt(\frac{2T}{3} -\et; \x\rt)  + \c_6H(\phi_2 - \phi_3; \x) .
\end{align}
\end{subequations}
Because only phase differences appear in the vector field, we may define
\[
\theta_1:=\phi_1-\phi_2 \quad \mbox{and}\quad \theta_2:=\phi_3-\phi_2,
\]
so that the following equations describe the dynamics of $\theta_1$ and $\theta_2$:
\begin{subequations} \label{torus:equation}
\begin{align}
&\dot\theta_1=  (\c_1 - \c_2)H\lt(\frac{2T}{3} - \et; \x\rt) + \c_5 H(-\theta_1; \x) - \c_4 H(\theta_1; \x) - \c_7 H(\theta_2; \x), \\ 
&\dot\theta_2=  (\c_3 - \c_2)H\lt(\frac{2T}{3} - \et; \x\rt) + \c_6 H(-\theta_2; \x) - \c_4 H(\theta_1; \x) - \c_7 H(\theta_2; \x).
\end{align}
\end{subequations}
Note that  Equations~(\ref{torus:equation}) are $T$-periodic in both variables,
i.e., $(\theta_1, \theta_2)\in \mathbb{T}^2$, where $\mathbb{T}^2$ is a 2-torus. 

In Equations~(\ref{torus:equation}), the tripod gait $\T$ corresponds 
to the fixed point $(T/2, T/2)$, the forward tetrapod gaits, $\FR$ and $\FL$,
correspond to the fixed point $(2T/3, T/3)$, and the transition gaits, 
$\FR(\et)$ and $\FL(\et)$, correspond to $(2T/3-\et, T/3+\et)$. 
\ZA{Note that since $\FR(\et)$ and $\FL(\et)$ correspond to the same fixed point on the 
torus, we may assume the  contralateral phase differences to be equal to 
$\phi_{i+3} - \phi_i = {2T}/{3} - \et$ or ${T}/{3} + \et$.} 
\ZAA{See \cite{Zhang2017} for another example of conditions on coupling functions that produce 
specific phase differences.}

\subsection{Qualitative behavior of the solutions of phase difference equations}\label{Example_Qualitative_behavior}

In this section, as an example, we illustrate the nullclines and phase plane of 
Equations (\ref{torus:equation}) with 
$H=H_{BN}$ and the coupling strengths as follows. 
\be{example_only_balance}
\c_1= 1, \;
 \c_2=2.5, \;
 \c_3=1.5, \;
 \c_4= 5, \; 
\c_5=7.5, \;
 \c_6= 7,\; 
 \c_7=1.    
\ee
Here and henceforth, in all the simulations, we normalize the range of  the coupling
function $H_{BN}$ and so the torus is represented by a $1\times1$ square. 
For example  $(2T/3, T/3)$ is shown by a point at $(2/3,1/3)$, etc. To
obtain phase portraits we solved Equations~(\ref{torus:equation}) using 
the fourth order Runge-Kutta method with fixed time step $0.001$ ms and ran the 
simulation up to 100 ms with multiple initial conditions. 

Figure~\ref{NC_PP_balance_Tr_Det} (left to right) shows the nullclines
and phase planes of  Equations (\ref{torus:equation}) with 
$H=H_{BN}$ computed in Figure~\ref{H_PRC_versus_delta}(left), for a small
$\delta = 0.0097$ and Figure~\ref{H_PRC_versus_delta}(right), for a large $\delta = 0.03$, respectively. 
Intersections of the nullclines indicate the location of  fixed points. We observe that for 
small $\delta$, the fixed points $(2/3, 1/3)$ (corresponding to the forward tetrapod) 
and $(1/3,2/3)$ (corresponding to the backward tetrapod) are stable, while $(1/2,1/2)$ (corresponding to the tripod)
is unstable. For larger $\delta$, the two tetrapod gaits 
 merge to $(1/2,1/2)$, which becomes a sink.  
\begin{figure}[h!]
\begin{center}
\includegraphics[scale=.1]{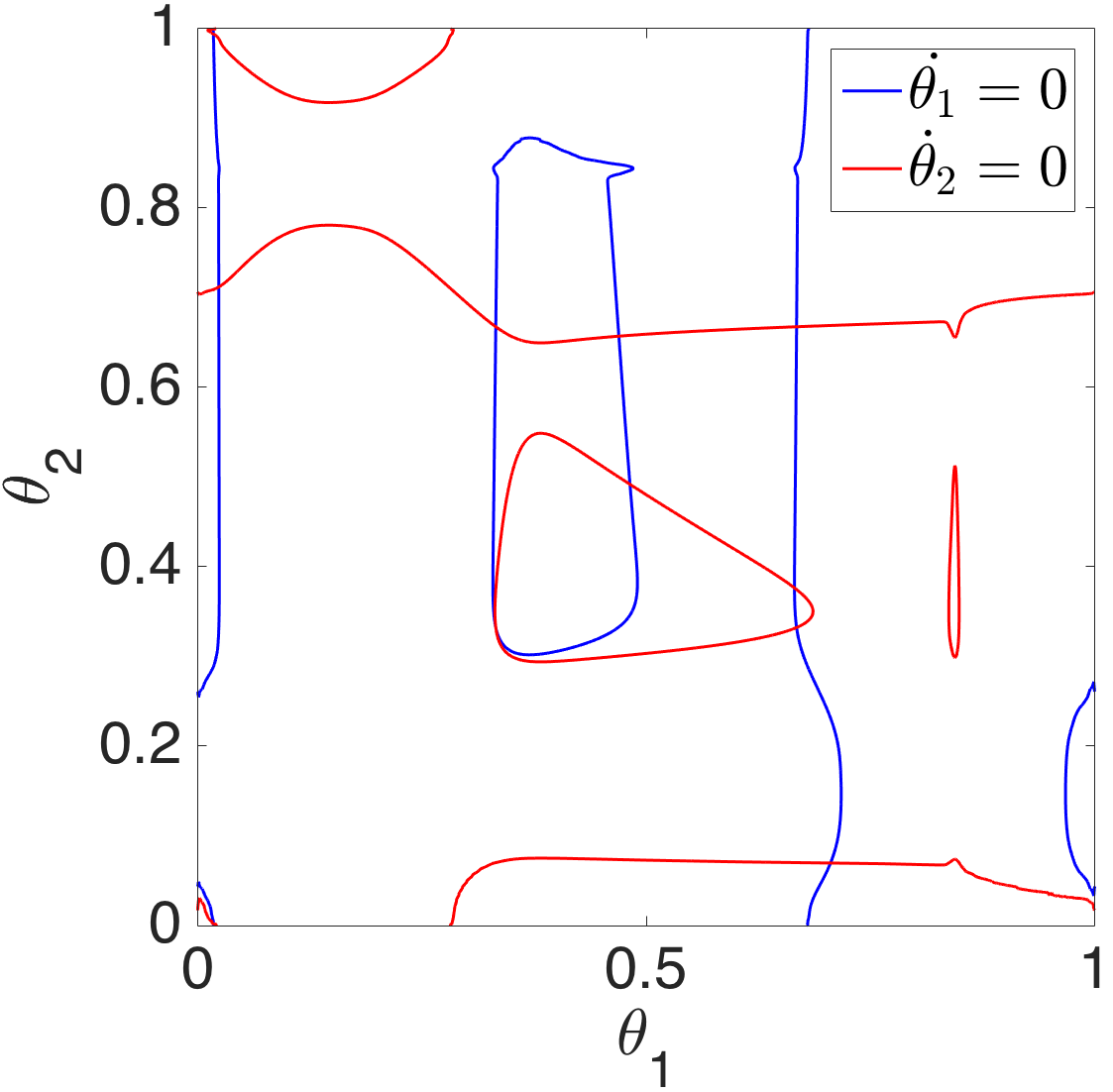}
\includegraphics[scale=.1]{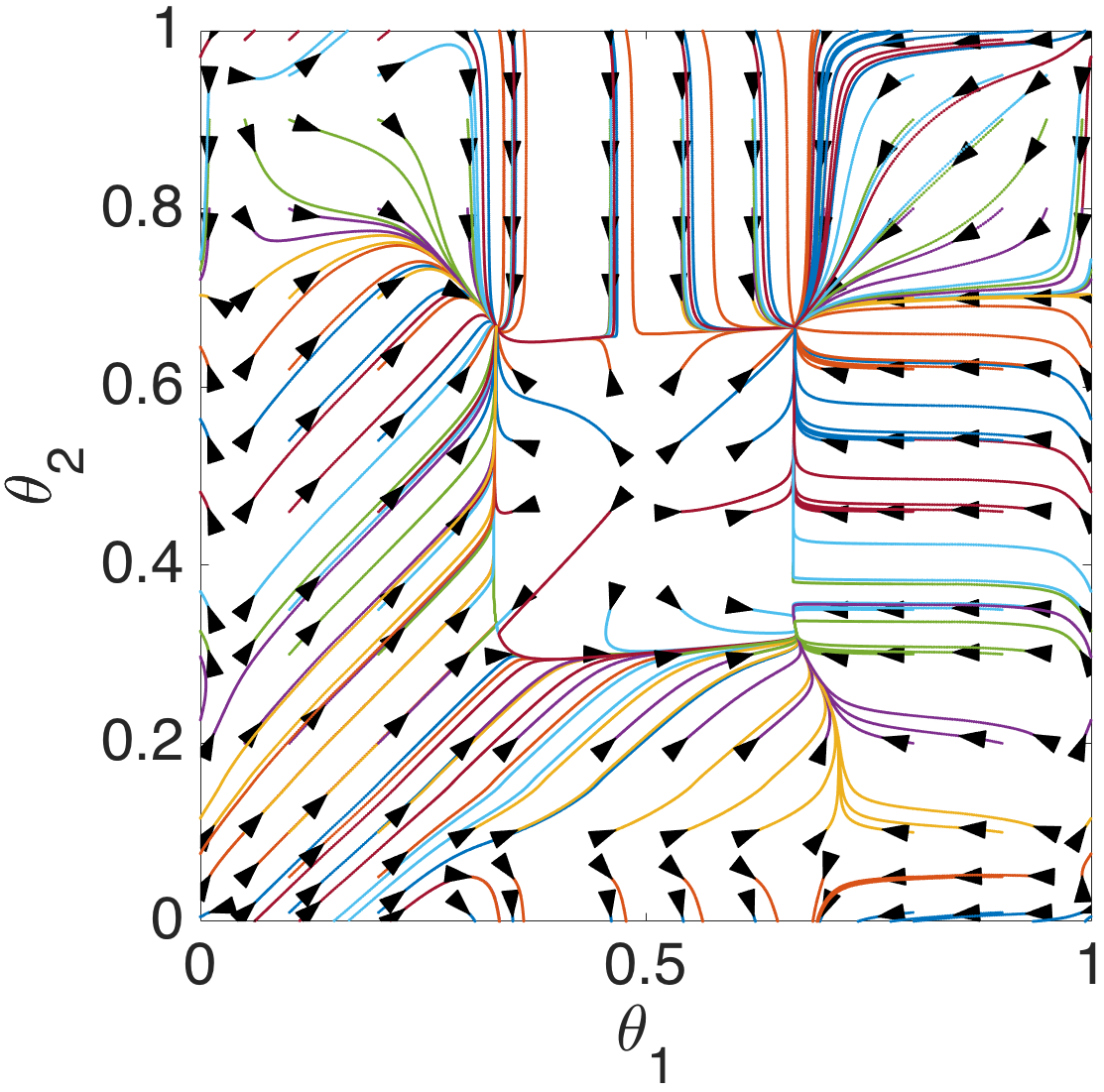}\quad
\includegraphics[scale=.1]{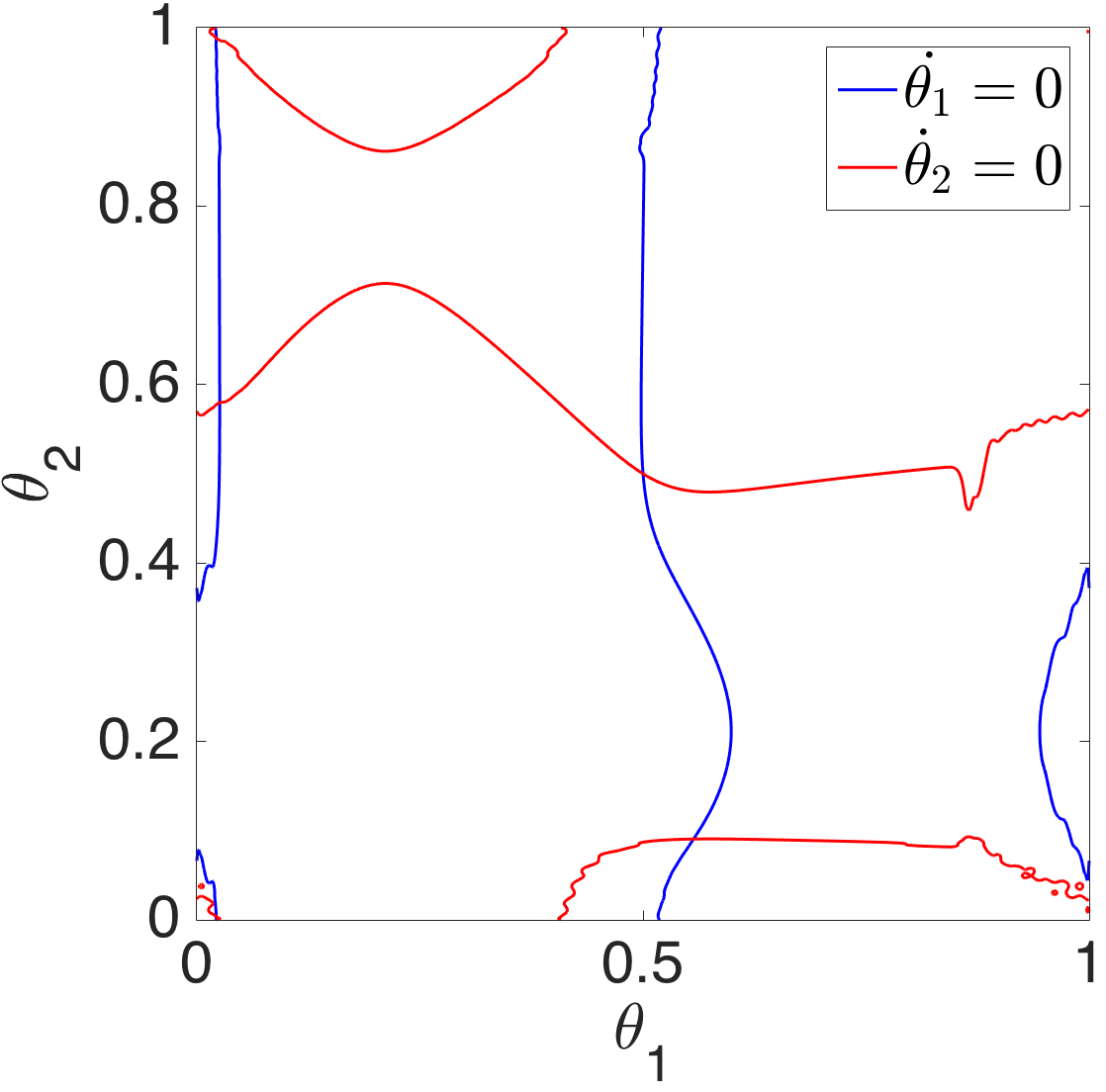}
\includegraphics[scale=.1]{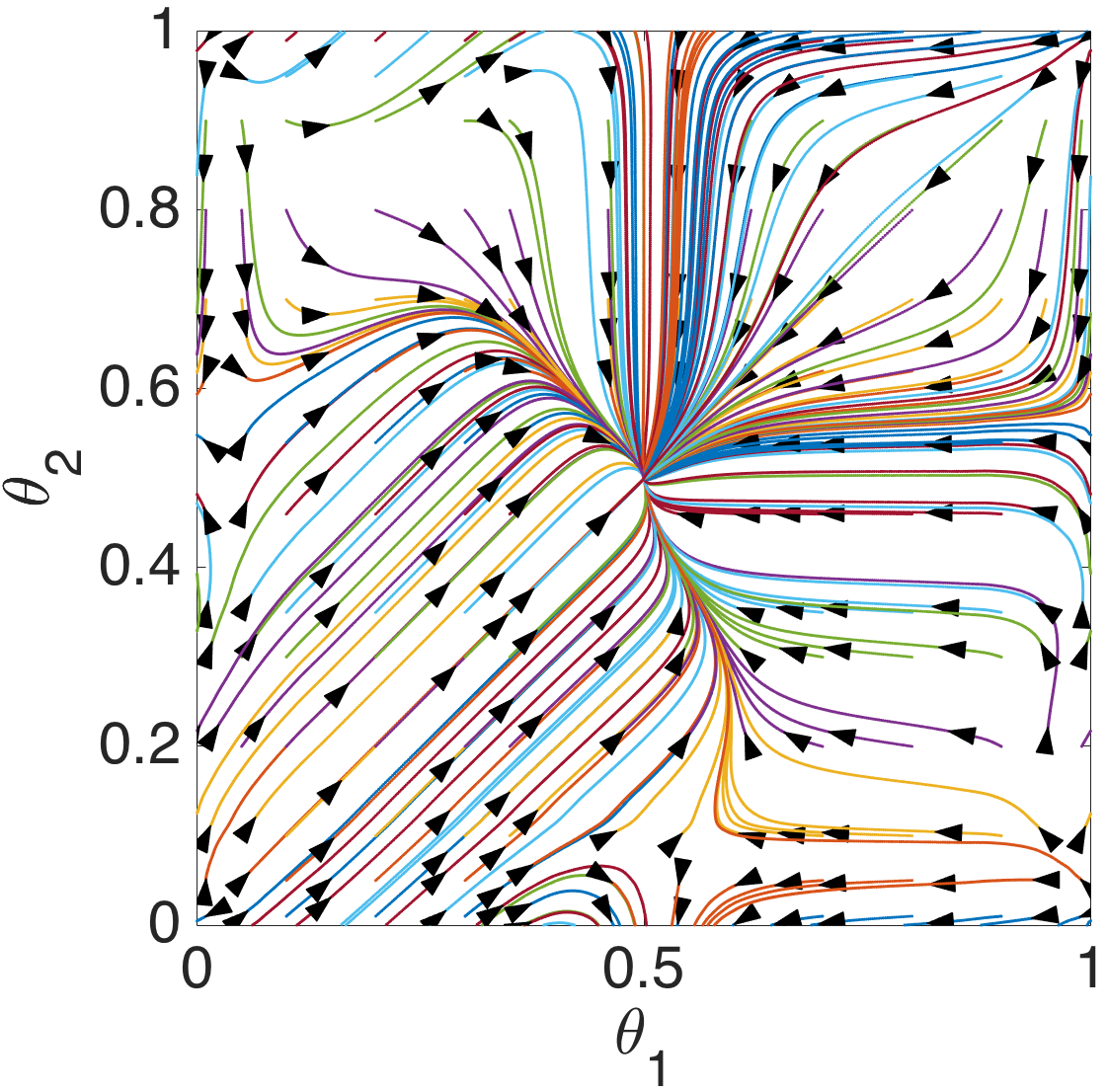}
\end{center}
\caption{(Left to right) Nullclines and  phase planes of 
Equations~(\ref{torus:equation}) when $\c_i$'s satisfy 
Equation~(\ref{example_only_balance}), and  $\delta = 0.0097$
and $0.03$, respectively. For computation of coupling functions,
all bursting neuron parameters are as in the first rows of 
Table~\ref{table_of_parameters}.}
\label{NC_PP_balance_Tr_Det}
\end{figure}

In the following sections we will address existence and stability of these fixed points and  
associated gaits and explore nonlinear phenomena involved in gait transitions.

\section{Existence and stability of tetrapod and tripod gaits}
\label{existence_and_stability}

We now prove that, under suitable conditions on the coupling
functions and coupling strengths, multiple fixed points exist for
Equations~(\ref{torus:equation}) and we derive explicit expressions
for eigenvalues of the linearized system at these fixed points.

\subsection{Existence with balance condition}
\label{existence_with_balance}

We first provide conditions on the coupling strengths $\c_i$ such that 
Equations~(\ref{torus:equation}) admit a stable fixed point at  
$(\theta_1^1, \theta_2^1):= (2T/3-\et, T/3+\et)$, for any $\et \in[0,T/6]$. 

\begin{proposition} \label{prop:tetra:stability}
If the coupling strengths $\c_i$ satisfy the following relations 
\begin{equation} \label{balance2}
\c_1+\c_5 = \c_2 + \c_4 +\c_7 = \c_3 +\c_6, 
\end{equation}
then for any $\et \in[0,T/6]$, Equations~(\ref{torus:equation}) admit 
a fixed point at $(\theta_1^1, \theta_2^1)= (2T/3-\et, T/3+\et)$. 
\ZA{Note that  $(\theta_1^1, \theta_2^1)= (2T/3-\et, T/3+\et)$ 
corresponds to  forward tetrapod ($\et=0$), forward transition ($0<\et<T/6$), 
and tripod ($\et=T/6$) gaits.}
In addition, if the following inequalities hold, 
then the fixed point is stable. 
\begin{subequations}\label{trace_det}
\begin{align}
&\mbox{Tr} := -(\c_5+\c_7) H'\lt(\frac{T}{3}+\et; \x\rt) -(\c_4+\c_6) H'\lt(\frac{2T}{3}-\et; \x\rt) <0 ,\\
&\mbox{Det} := \c_5\c_6 H'\lt(\frac{T}{3}+\et; \x\rt)H'\lt(\frac{2T}{3}-\et; \x\rt) + \c_4\c_6 \lt[H'\lt(\frac{2T}{3}-\et; \x\rt)\rt]^2 + \c_5\c_7 \lt[H'\lt(\frac{T}{3}+\et; \x\rt)\rt]^2 >0.
\end{align}
\end{subequations}
\end{proposition}
Equation~(\ref{balance2}) is called the balance equation; it expresses
the fact that the sum of the coupling strengths entering each leg are
equal. \ZAA{The equalities were assumed, without biological support, in
\cite{SIAM2}, and were subsequently found to approximately hold for
fast running cockroaches in \cite[Figure~9c]{Fuchs14}, 
according to the best data fits, judged by AIC and BIC, as reported in that paper.}

 \begin{proof}
Since by Equation~(\ref{eta}), 
$H({2T}/{3} - \et; \x) = H({T}/{3}+\et; \x)$, and 
\[-(T/3+\et) = 2T/3 -\et\quad\mbox{ mod $T$},\]
 the right hand sides of  
Equations~(\ref{torus:equation}) at $(\theta_1^1, \theta_2^1) = 
(2T/3-\et, T/3+\et)$ are 
 \begin{subequations}\label{torus:equation:fixed:points}
\begin{align}
 &(\c_1-\c_2+\c_5- \c_4- \c_7) H\lt(\frac{T}{3} +\et; \x\rt), \\ 
& (\c_3-\c_2+\c_6  - \c_4 - \c_7)H\lt(\frac{T}{3} +\et; \x\rt),
\end{align}
\end{subequations}
which both are zero by  Equations~(\ref{balance2}). Therefore, $(\theta_1^1,
\theta_2^1)$ is a fixed point of  Equations~(\ref{torus:equation}). 

To study the stability of $(\theta_1^1, \theta_2^1)$, we consider the linearization
of Equations~(\ref{torus:equation}) and evaluate the Jacobian of their
right hand side at  $(\theta_1^1, \theta_2^1)= (2T/3-\et, T/3+\et)$:
\begin{equation}\label{Jacobian_1}
J_1\;=\; -  \left(\begin{array}{ccc}\c_5 H'\lt(\frac{T}{3}+\et;\x\rt) +\c_4 H'\lt(\frac{2T}{3}-\et;\x\rt) & \c_7 H'\lt(\frac{T}{3}+\et;\x\rt)  
\\ \\
 \c_4 H'\lt(\frac{2T}{3}-\et;\x\rt) & \c_6 H'\lt(\frac{2T}{3}-\et;\x\rt)  + \c_7 H'\lt(\frac{T}{3}+\et;\x\rt) \end{array}\right),
\end{equation}
where $H'$ stands for the derivative $dH/d\theta$.
A calculation shows that the trace and the determinant of $J_1$ at 
$(\theta_1^1, \theta_2^1)$ are as in Equation~(\ref{trace_det}). 
Since $\mbox{Tr} <0$ and $\mbox{Det} >0$, both eigenvalues
of $J_1$ have  negative real parts and $(\theta_1^1, \theta_2^1)$ is a
stable fixed point of Equations~(\ref{torus:equation}).
\end{proof}

\bcor \label{cor:tetra}
Assume that $(\theta_1^1, \theta_2^1)= (2T/3-\et, T/3+\et)$ is a fixed point
of Equations (\ref{torus:equation}). Then, 
\bi
\item  $(\theta_1^2, \theta_2^2)= (T/3+\et, T/3+\et)$,
\item  $(\theta_1^3, \theta_2^3)= (T/3+\et, 2T/3-\et)$ which corresponds to
 a backward transition gait; and 
\item  $(\theta_1^4, \theta_2^4)= (2T/3-\et, 2T/3-\et)$, 
\ei 
are also fixed points of Equations~(\ref{torus:equation}). 
\ecor
\bp 
Since $-(T/3+\et) = 2T/3 -\et$ mod $T$ and by Equation~(\ref{eta}),
$H({2T}/{3} - \et; \x) = H({T}/{3}+\et,;\x)$, the right hand sides of
Equations~(\ref{torus:equation}) at $(\theta_1^1, \theta_2^1) = 
(2T/3-\et, T/3+\et)$ are equal to the  right hand sides of 
Equations~(\ref{torus:equation}) at $(\theta_1^i, \theta_2^i)$, $i=2,3,4,$
and  both are therefore equal to zero. 
\ep 

\bremark\label{index_zero}
Besides the four fixed points $(\theta_1^i, \theta_2^i)$, $i=1,2,3,4,$
and depending  on their stability types,  Equations~(\ref{torus:equation})
may or may not admit more fixed points. By the Euler characteristic 
\cite[Section 1.8]{Holmes_book}, the sum of the indices of all the fixed
points on a 2-torus must be zero; thus allowing us to infer the existence
of additional fixed points. 
\eremark

Next we determine the coupled stepping frequency $\hat\omega$
such that  the transition gaits defined in Equations~(\ref{transition_gaits})
become solutions of Equations~(\ref{eq.osc3}). 

\begin{proposition} \label{omega_hat}
If the coupling strengths $c_i$ satisfy Equations~(\ref{balance2})
and~(\ref{trace_det}), then for any $\et \in[0,T/6]$, Equations~(\ref{eq.osc3})
admit the following  stable $T$-periodic solutions
\ZA{
 \begin{subequations}\label{OMEGA1}
\begin{align}
 \FR(\et)&:= \lt(\hat\omega t+\frac{2T}{3}-\et, \; \hat\omega t, \; \hat\omega t+\frac{T}{3}+\et; \;\;
  \hat\omega t+\frac{T}{3}-2\et, \; \hat\omega t+\frac{2T}{3}-\et, \; \hat\omega t\rt),\\
\FL(\et)& := \left(\hat\omega t+\frac{2T}{3} -\et, \; \hat\omega t, \; \hat\omega t+\frac{T}{3}+\et;\;\;
\hat\omega t, \; \hat\omega t+\frac{T}{3}+\et, \; \hat\omega t+\frac{2T}{3} +2\et\right),
\end{align}
\end{subequations}
}
where the coupled stepping frequency  $\hat\omega= \hat\omega(\x)$,  satisfies 
\[
\hat\omega = \omega(\x) + (\c_1+\c_5) H\lt(\frac{2T}{3}-\et; \x\rt)= \omega(\x)+ (\c_2+\c_4+\c_7) H\lt(\frac{2T}{3}-\et; \x\rt)
= \omega(\x) + (\c_3+\c_6) H\lt(\frac{2T}{3}-\et; \x\rt).
\]
\end{proposition}

 \begin{proof}
 By the definition of $\hat\omega$, and using Equation~(\ref{eta}), it can 
 be seen that both $\FR(\et)$ and $ \FL(\et)$ are $T$-periodic solutions of
 Equations~(\ref{eq.osc3}). To check the stability of these solutions, we
 linearize the right hand side of Equations~(\ref{eq.osc3}) at $\FR(\et)$ 
 and $ \FL(\et)$ to obtain
 \[ 
 J_2 =\left(\begin{array}{cc}A & 0 \\0 & A\end{array}\right),
  \]
 where $0$ represents a $3\times3$ zero matrix and 
  \[ 
  A=\left(\begin{array}{ccccc}
  -\c_5 H'\lt(\frac{T}{3}+\et; \xi\rt) &\c_5 H'\lt(\frac{T}{3}+\et; \xi\rt) &0 \\ \\
  \c_4 H'\lt(\frac{2T}{3}-\et; \xi\rt) & -\c_4 H'\lt(\frac{2T}{3}-\et; \xi\rt)-\c_7 H'\lt(\frac{T}{3}+\et; \xi\rt)&\c_7 H'\lt(\frac{T}{3}+\et; \xi\rt)\\ \\
   0&\c_6 H'\lt(\frac{2T}{3}-\et; \xi\rt)&-\c_6 H'\lt(\frac{2T}{3}-\et; \xi\rt)
    \end{array}\right). 
   \]
Note that since we assumed a constant contralateral symmetry between
the right and left legs in Equations~(\ref{eq.osc3}), these sets of legs are
effectively decoupled and hence $J_2$ is a block diagonal matrix. 
    
Some calculations show that the characteristic polynomial of $A$ is
\[
g(\lambda) = -\lambda f(\lambda),
\]
where
\[ f
(\lambda) = \lambda^2 - \mbox{Tr} \lambda + \mbox{Det},
 \]
is the characteristic polynomial of $J_1$ (Equation~(\ref{Jacobian_1})) and
Tr and Det are defined in Equations~(\ref{trace_det}). The non-zero 
eigenvalues of $A$ therefore have the same stability properties as the non-zero
eigenvalues of $J_2$, and  Equations~(\ref{trace_det}) guarantee the
stability of both $\FR(\et)$ and $ \FL(\et)$, up to  overall shifts in phase
 \[
 \phi_i \rightarrow \phi_i+\bar\phi_R,\quad \mbox{for} \; i= 1,2,3, \qquad \mbox{and}\qquad \phi_i \rightarrow \phi_i+\bar\phi_L,  \quad\mbox{for} \;  i= 4,5,6,
 \]
 that correspond to the two zero eigenvalues of $J_2$. 
\end{proof}

\bremark
The balance condition Equation~(\ref{balance2}) is sufficient for the
existence of tripod or tetrapod gaits. In Section \ref{DATA}, we will show 
the existence of such gaits for coupling strengths which approximate 
balance and also which are far from balance.
\eremark

\subsection{Existence with balance condition and equal contralateral couplings}
\label{Existence_with_balance+RCsym}

In Proposition \ref{prop:tetra:stability}, we provided sufficient conditions
for the stability of tetrapod gaits when the coupling strengths satisfy
the balance condition, Equation~(\ref{balance2}).

In this section, in addition to the balance condition, we assume that  $\c_1=\c_2=\c_3$.
Then under some extra conditions on $\c_i$'s and $H$, 
we show that for any $\et\in[0,T/6]$, the fixed point  $(2T/3-\et, T/3+\et)$ is stable. 
The reason that we are interested in the assumption $\c_1=\c_2=\c_3$ is the following
estimated coupling strengths from fruit fly data \cite{Couzin_notes_16}.
We will return to this data set in Section \ref{DATA}. 
\[ \c_1=2.9145,\;
 \c_2=2.5610,\;
\c_3=2.6160,\;
 \c_4=2.9135,\;
 \c_5=5.1800,\;
 \c_6=5.4770,\;
 \c_7=2.6165.
 \]
In this set of data, the $\c_i$'s approximately satisfy the balance condition and also
\[\c_1\approx \c_2 \approx \c_3,\; \c_5 \approx \c_4+\c_7 \approx \c_6.\] 

\begin{proposition} \label{special_coupling}
Assume that the coupling strengths $\c_i$ satisfy Equation~(\ref{balance2}) and $\c_1=\c_2=\c_3$. 
 Also assume that $\forall \et \in [0,T/6]$, $H'=dH/d\theta$ satisfies
  \be{condition_dH_1} 
  \quad H'\lt( \frac{T}{3} +\eta; \x\rt)+H'\lt( \frac{2T}{3} -\eta; \x\rt) >0.
  \ee
  Let $\alpha$ and $\alpha_{max} $ be as follows:
  \be{alpha_alphamax}
  \alpha:=\frac{\c_4}{\c_4+\c_7},
   \quad 
   \alpha_{max} :=\frac{H'\lt(\frac{T}{3}+\et; \x\rt)}{H'\lt(\frac{T}{3}+\et; \x\rt)-H'\lt(\frac{2T}{3}-\et; \x\rt)}\;.
  \ee
If 
 \be{lambda12negative}
 (\alpha_{max}-\alpha)\lt(H'\lt( \frac{T}{3} +\eta; \x\rt)-H'\lt( \frac{2T}{3} -\eta; \x\rt)\rt) >0,
 \ee
 then $(\theta_1^1, \theta_2^1)= (2T/3-\et, T/3+\et)$
 is a stable fixed point of Equations~(\ref{torus:equation}) and if 
 \be{lambda12positive}
 (\alpha_{max}-\alpha)\lt(H'\lt( \frac{T}{3} +\eta; \x\rt)-H'\lt( \frac{2T}{3} -\eta; \x\rt)\rt) <0,
 \ee
  then $(\theta_1^1, \theta_2^1)= (2T/3-\et, T/3+\et)$ is a saddle point.
 \end{proposition}

\begin{proof}
Using the assumption $\c_1=\c_2=\c_3$ and  Equation~(\ref{balance2}),
the following relations among  the coupling strengths hold:
\begin{equation} \label{coupling_strengths}
\c_1 = \c_2 = \c_3, \; \c_5 = \c_4 + \c_7 = \c_6.
\end{equation}
Letting 
\be{alpha}
\alpha := \frac{\c_4}{\c_4+\c_7}, \quad (0< \alpha<1),
\ee
 and making a change of time variable that eliminates $\c_5$, 
 Equations~(\ref{torus:equation}) become
\begin{subequations} \label{eq.osc.simplified}
\begin{align}
&\dot\theta_1= H(-\theta_1; \x) - \alpha  H(\theta_1; \x)- (1-\alpha)  H(\theta_2; \x), \\ 
&\dot\theta_2 = H(-\theta_2; \x) -  \alpha H(\theta_1; \x) - (1-\alpha) H(\theta_2; \x) .
\end{align}
\end{subequations}
 
Consider the linearization of  Equation~(\ref{eq.osc.simplified}) at $(\theta_1, \theta_2)$:
\[
J_3(\theta_1,\theta_2)\;=\; -  \left(\begin{array}{ccc}H'(-\theta_1; \x)+\a H'(\theta_1; \x) & (1-\a)H'(\theta_2; \x) \\
 \\ 
\a H'(\theta_1; \x) & H'(-\theta_2; \x) + (1-\a) H'(\theta_2; \x)\end{array}\right).
\]
Standard calculations show that the eigenvalues of $J_3$ at
$(\theta_1^1, \theta_2^1)=(2T/3-\et, T/3+\et)$ are
\[ \lambda_1^1= -H'\lt(\frac{2T}{3}-\et; \x\rt)-H'\lt(\frac{T}{3}+\et; \x\rt), \quad \mbox{and}  \quad  
\lambda_2^1= -(1-\a)H'\lt(\frac{T}{3}+\et; \x\rt)-\a H'\lt(\frac{2T}{3}-\et; \x\rt). \]
By Equation~(\ref{condition_dH_1}), $H'\lt(\frac{2T}{3}-\et; \x\rt)+H'\lt(\frac{T}{3}+\et; \x\rt)>0$, hence $\lambda_1^1<0$. 
A calculation shows that $\lambda_2^1<0$ if 
$H'\lt( \frac{T}{3} +\eta; \x\rt)-H'\lt( \frac{2T}{3} -\eta; \x\rt)>0$ and $\a<\a_{max}$ 
 or 
$H'\lt( \frac{T}{3} +\eta; \x\rt)-H'\lt( \frac{2T}{3} -\eta; \x\rt)<0$ and $\a>\a_{max}$.
 Therefore, if Inequality (\ref{lambda12negative}) holds, then
$(\theta_1^1, \theta_2^1)=(2T/3-\et, T/3+\et)$  is a stable fixed point. Otherwise, $(\theta_1^1, \theta_2^1)=(2T/3-\et, T/3+\et)$  is a saddle point. 
\end{proof}

In the following corollary,  assuming that Equation (\ref{condition_dH_1}) holds and $H'\lt( {2T}/{3} -\eta; \x\rt) < 0$, we verify the stability types of the other fixed points introduced in Corollary \ref{cor:tetra} (in Section~\ref{application_BN} we will see that the coupling function computed for the bursting neuron model satisfies both of these assumptions):

 \begin{proposition} \label{special_coupling_corollary}
Assume that for some $\et \in [0,T/6]$, Equation (\ref{condition_dH_1}) holds and $H'\lt( {2T}/{3} -\eta; \x\rt) < 0$. Then

\begin{enumerate}[leftmargin=*]
\item 
$\lt(\theta_1^2, \theta_2^2\rt)= (T/3+\et, T/3+\et)$ is a saddle  point. 

\item  $\lt(\theta_1^3, \theta_2^3\rt)= (T/3+\et, 2T/3-\et)$, which corresponds to a backward tetrapod gait, is a sink if  
\be{alphamin}
\alpha>{\alpha_{min}:= \frac{H'\lt(\frac{2T}{3}-\et; \x\rt)}{H'\lt(\frac{2T}{3}-\et; \x\rt)-H'\lt(\frac{T}{3}+\et; \x\rt)}}\;,
\ee
and a saddle point if $\a_{min}>0$ and $\a<\a_{min}$. 

\item   $\lt(\theta_1^4, \theta_2^4\rt)= (2T/3-\et, 2T/3-\et)$ is a sink. 
\end{enumerate}
 \end{proposition}

\bp 
Note that for any  $i=1,\ldots,4$,  the fixed point $\lt(\theta_1^i, \theta_2^i\rt)$  either lies on the line $\theta_1=\theta_2$ 
or on the line $\theta_1= - \theta_2$. 
\begin{enumerate}[leftmargin=*]
\item The eigenvalues of $J_3$ at $\lt(\theta_1^2, \theta_2^2\rt)= (T/3+\et, T/3+\et)$ are 
\[
\lambda_1^2=  -H'(2T/3-\et; \x)-H'(T/3+\et; \x)\quad \mbox{and} \quad \lambda_2^2= -H'(2T/3-\et; \x).
\]
By Equation~(\ref{condition_dH_1}),  $\lambda_1^2<0$  and  since we assumed $H'(2T/3-\et; \x)<0$, $\lambda_2^2>0$. 
Therefore, independent of the choice of $\alpha$, $\lt(\theta_1^2, \theta_2^2\rt)$ is always a saddle point. 

\item   The eigenvalues  of $J_3$ at $\lt(\theta_1^3, \theta_2^3\rt)= (T/3+\et, 2T/3-\et)$ are
\[\lambda_1^3=  -H'(2T/3-\et; \x)-H'(T/3+\et;\x) \quad \mbox{and} \quad \lambda_2^3= -(1-\a)H'(2T/3-\et; \x)-\a H'(T/3+\et; \x).\]
By Equation~(\ref{condition_dH_1}), $\lambda_1^3<0$. Since $H'\lt( {2T}/{3} -\eta; \x\rt) < 0$, for $\alpha>\alpha_{min}$,   
$\lambda_2^3<0$. Therefore, $(\theta_1^3, \theta_2^3)$ is a  sink. 
Note that for $\alpha<\alpha_{min}$, $\lambda_2^3$ becomes positive and so 
$(\theta_1^3, \theta_2^3)$ becomes a saddle point. 

\item The eigenvalues  of $J_3$ at $(\theta_1^4, \theta_2^4)= (2T/3-\et, 2T/3-\et)$ are 
\[\lambda_1^4=  -H'(2T/3-\et; \x)-H'(T/3+\et; \x) \quad \mbox{and} \quad \lambda_2^4= -H'(T/3+\et; \x).\]
$H'(2T/3-\et; \x)+H'(T/3+\et; \x)>0$ and $H'(2T/3-\et; \x)<0$ imply that $H'(T/3+\et; \x)>0$. Therefore,  
 both eigenvalues are negative and independent of the choice of $\alpha$, $(\theta_1^4, \theta_2^4)$ is always a sink. 
\end{enumerate}
\vspace{-0.9cm} 
\ep

On the other hand, if we assume that $H'(2T/3-\et;\x)>0$, then all stable fixed points  
become saddle points and the saddle points  become stable fixed points.

 \begin{proposition}\label{fixedpoints_zero_T2}
 In addition to $\lt(\theta_1^i, \theta_2^i\rt)$, $i=1,\ldots,4$, when $\c_1=\c_2=\c_3$, 
Equations~(\ref{torus:equation}) admit the following fixed points. 

\begin{enumerate}[leftmargin=*]

\item  $(\theta_1^5, \theta_2^5) = (T/2,T/2)$ is a fixed point and if $\exists \, \x_* \in [\x_1,\x_2]$ such that for $\x < \x_*$, 
 $H'\lt({T}/{2}; \xi\rt) < 0$, while for  $\x > \x_*$, $H'\lt({T}/{2}; \xi\rt) > 0$, 
 then the fixed point $(T/2, T/2)$ changes its stability to a sink from a source as $\x$ increases. 
 
 \item $\lt(\theta_1^6, \theta_2^6\rt)=(0,0)$ is a fixed point and when $H'(0; \x)<0$, it is a source. 
\end{enumerate}

  \end{proposition}

\bp
\begin{enumerate}[leftmargin=*]

\item The eigenvalues  of $J_3$ at $(T/2, T/2)$ are 
\[\lambda_1^5 = -H'(T/2; \x)\quad \mbox{and}\quad \lambda_1^5  =-2H'(T/2; \x),\]
 so  the stability depends on the sign of $H'(T/2; \x)$, which by assumption is positive for $\x<\x_*$.  
  Hence, for $\x<\x_*$, both eigenvalues are positive and $(\theta_1^5, \theta_2^5)$ is a source
and for  for $\x>\x_*$,  both eigenvalues becomes negative and hence  
$(\theta_1^5, \theta_2^5)= (T/2, T/2)$ becomes a sink. 

\item The eigenvalues  of $J_3$ at $(0,0)$ are 
\[\lambda_1^6 = -H'(0; \x)\quad \mbox{and}\quad \lambda_1^6  =-2H'(0; \x),\]
 so the stability depends on the sign of $H'(0; \x)$, which we assumed 
 is negative. Therefore, $(0,0)$ is a source.
\end{enumerate}
\vspace{-0.5cm} 
\ep

\ZA{Note that as explained in Remark~\ref{index_zero}, 
by the Euler characteristic of zero for the 2-torus, 
there should exist more fixed points (e.g. saddle points).}

\begin{proposition} \label{invariant_diagonal_line}
If $\c_1=\c_3$ and $\c_5=\c_6$, then $\theta_1=\theta_2$ is an invariant line.
\end{proposition}

\bp
Setting $\c_1=\c_3$ and $\c_5=\c_6$ in Equations~(\ref{torus:equation}), 
we conclude that $\dot\theta_1=\dot\theta_2$. Hence $\theta_1=\theta_2$
is invariant. 
\ep

\bcor
\label{reflection_symmetric}
Under the conditions of Proposition \ref{special_coupling},  $\theta_1 =
\theta_2$ is an invariant line. In addition, if $\c_4=\c_7$, then the system
is reflection symmetric with respect to $\theta_1=\theta_2$; i.e., 
if $(\dot\theta_1,\dot\theta_2)=(a,b)$ at $(\bar\theta_1, \bar\theta_2)$,
then $(\dot\theta_1,\dot\theta_2)=(b,a)$ at $(\bar\theta_2, \bar\theta_1)$. 
\ecor

\bp 
Setting $(\theta_1, \theta_2) = (\bar\theta_1, \bar\theta_2)$  and 
$(\theta_1, \theta_2) = (\bar\theta_2, \bar\theta_1)$ in 
Equations~(\ref{eq.osc.simplified}) yields the result.
\ep

\ZA{In the following sections we first apply the results of this section to the coupling functions
 computed  for the bursting neuron model (Section \ref{application_BN}). 
 Then, we characterize a class of functions $H$ which satisfies 
 Assumption \ref{eta_assumption} (Section \ref{application_general_H}).}

\section{\ZAA{Application to the bursting neuron model}}\label{application_BN}
\ZA{In Section \ref{pair_coupled}, for some $\delta$ and $I_{ext}$ values,  we numerically 
computed  the coupling function $H_{BN}$ for the bursting neuron model 
(see Figures~\ref{H_PRC_versus_delta} and \ref{H_PRC_versus_Iext}). 
Here we show that  the results of  Section~\ref{existence_and_stability} apply to the coupling function 
$H_{BN}$.}
 
\begin{lemma}
The coupling function $H_{BN}$, which is computed numerically from the bursting
neuron model, satisfies Assumption \ref{eta_assumption}. 
\end{lemma}

\bp
 Figure~\ref{eta_delta_Iext} shows the graphs of  $\et=\et(\x)$, the solutions of  Equation~(\ref{eta}) for $H=H_{BN}$, where 
$\x=\delta\in\ZAA{[\delta_1,\delta_2]} =[0.0097,0.04]$ (left) and $\x=I_{ext}
\in \ZAA{[I_1, I_2]} =[35.65,37.7]$  (right). (Note that solving Equation~(\ref{eta}) is equivalent to solving $G_{BN}(\theta; \x)=0$ for $\theta$, where 
$G_{BN}(\theta; \x) := H_{BN}(\theta; \x) - H_{BN}(-\theta; \x)$.) Note that $\et$ is the unique solution of Equation~(\ref{eta}) which is non-decreasing and onto. Therefore,  Assumption  \ref{eta_assumption}  is satisfied. 
\begin{figure}[h!]
\begin{center}
 \includegraphics[scale=.3]{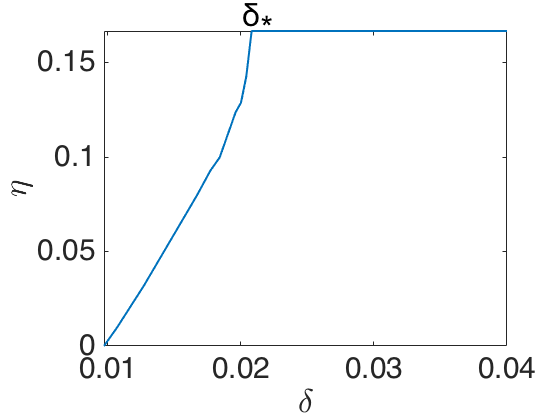}\qquad
 \includegraphics[scale=.3]{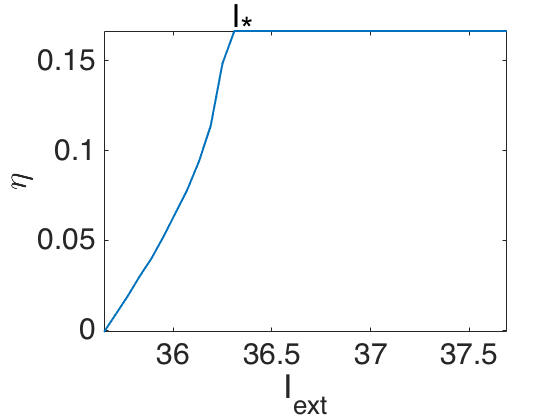}
\end{center}
\caption{The solution $\et(\x)$ of $H_{BN}\lt({2T}/{3} - \et; \x\rt) = H_{BN}\lt({T}/{3}+\et; \x\rt)$ where  $\x=\delta$ and 
$\et:[0.0097,0.04] \to [0,1/6]$ (left); and where $\x=I_{ext}$ and $\et:[35.65,37.7] \to [0,1/6]$ (right).}
\label{eta_delta_Iext}
\end{figure}
\ep

\subsection{Balance condition}

Since $H_{BN}$ satisfies Assumption \ref{eta_assumption}, one can apply Proposition \ref{prop:tetra:stability} to show that under balance condition for the coupling strengths, and Equations (\ref{trace_det}), $(2T/3-\et, T/3+\et)$ is a stable fixed point of Equation~(\ref{torus:equation}) with $H=H_{BN}$. 
In Section~\ref{Example_Qualitative_behavior}, Figure~\ref{NC_PP_balance_Tr_Det},  we showed the nullclines and
phase planes of  Equations (\ref{torus:equation}) with coupling strengths given in Equation~(\ref{example_only_balance}). 
Note that those coupling strengths satisfy the balance equation and for $\delta=0.0097$, they satisfy 
 Equations~(\ref{trace_det})
\ZA{(Tr $\approx -2.78 <0$ and  Det $\approx 0.61 >0$).} 
In Figure~\ref{NC_PP_balance_Tr_Det} (left to right) we observe that for 
small $\delta$, there exist 3 sinks corresponding to $(\theta_1^i, \theta_2^i)$,
$i=1,3,4$, and  1 saddle point corresponding to $(\theta_1^2, \theta_2^2)$. 
 In addition, there exist 2 sources (one located at $(1/2,1/2)$ and the other
 one at $(0,0)$), and  $4$ more saddle points. When $\delta$ is large, 
  for $i=1,2,3,4$, $(\theta_1^i, \theta_2^i)$ merge to $(1/2,1/2)$,  and 
 $(1/2,1/2)$, which corresponds to the tripod gait, becomes a sink. The
 unstable fixed point $(0,0)$ and the two remaining saddle points, near the boundary, 
 preserve their stability types. 
 
\subsection{Balance condition and equal contralateral couplings}

In this section we apply Proposition~\ref{special_coupling} to $H_{BN}$ to show existence and stability of tetrapod and tripod gaits. 

\begin{proposition} \label{special_coupling_HBN}
Consider Equations~(\ref{eq.osc.simplified}) for $H=H_{BN}$. 
 If $\alpha<\alpha_{max}$ (as defined in Equation~(\ref{alpha_alphamax})), then 
 $(\theta_1^1, \theta_2^1)= (2T/3-\et, T/3+\et)$
 is a stable fixed point of Equations~(\ref{torus:equation}) and if $\alpha>\alpha_{max}$
  then $(\theta_1^1, \theta_2^1)= (2T/3-\et, T/3+\et)$ is a saddle point.
   \end{proposition}

\bp

Figure \ref{eigenvalues_tetrapods_diffrent_alpha_delta} shows that $H_{BN}'(2T/3-\et; \x)+H_{BN}'(T/3+\et; \x)>0$, for $\x=\delta$ and $\x=I_{ext}$. Hence, Equation~(\ref{condition_dH_1}) holds. 
\begin{figure}[h!]
\begin{center}
\includegraphics[scale=.32]{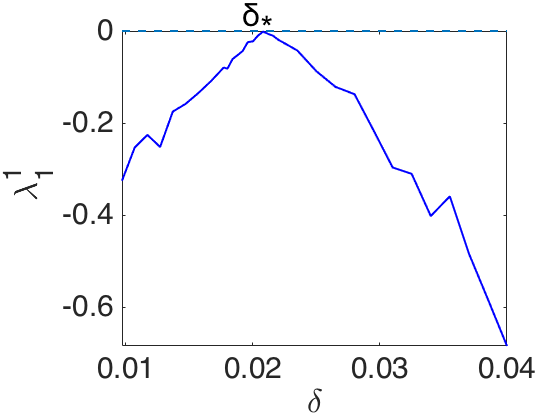}\qquad
\includegraphics[scale=.32]{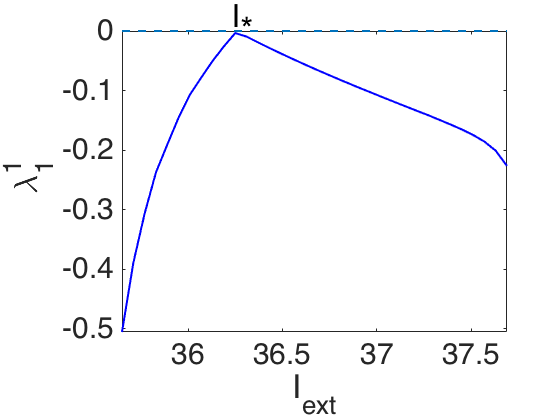}
\end{center}
\caption{$\lambda_1^1= -H_{BN}'(2T/3-\et; \x)-H_{BN}'(T/3+\et; \x)$ vs.
$\x=\delta$ (left) and $\xi=I_{ext}$ (right) are  shown. }
\label{eigenvalues_tetrapods_diffrent_alpha_delta}
\end{figure}
Figure \ref{dH13_dH23_xi} shows that $H_{BN}'(2T/3-\et; \x)<0$ and  $H_{BN}'(T/3+\et; \x)>0$. Therefore,  
$H_{BN}'(T/3+\et; \x) - H_{BN}'(2T/3-\et; \x)>0$. 
Hence, by Proposition~\ref{special_coupling}, If $\alpha<\alpha_{max}$, Equation~(\ref{lambda12negative}) holds and  
 $(\theta_1^1, \theta_2^1)= (2T/3-\et, T/3+\et)$
 is a stable fixed point of Equations~(\ref{torus:equation}) and 
 if $\alpha>\alpha_{max}$, Equation~(\ref{lambda12positive}) holds and
  $(\theta_1^1, \theta_2^1)= (2T/3-\et, T/3+\et)$ is a saddle point.
\begin{figure}[h!]
\begin{center}
\includegraphics[scale=.32]{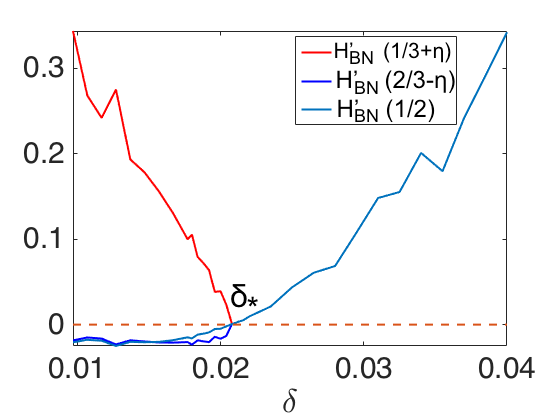}\qquad\qquad
\includegraphics[scale=.33]{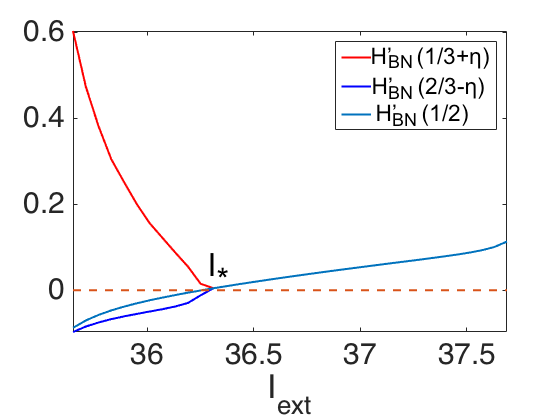}
\end{center}
\caption{ $H_{BN}'(2T/3-\et; \x)$,  $H_{BN}'(T/3+\et; \x)$, and $H_{BN}'(T/2; \x)$ vs. $\x=\delta$ (left) and $\x=I_{ext}$ (right) are  shown. 
Note that the curves first meet at $\delta_*$ (left) and $I_*$ (right) and subsequently overlap for $\delta>\delta_*$ and $I_{ext}>I_*$. }
\label{dH13_dH23_xi}
\end{figure}
\ep

Moreover, we apply Propositions~\ref{special_coupling_corollary} and \ref{fixedpoints_zero_T2} to show the existence and stability of more fixed points. 

\begin{proposition} \label{other_fixedpoints_HBN}
In Equations~(\ref{eq.osc.simplified}) with $H=H_{BN}$:

\begin{enumerate}[leftmargin=*]
\item $\lt(\theta_1^2, \theta_2^2\rt)= (T/3+\et, T/3+\et)$ is a saddle  point. 

 \item If $\alpha>\alpha_{min}$ (as defined in Equation~(\ref{alphamin})), then 
 $(\theta_1^3, \theta_2^3)= (T/3+\et, 2T/3-\et)$
 is a stable fixed point, otherwise, it is a saddle point.
 
\item   $\lt(\theta_1^4, \theta_2^4\rt)= (2T/3-\et, 2T/3-\et)$ is a sink. 

\item For $\x < \x_*$ ($\x_*=\delta_* \approx 0.0208$ and $\x_*=I_*\approx 36.3$), $(T/2,T/2)$ is a source and   for  $\x > \x_*$, $(T/2, T/2)$ becomes a sink. 
 
\end{enumerate}
 
   \end{proposition}

\bp 

By Figure \ref{eigenvalues_tetrapods_diffrent_alpha_delta}, $H_{BN}'(2T/3-\et; \x)+H_{BN}'(T/3+\et; \x)>0$ and by Figure \ref{dH13_dH23_xi}, $H_{BN}'(2T/3-\et; \x)<0$, and $H_{BN}'(T/2; \x)$ changes sign from negative to positive at $\x=\x_*$
($\x_*=\delta_* \approx 0.0208$ and $\x_*=I_*\approx 36.3$). Therefore, Propositions~\ref{special_coupling_corollary} and \ref{fixedpoints_zero_T2} give the desired results. 
\ep 

\subsection{Phase plane analyses} 
We now study Equations (\ref{eq.osc.simplified}) by analyzing phase planes.
In the following cases we preserve the balance condition and let $\c_1=\c_2=\c_3$, 
but allow $\alpha$ to vary.  First we assume that $\alpha =1/2$ (rostrocaudal symmetry),
for which, by Corollary~\ref{reflection_symmetric}, the system is reflection
symmetric with respect to $\theta_1 = \theta_2$. For example, we let 
\[ \c_1=\c_2=\c_3=0.5,\; \c_4=\c_7=1,\; \c_5=\c_6=2. \]

\begin{figure}[h!]
\begin{center}
\includegraphics[scale=.1]{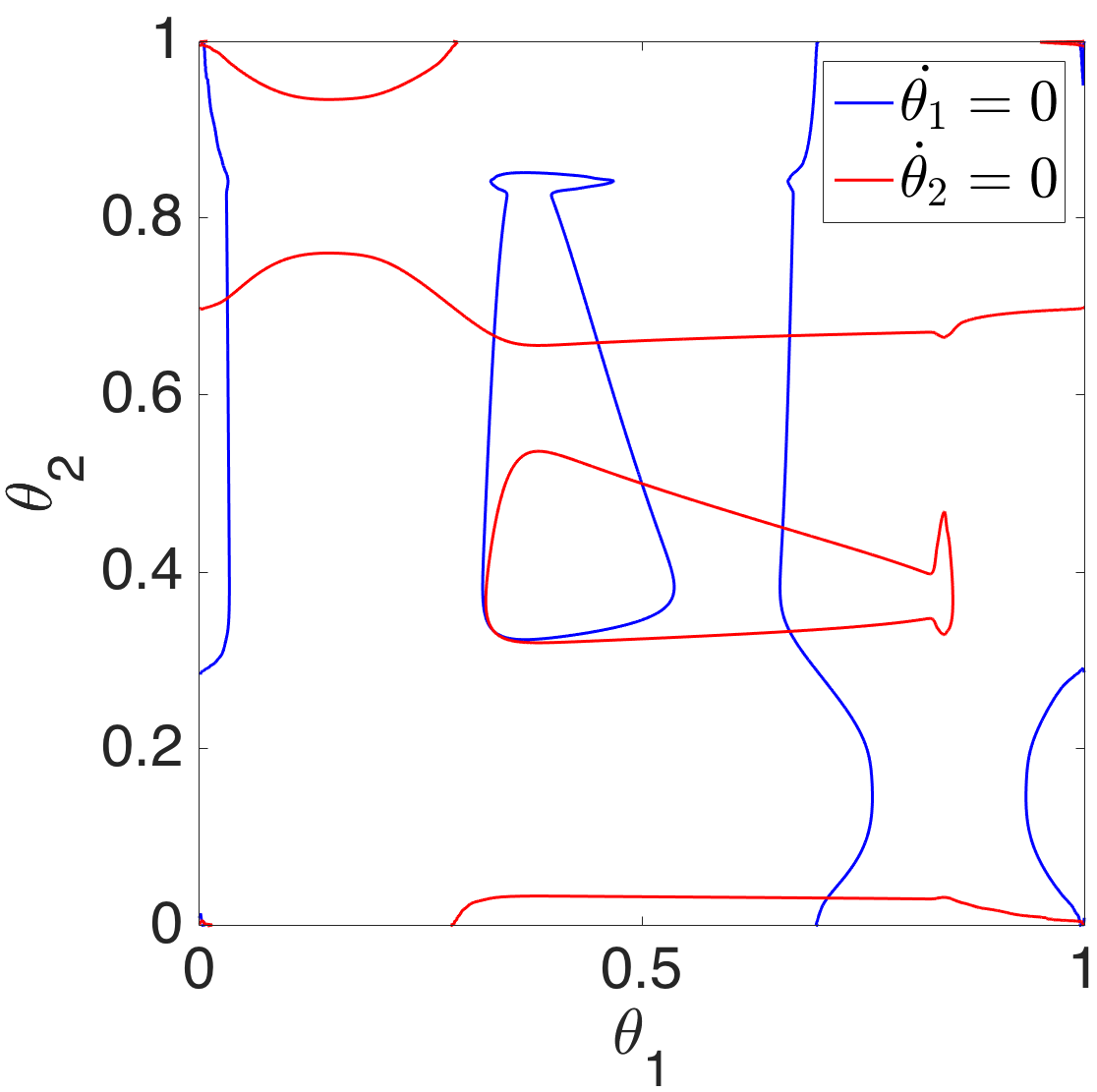}
\includegraphics[scale=.1]{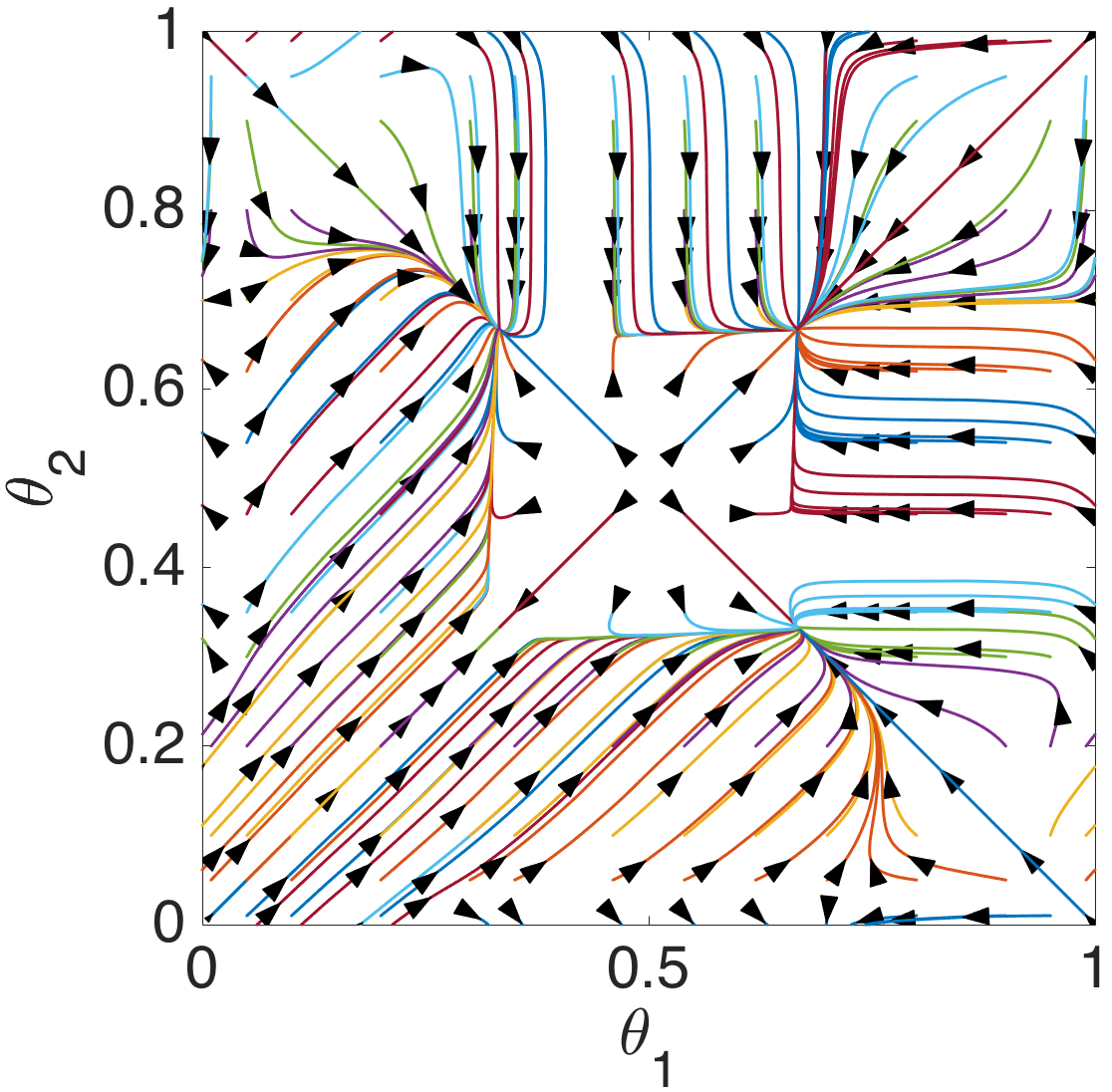}\quad
\includegraphics[scale=.1]{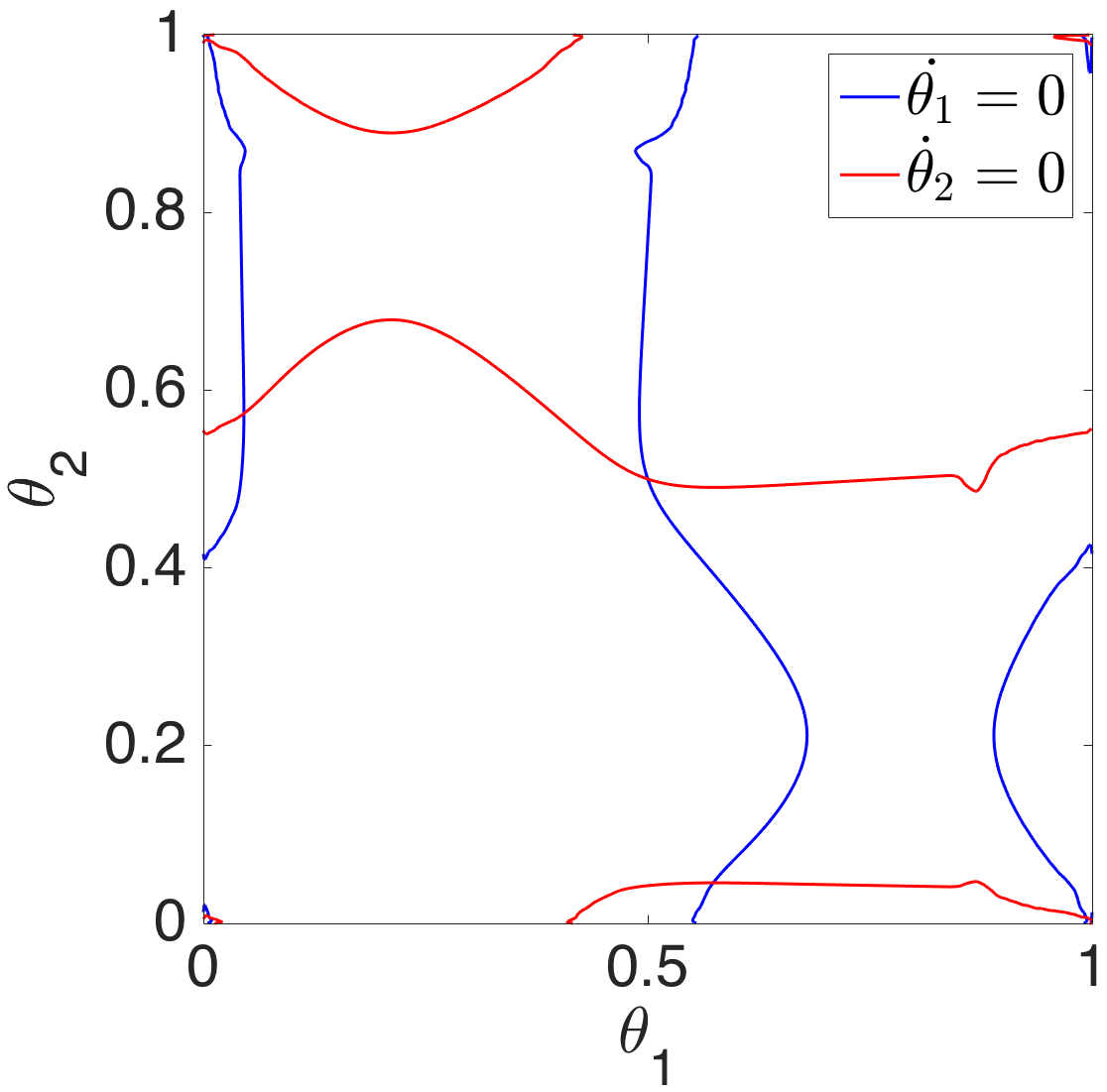}
\includegraphics[scale=.1]{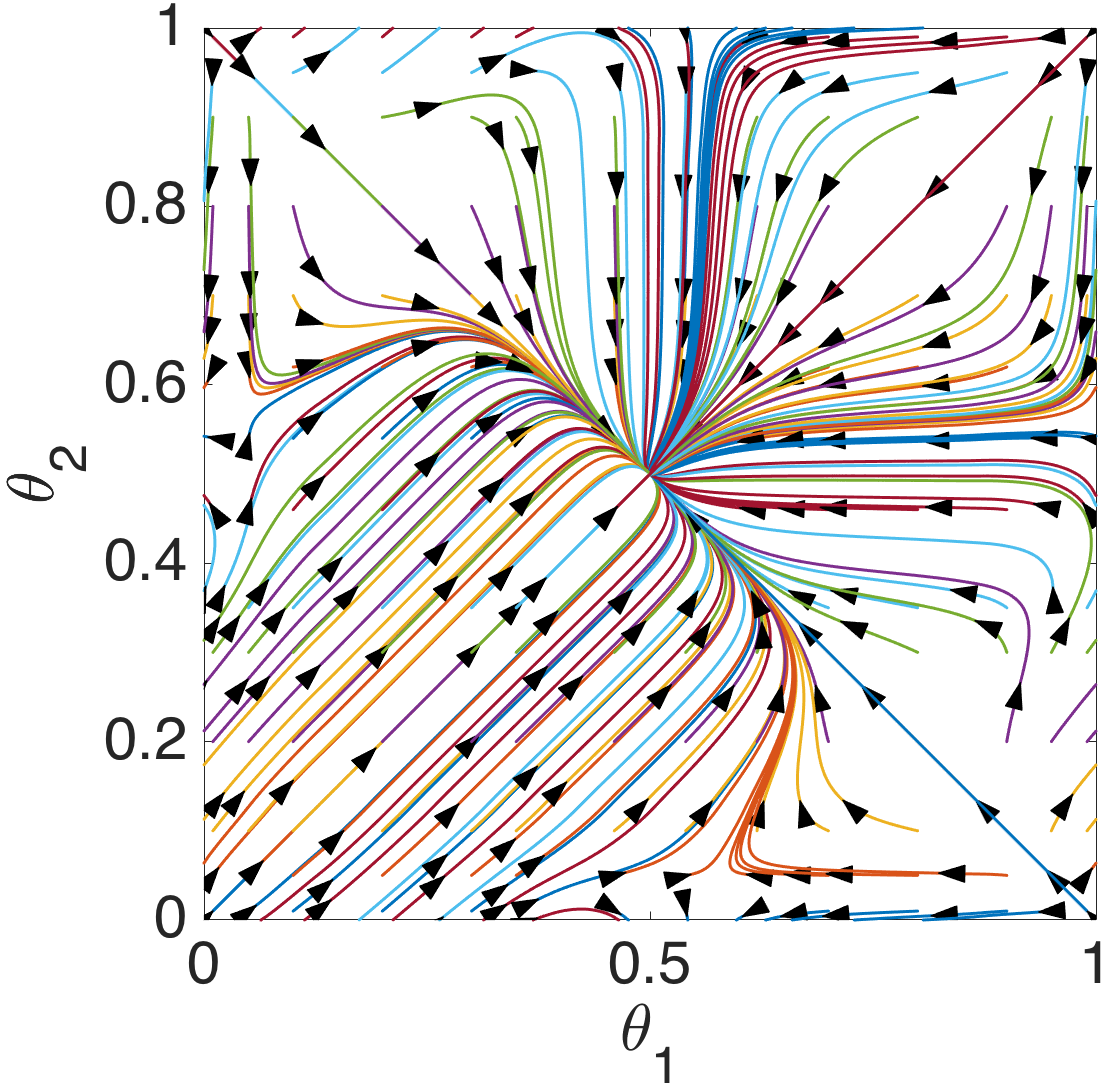}
\includegraphics[scale=.2]{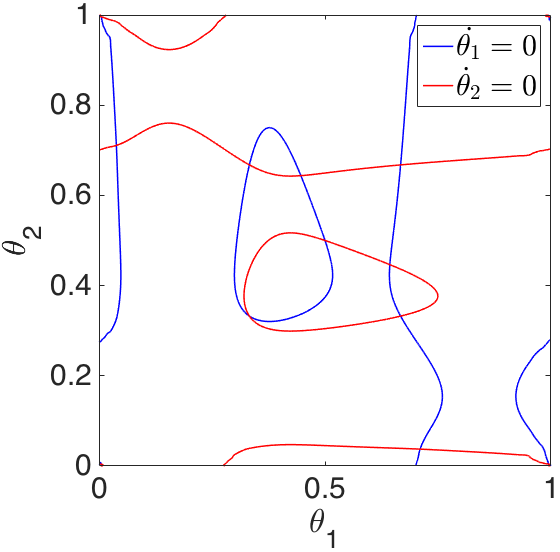}
\includegraphics[scale=.2]{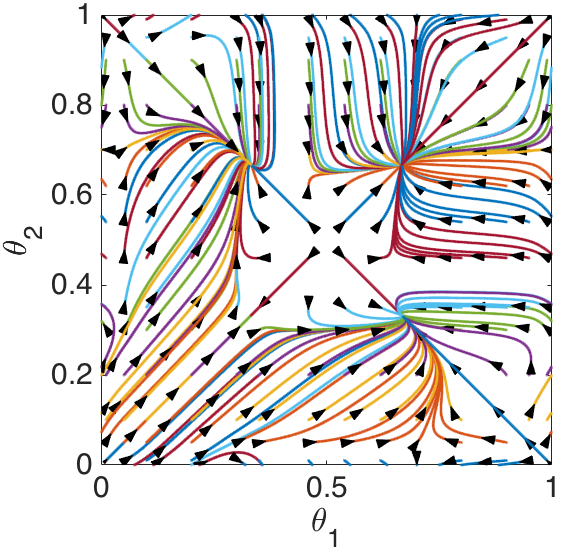}\quad
\includegraphics[scale=.2]{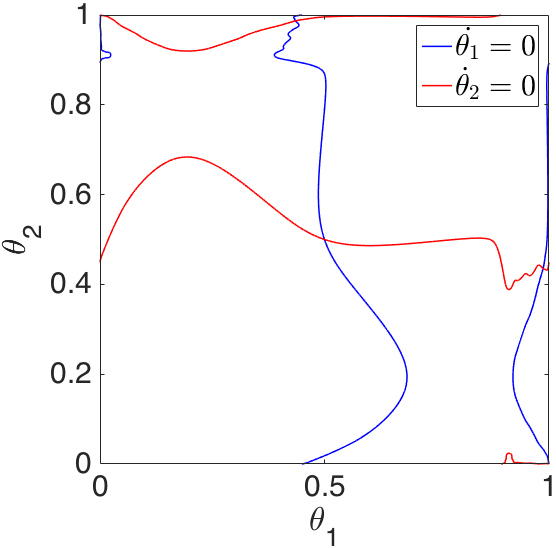}
\includegraphics[scale=.2]{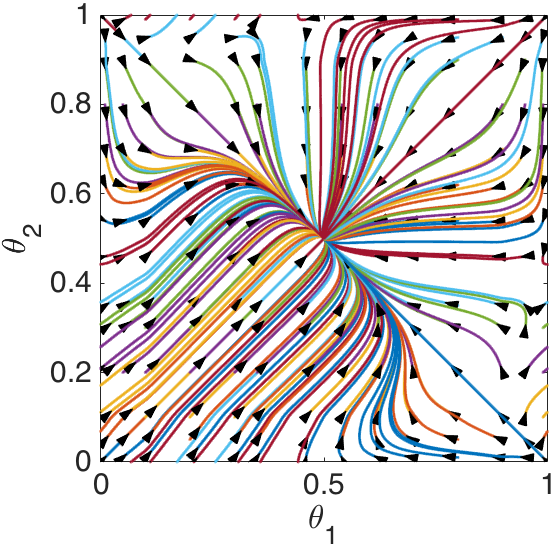}
\end{center}
\caption{Nullclines and  phase planes of Equation~(\ref{eq.osc.simplified})
when $\alpha=1/2$. First row: $\delta=0.0097$ (left) and  $\delta=0.03$ (right).
Second row: $I_{ext}= 35.65$ (left) and $I_{ext}= 37.1$ (right). Note reflection
symmetry.} 
\label{phase differences model}
\end{figure}
Figure~\ref{phase differences model} (first row, left to right) shows the 
nullclines and phase planes of  Equations (\ref{eq.osc.simplified}), for a small 
$\delta = 0.0097 < \delta_*$, and a large $\delta = 0.03 > \delta_*$, respectively. 
Figure~\ref{phase differences model} (second row, left to right) shows the 
nullclines and phase planes of  Equations~(\ref{eq.osc.simplified}), for a small 
$I_{ext}= 35.65 < I_*$ and a large  $I_{ext}= 37.1 > I_*$, respectively. 
 As expected from Proposition~\ref{special_coupling} and
 Proposition~\ref{special_coupling_corollary}, we observe that when
 $\delta$ or $I_{ext}$ is small, there exist 3 sinks corresponding to
 $(\theta_1^i, \theta_2^i)$, $i=1,3,4$, and 2 sources corresponding to
 $(\theta_1^i, \theta_2^i)$, $i=5,6$. In addition, there exist  $5$ saddle
 points, of which one corresponds to $(\theta_1^2, \theta_2^2)$. 
When $\delta$ or $I_{ext}$ is large, $(\theta_1^i, \theta_2^i)$,
for $i=1,2,3,4$ merge to $(\theta_1^5, \theta_2^5)= (1/2,1/2)$,  and we
observe that  $(1/2,1/2)$  which corresponds to the tripod gait, becomes
a sink. The unstable fixed point $(0,0)$ and two saddle points continue 
to exist and preserve their stability types. 

\begin{figure}[h!]
\begin{center}
\includegraphics[scale=.1]{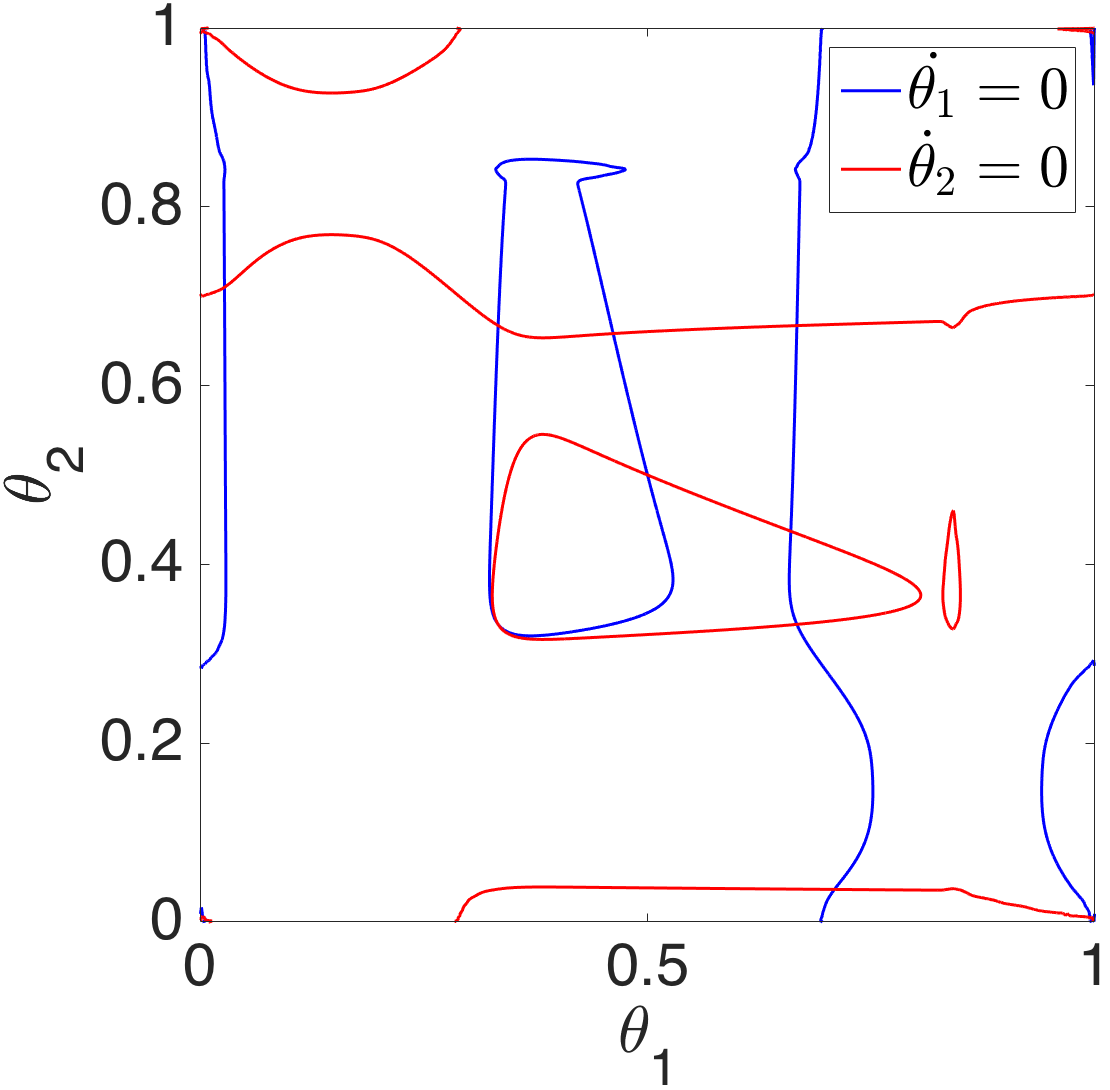}
\includegraphics[scale=.1]{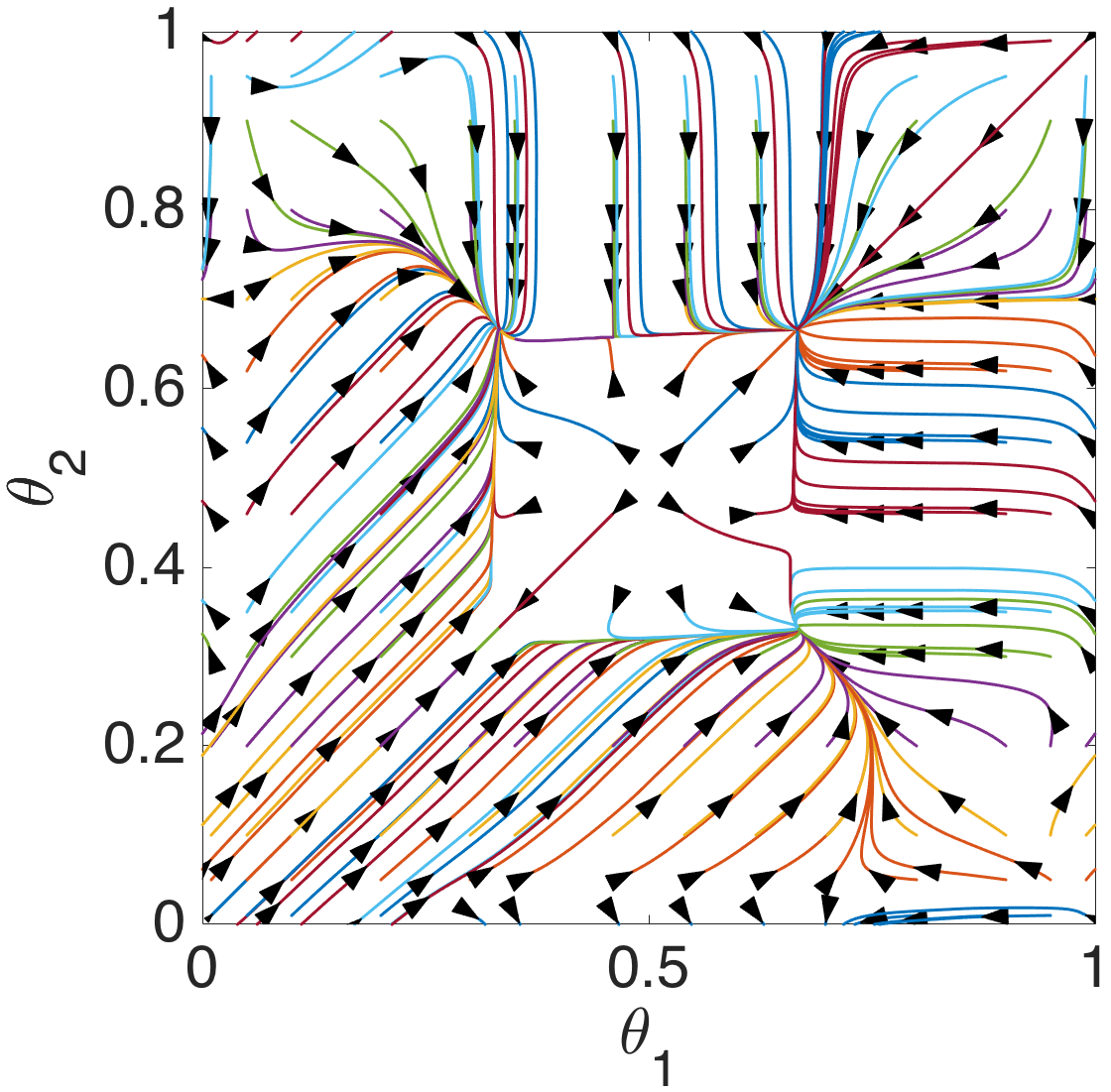}\quad
\includegraphics[scale=.1]{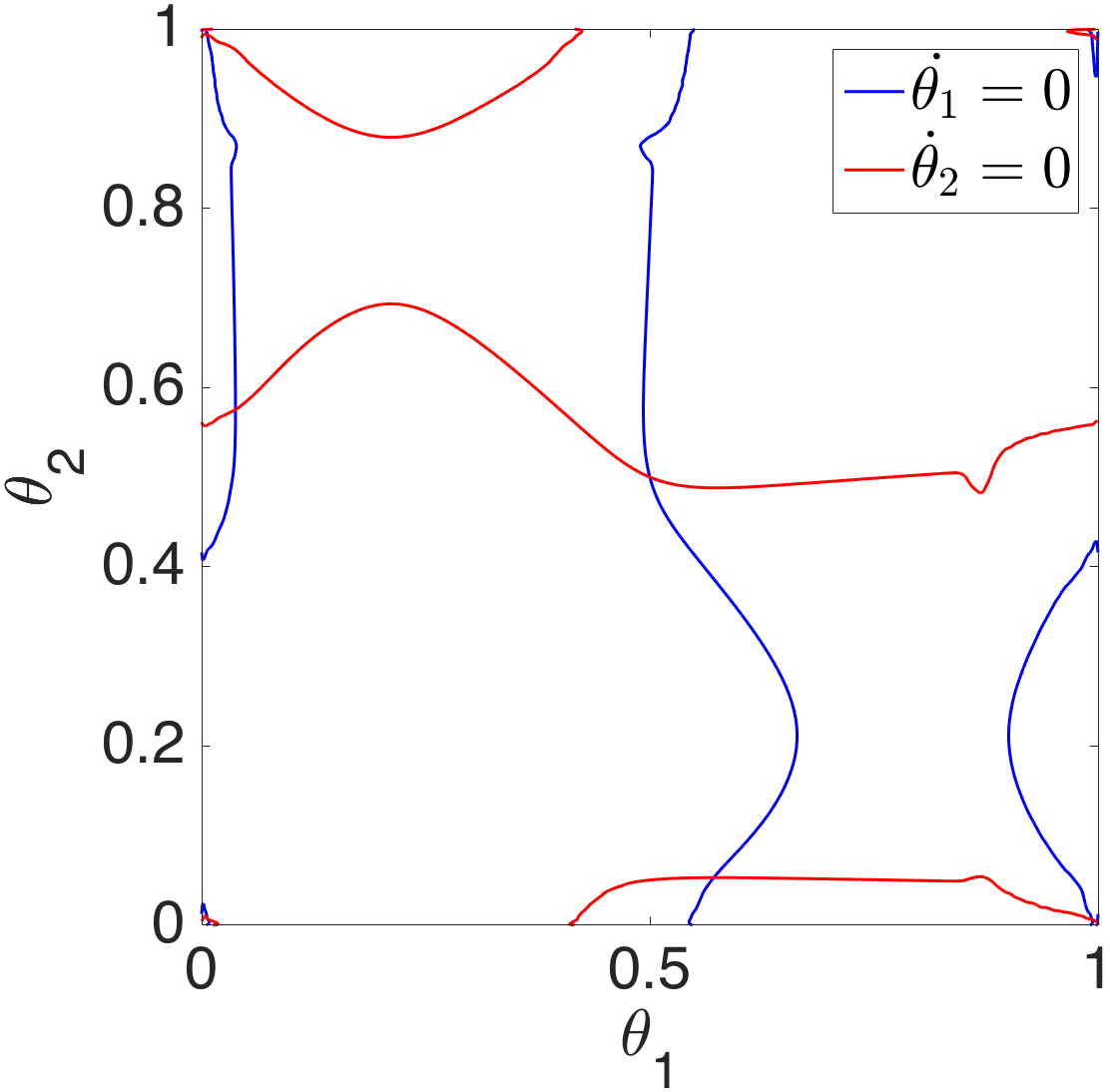}
\includegraphics[scale=.1]{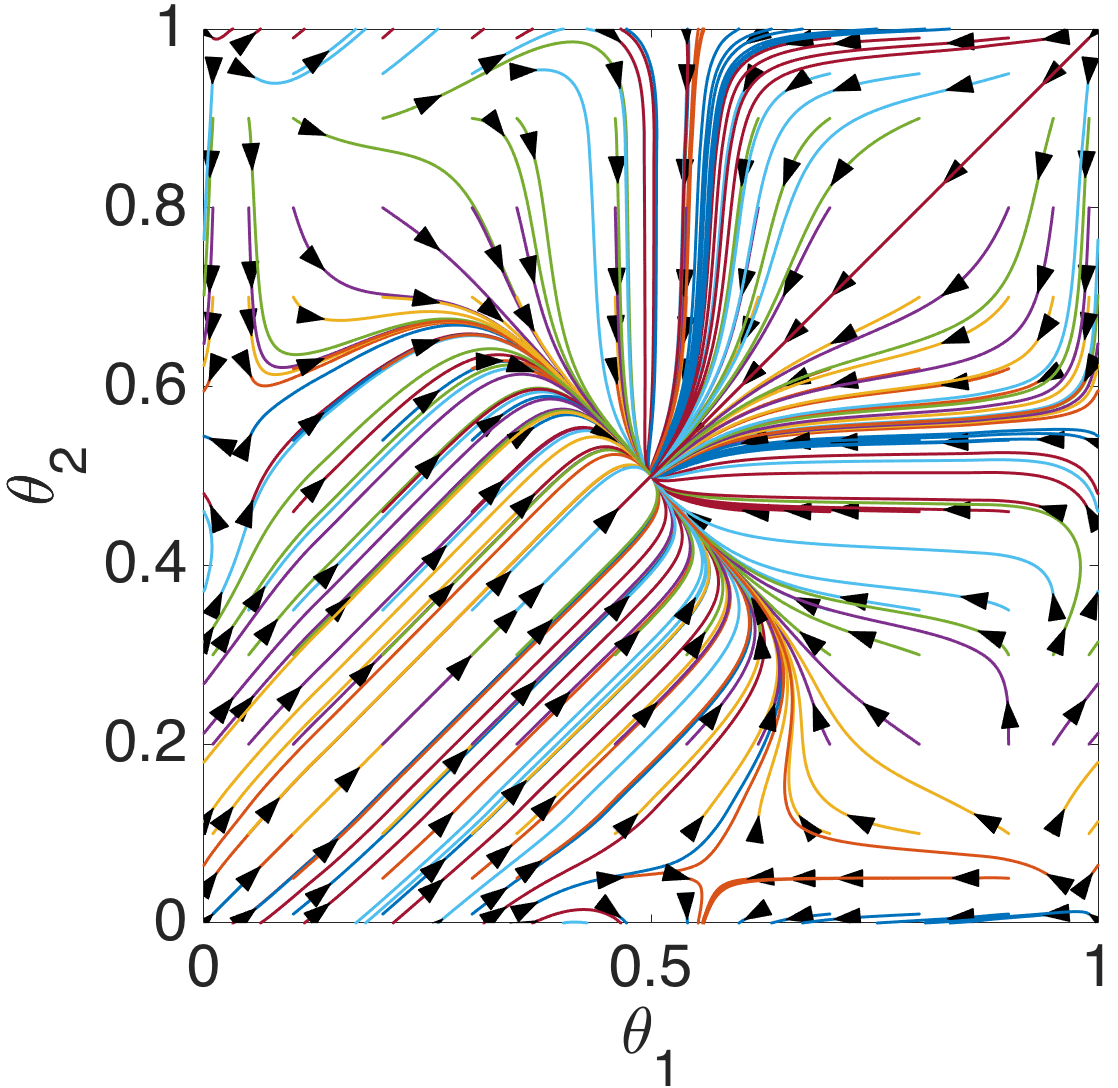}
\includegraphics[scale=.2]{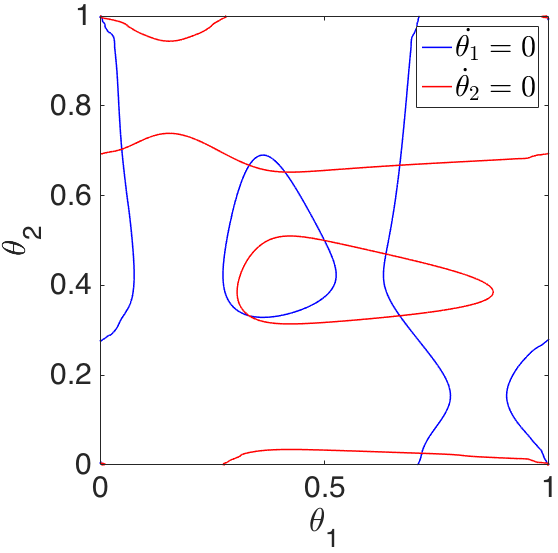}
\includegraphics[scale=.2]{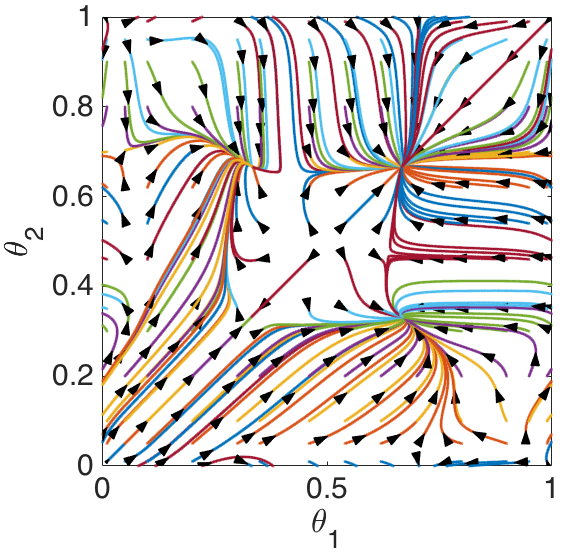}\quad
\includegraphics[scale=.2]{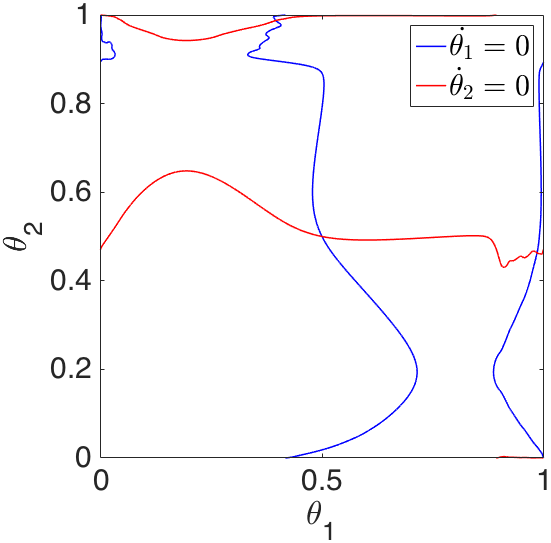}
\includegraphics[scale=.2]{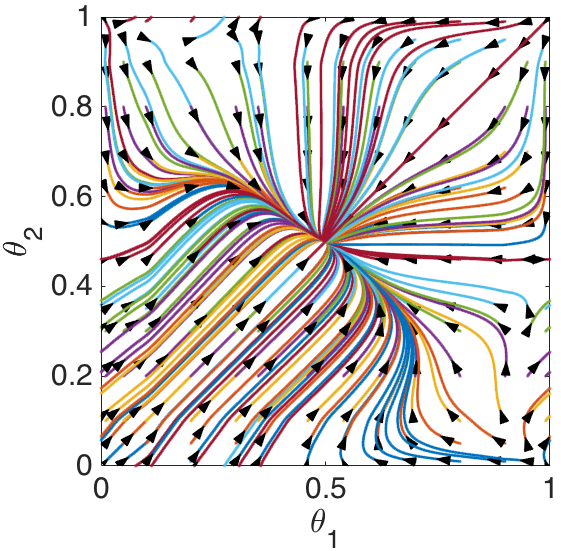}
\end{center}
\caption{Nullclines and  phase planes of Equations~(\ref{eq.osc.simplified})
when $\alpha=1/3$. First row: $\delta=0.0097$ (left) and  $\delta=0.03$ (right).
Second row: $I_{ext}= 35.65$ (left) and $I_{ext}= 37.1$ (right). Reflection 
symmetry is slightly broken, but the invariant line $\theta_1 = \theta_2$ persists.} 
\label{phase differences model_alpha_13}
\end{figure}

Next, we let $\a \neq 1/2$ but keep it close to $1/2$, i.e., we want 
$\alpha_{min} < \alpha <\alpha_{max}$. Specifically, we set 
\[\c_1=\c_2=\c_3= \c_4=1,\; \c_7=2,\; \c_5=\c_6=3,\]
so that $\a=1/3$. 
Figure~\ref{phase differences model_alpha_13} (first row, left to right)
shows the nullclines and the phase planes of  Equations (\ref{eq.osc.simplified}),
for a small $\delta = 0.0097$, and a large $\delta = 0.03$, respectively.
Figure~\ref{phase differences model_alpha_13} (second row, left to right)
shows the nullclines and the phase planes of  Equations (\ref{eq.osc.simplified}),
for a small $I_{ext}= 35.65$ and a large $I_{ext}= 37.1$, respectively. 
As we expect, the qualitative behaviors of the fixed points do not change, but 
reflection symmetry about the diagonal $\theta_1=\theta_2$ is broken, most
easily seen in the nullclines.

Finally, we let $\a \approx 1$, i.e., $\alpha > \alpha_{max}$. 
For $\delta < \delta_*$ (resp. $I_{ext} <I_*$), we expect to have a stable
backward tetrapod gait at $(T/3+\et, 2T/3-\et)$ and an unstable forward
tetrapod gait at $(2T/3-\et, T/3+\et)$. For $\delta > \delta_*$ (resp. $I_{ext}
> I_*$), the tripod gait at $(T/2,T/2)$ becomes stable. In the simulations
shown below we let
 \[\c_1=\c_2=\c_3=0.5,\; \c_4=2,\; \c_7=0.1,\; \c_5=\c_6=2.1, \]
 so that $\a \approx 0.952.$

\begin{figure}[h!]
\begin{center}
\includegraphics[scale=.1]{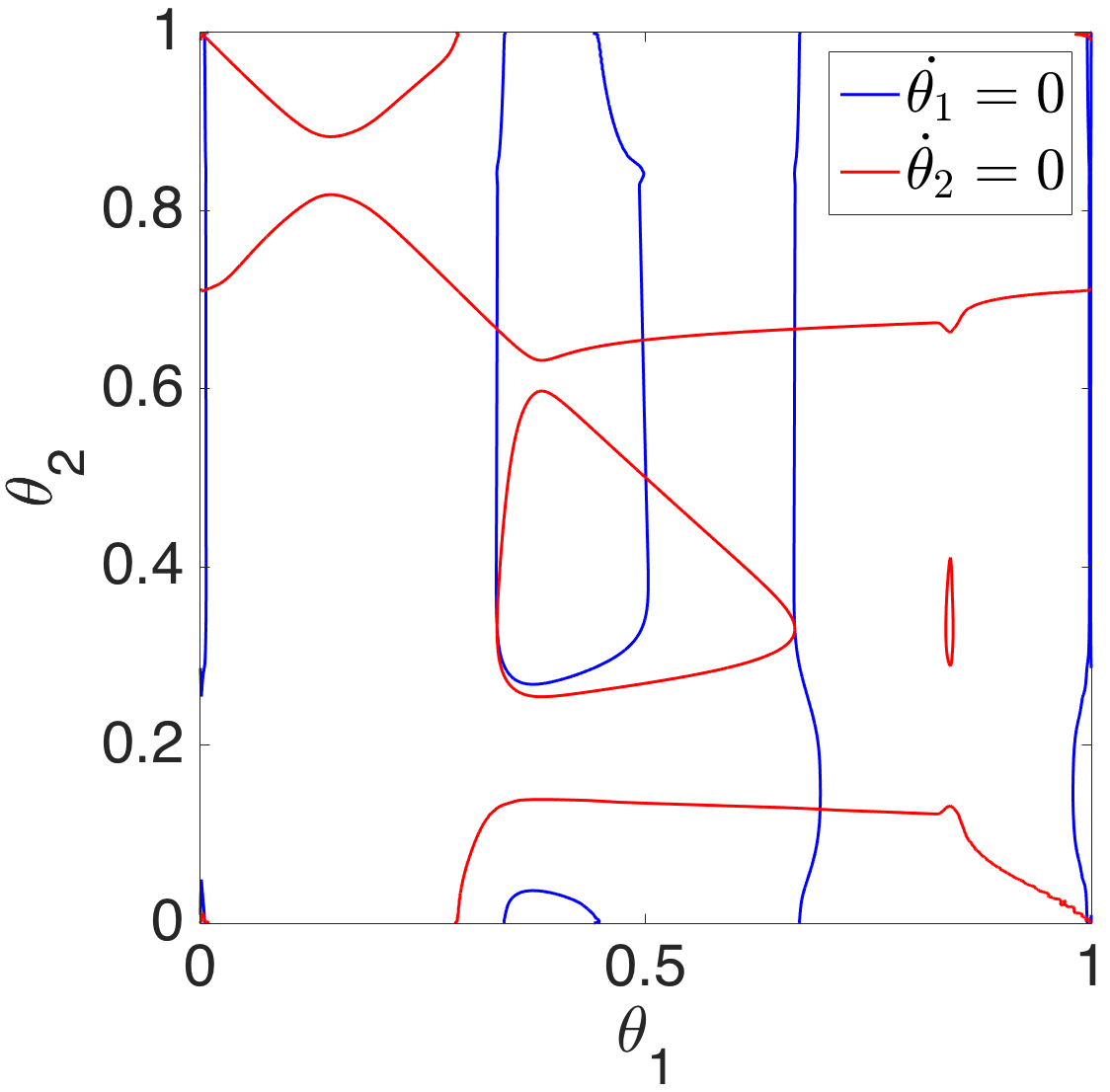}
\includegraphics[scale=.1]{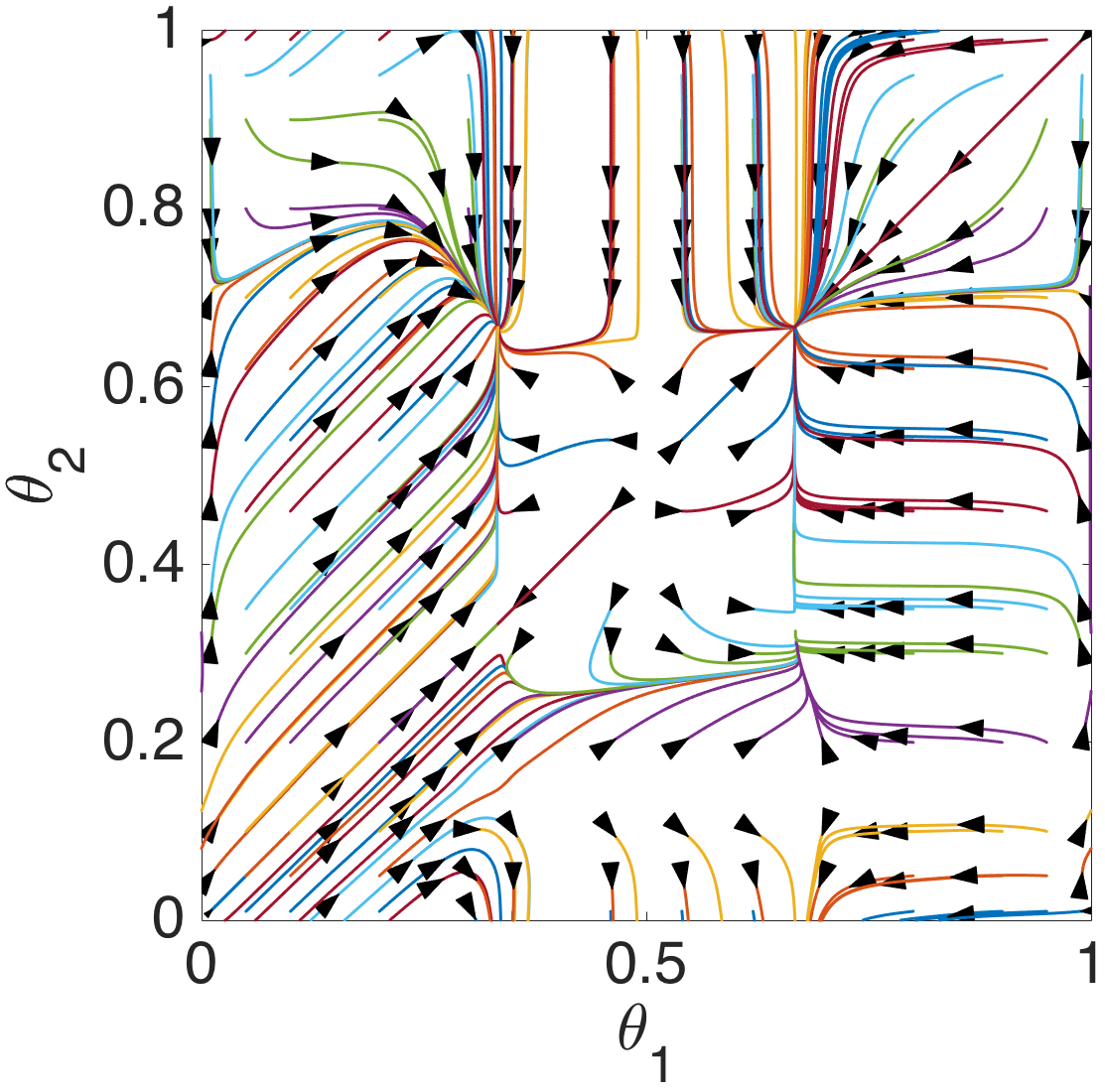}\quad
\includegraphics[scale=.1]{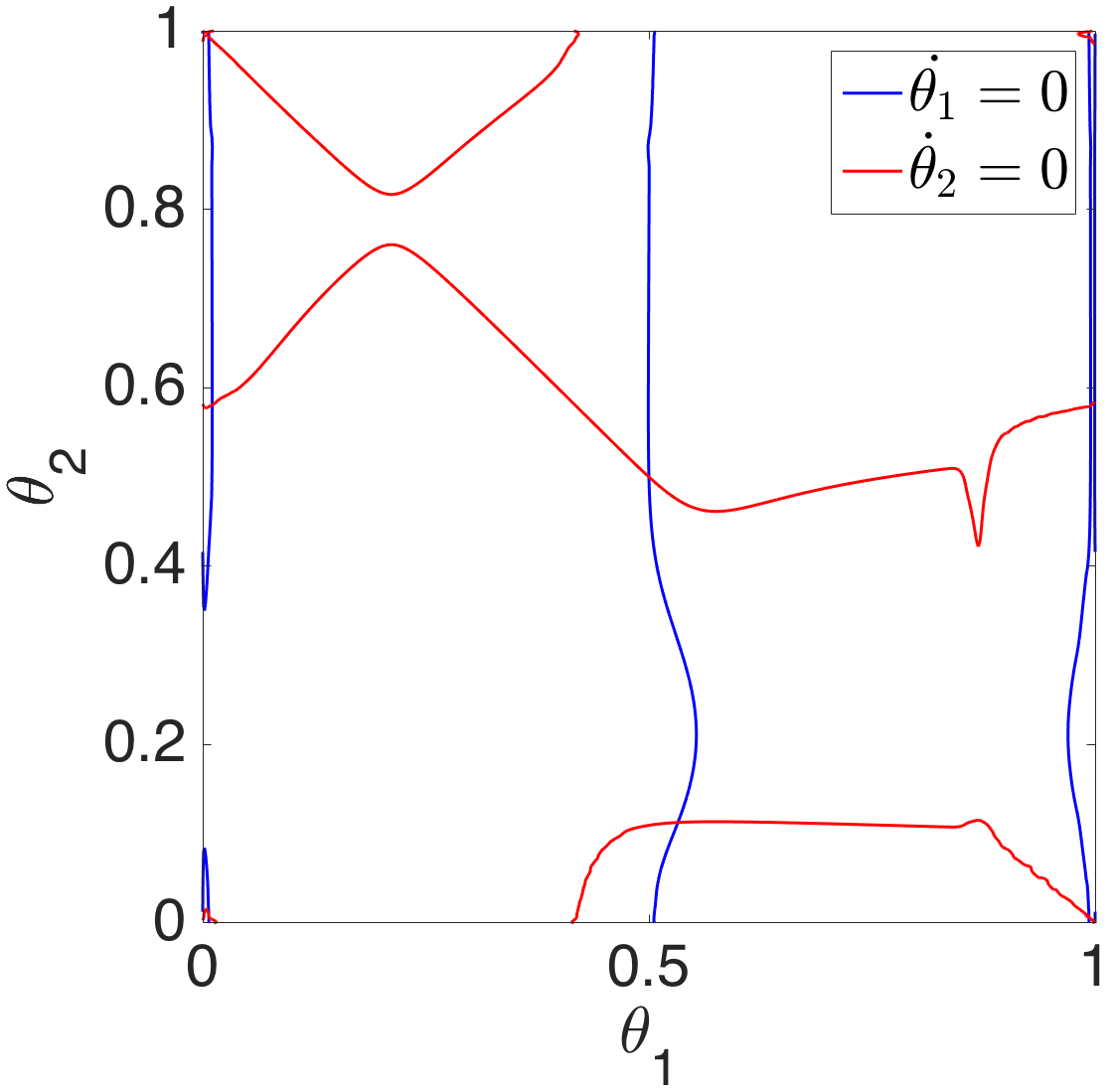}
\includegraphics[scale=.1]{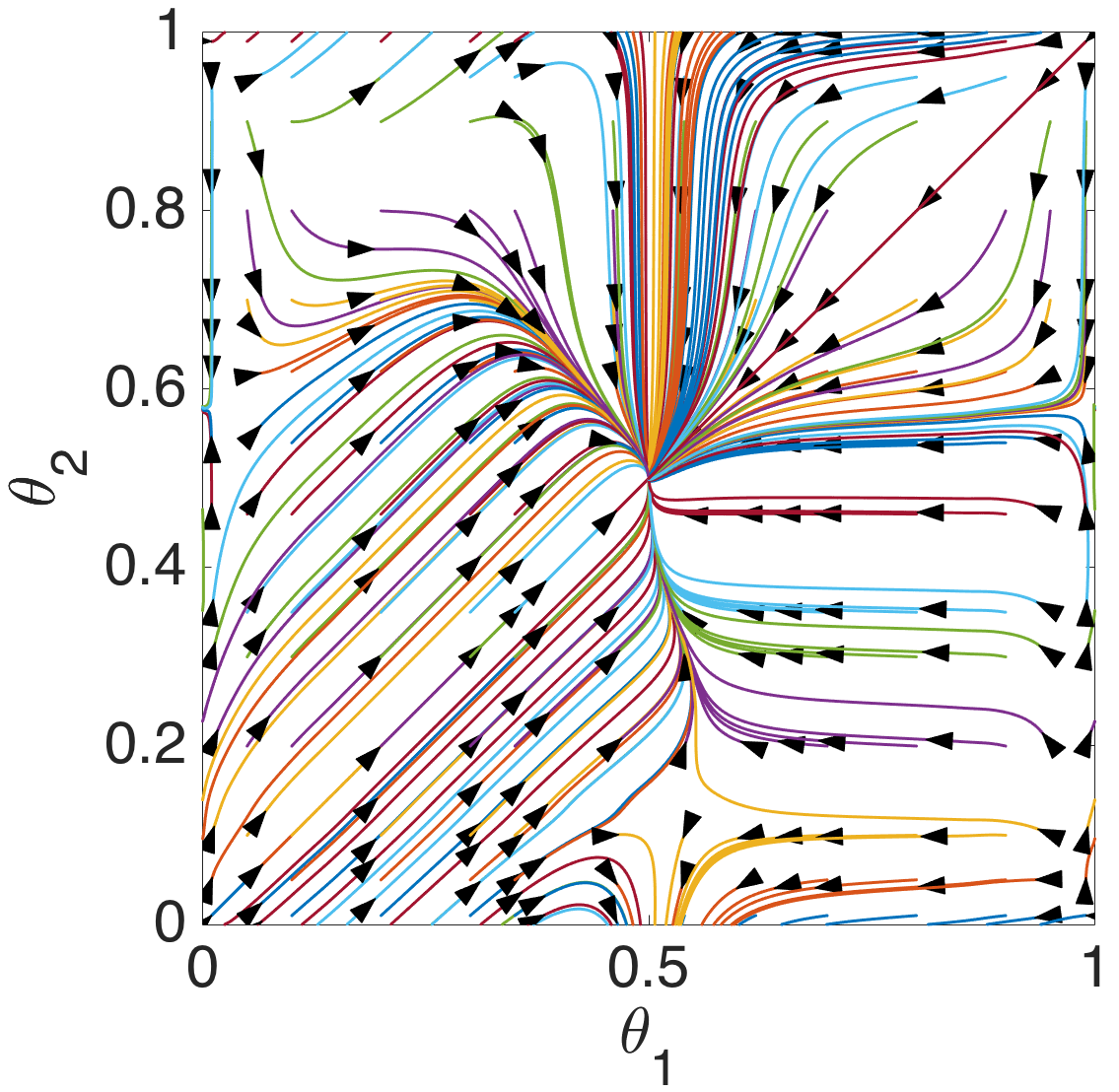}
\includegraphics[scale=.2]{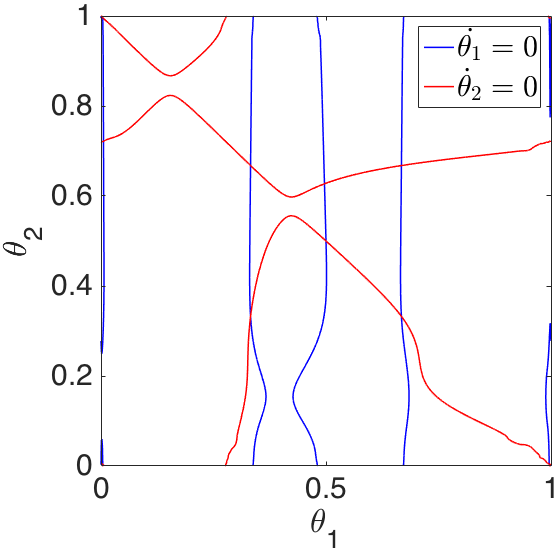}
\includegraphics[scale=.2]{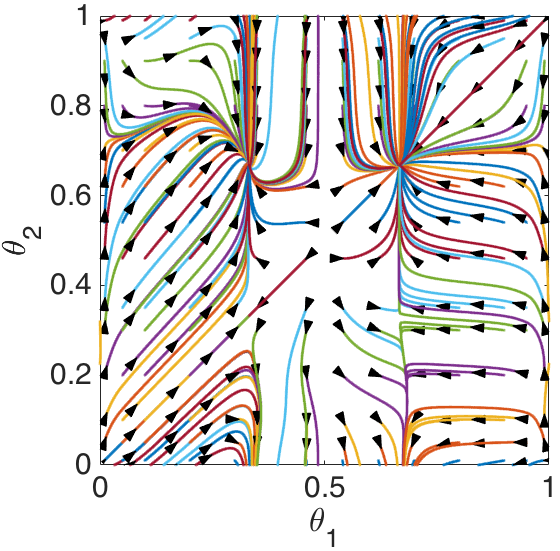}\quad
\includegraphics[scale=.2]{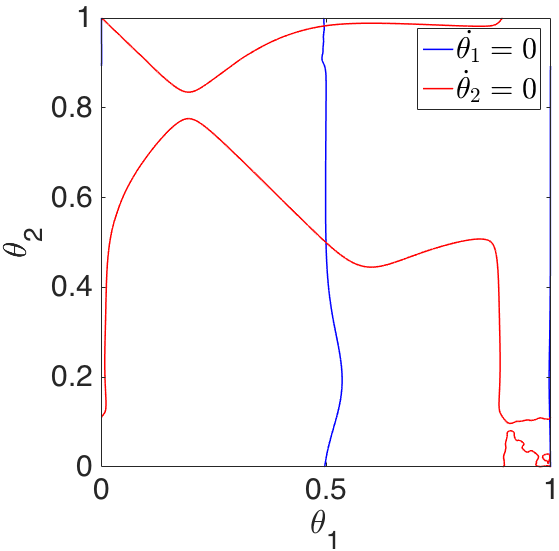}
\includegraphics[scale=.2]{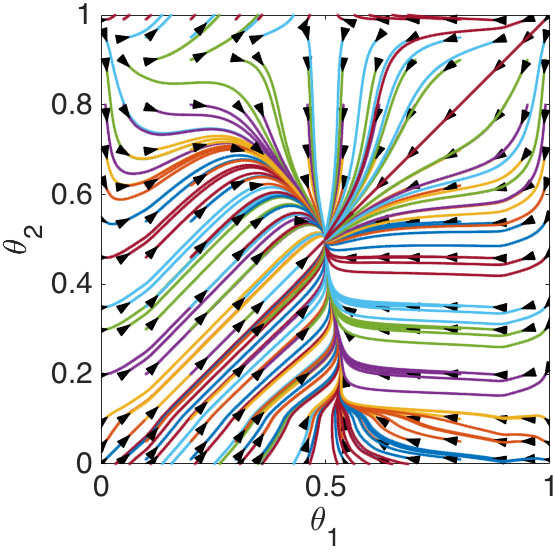}
\end{center}
\caption{Nullclines and  phase planes of Equations~(\ref{eq.osc.simplified}) 
when $\alpha\approx 0.95 > \a_{max}$. First row: $\delta=0.0097$ (left) and
$\delta=0.03$ (right).  Second row: $I_{ext}= 35.65$ (left) and $I_{ext}= 37.1$ (right).
Reflection symmetry is clearly broken.} 
\label{phase differences model_near_one_del}
\end{figure}

Figure~\ref{phase differences model_near_one_del} (first row, left to right)
shows the nullclines and phase planes of  Equation (\ref{eq.osc.simplified})
for a small $\delta = 0.0097$, and a large $\delta = 0.03$, respectively. 
Figure~\ref{phase differences model_near_one_del} (second row, left to right)
shows the nullclines and the phase planes of  Equation (\ref{eq.osc.simplified}),
for a small $I_{ext}= 35.65$ and a large $I_{ext}= 37.1$, respectively. Here
reflection symmetry is broken more obviously.
Similarly, when $\a$ is near zero, i.e., $\alpha < \alpha_{min}$, we expect to
have a stable forward tetrapod gait, and an unstable backward tetrapod gait. 
\ZA{In Figure~\ref{phase differences model_near_0_del}, we let 
$\c_1=\c_2=\c_3=0.5,\; \c_4=0.1,\; \c_7=3,\; \c_5=\c_6=3.1,$
 so that $\a \approx 0.032.$
 As we expect, the forward tetrapod gait remains stable while the backward tetrapod gait 
 becomes a saddle through a transcritical bifurcation. 
 However, a stable fixed point appears (through the same transcritical bifurcation) 
 very close to the backward tetrapod gait. 
\begin{figure}[h!]
\begin{center}
\includegraphics[scale=.1]{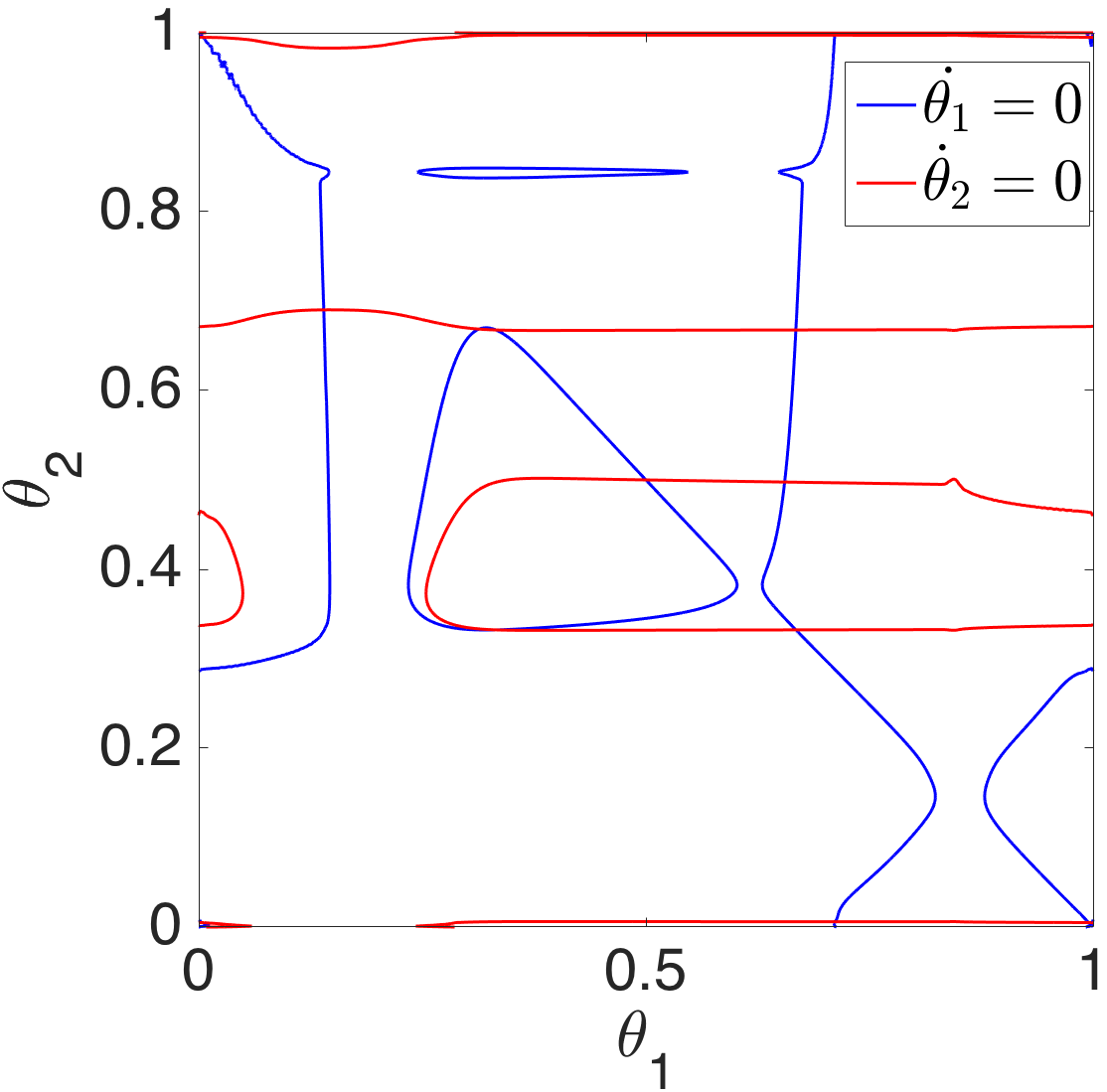}
\includegraphics[scale=.1]{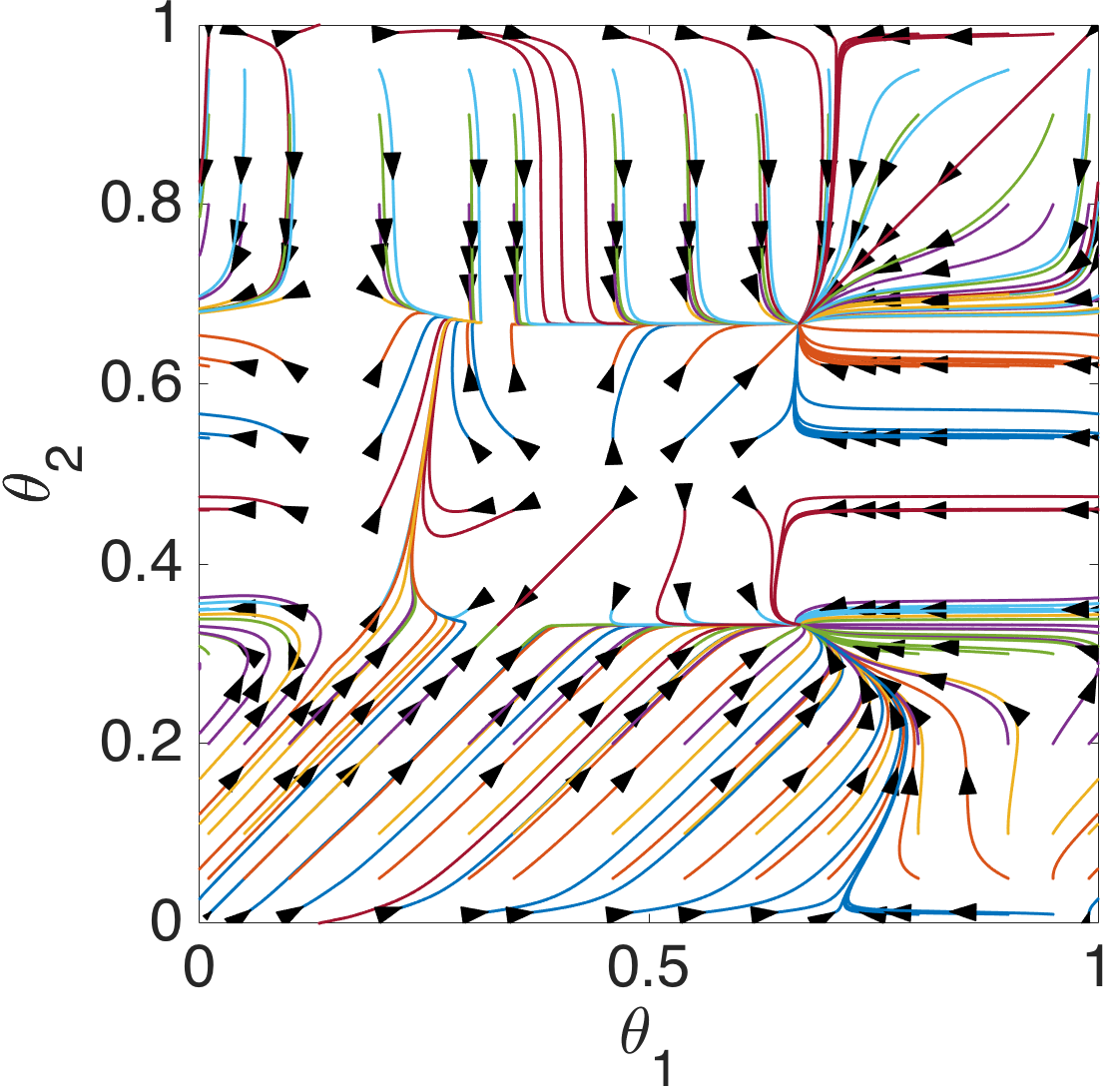}\quad
\includegraphics[scale=.1]{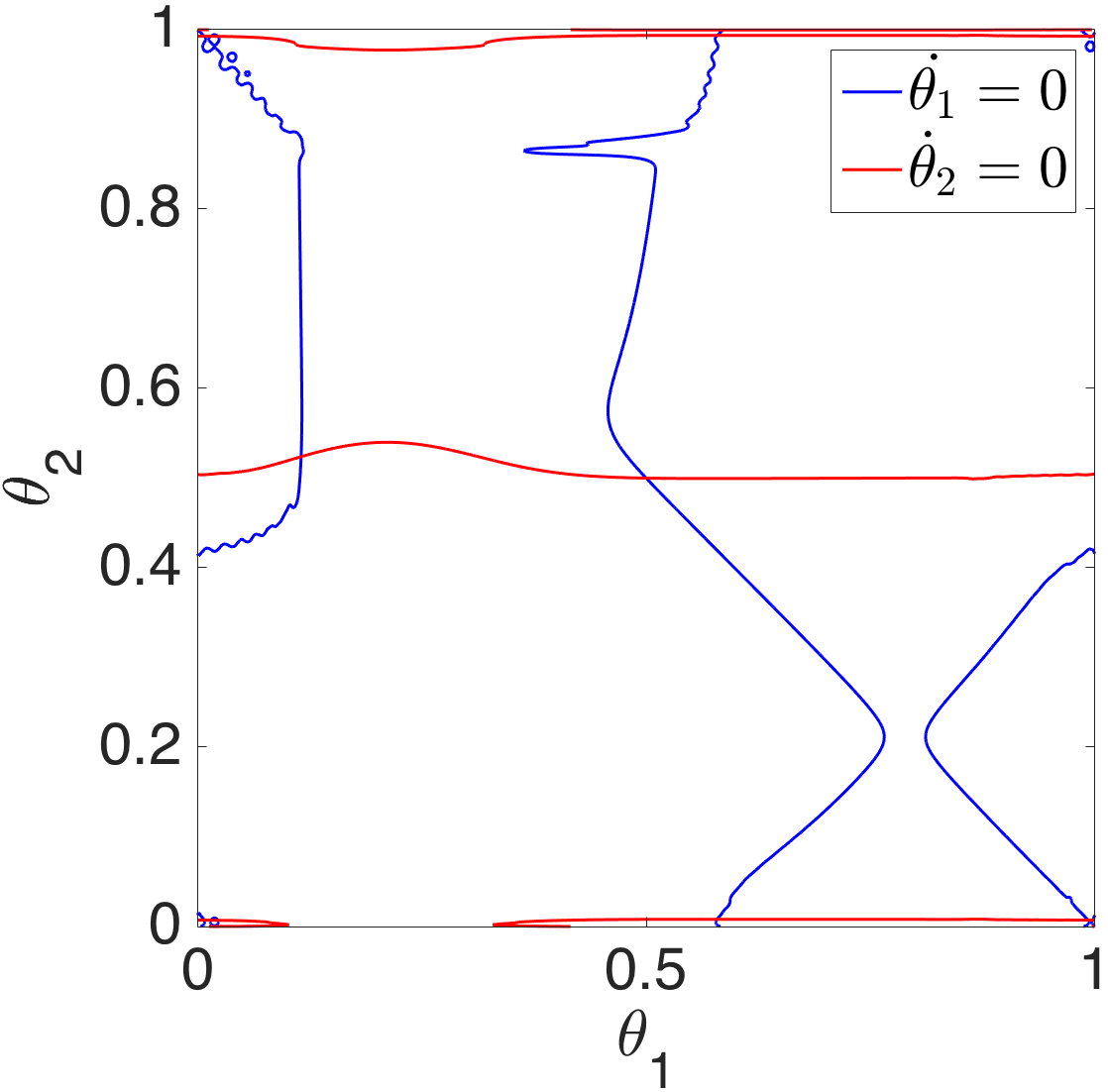}
\includegraphics[scale=.1]{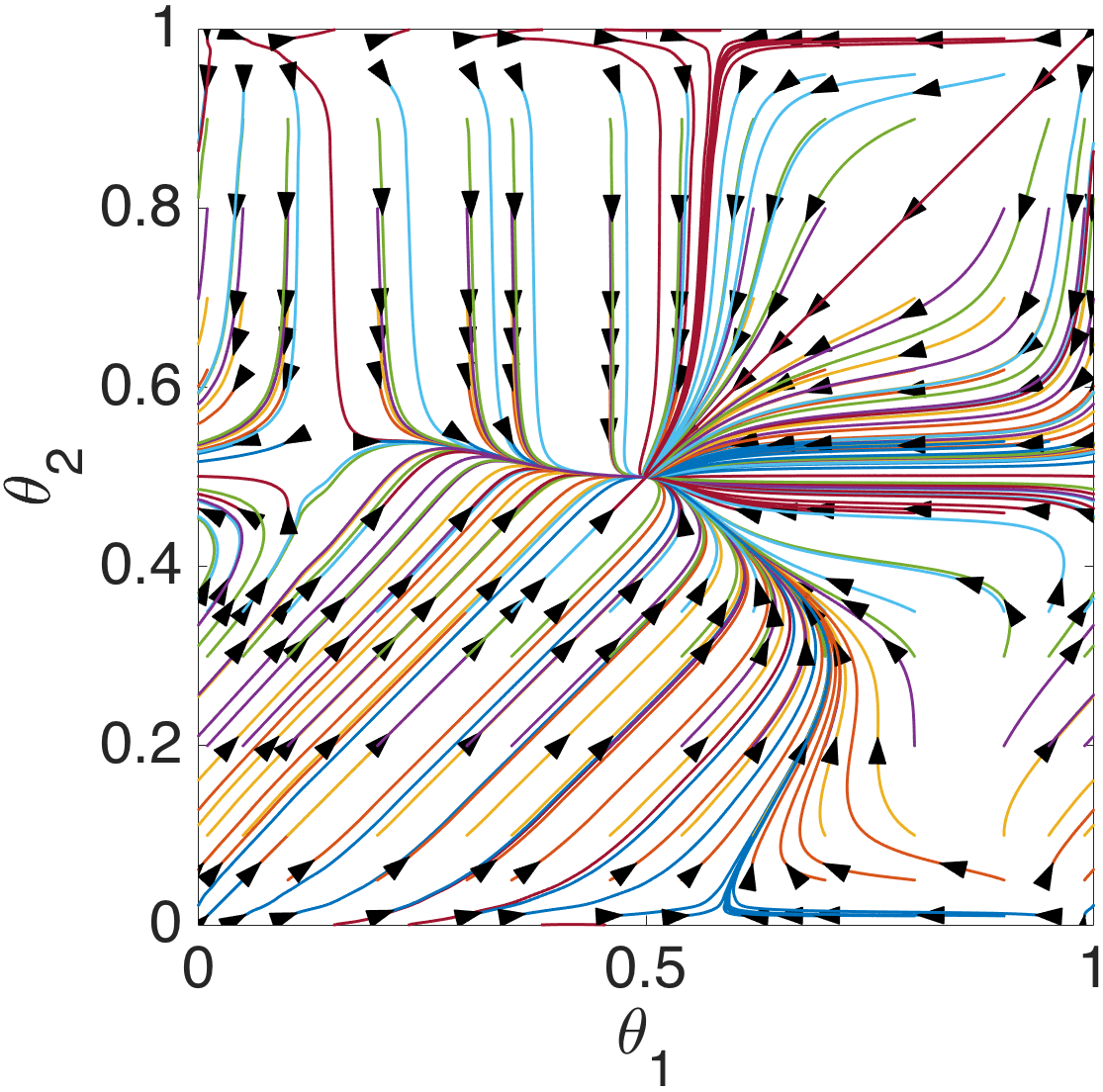}
\includegraphics[scale=.1]{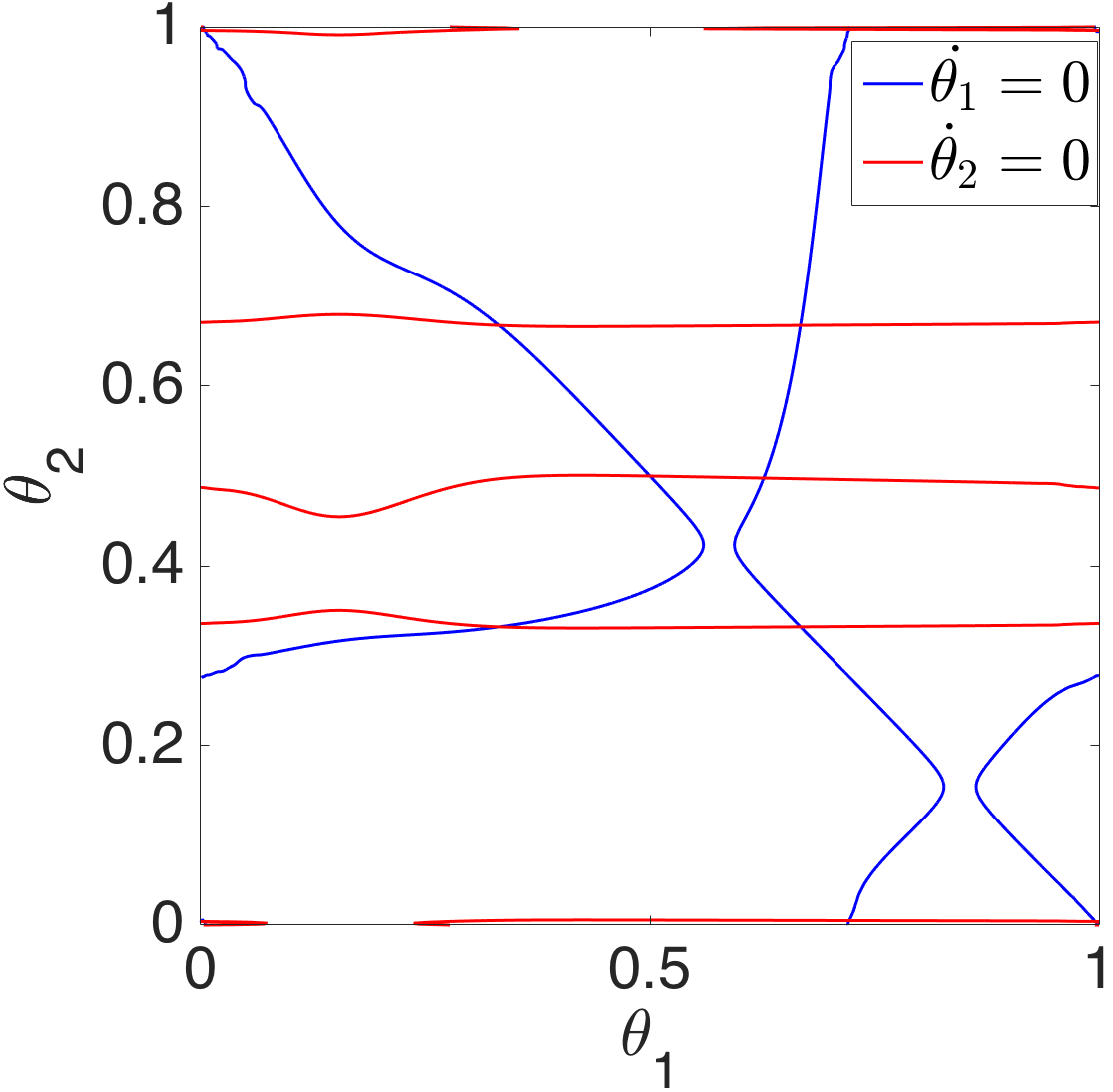}
\includegraphics[scale=.1]{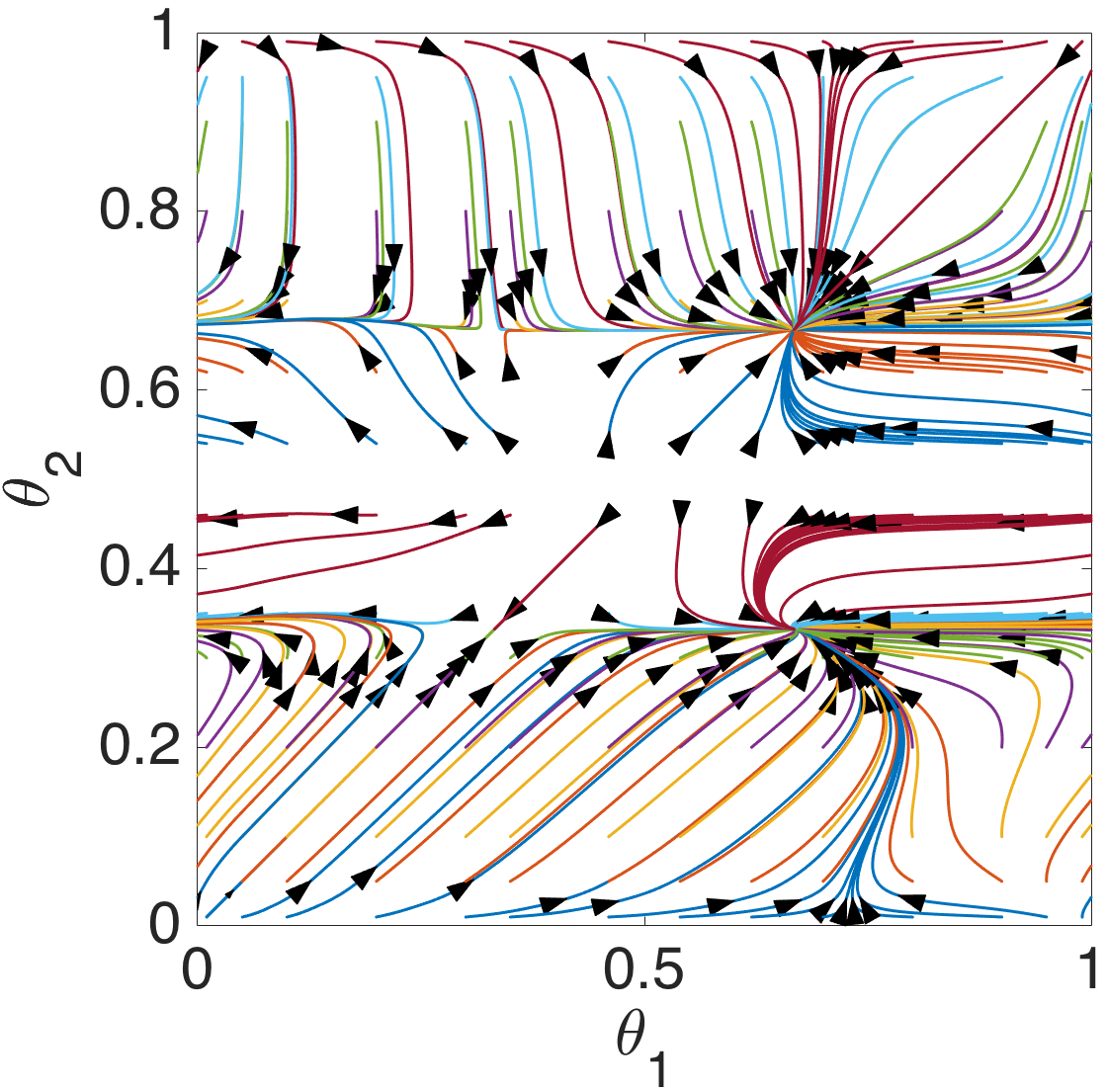}\quad
\includegraphics[scale=.1]{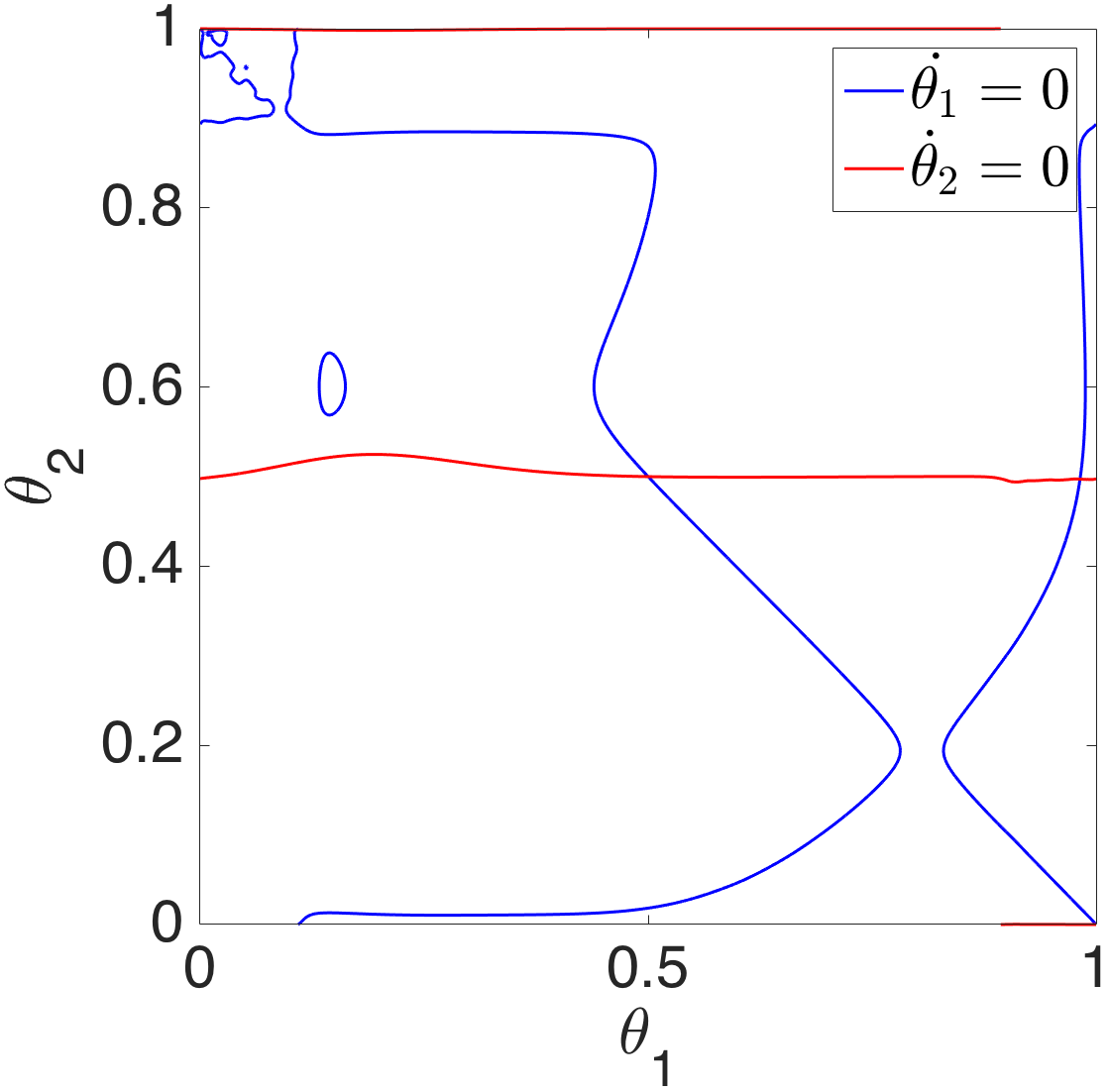}
\includegraphics[scale=.1]{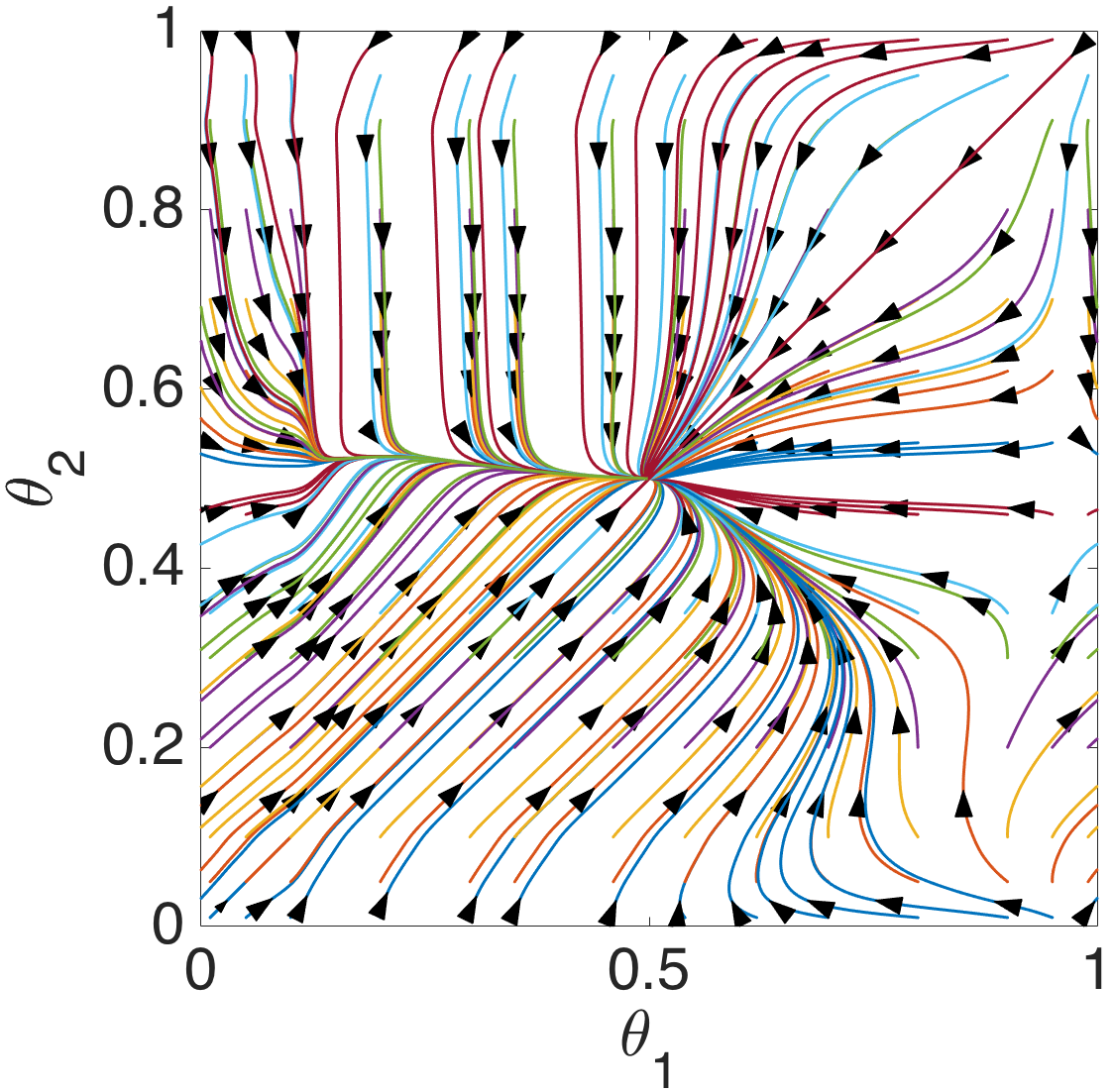}
\end{center}
\caption{Nullclines and  phase planes of Equations~(\ref{eq.osc.simplified}) 
when $\alpha\approx 0.032 <\a_{min}$. First row: $\delta=0.0097$ (left) and
$\delta=0.03$ (right).  Second row: $I_{ext}= 35.65$ (left) and $I_{ext}= 37.1$ (right).
Reflection symmetry is clearly broken.} 
\label{phase differences model_near_0_del}
\end{figure}
}

In this section, using the coupling functions $H_{BN}$ that we computed
numerically and with appropriate conditions on coupling strengths $\c_i$,
we saw that the phase difference equations admit 10 fixed points when
the speed parameter is small
 \ZAB{(Figures~\ref{phase differences model}-\ref{phase differences model_alpha_13} (left))},  
and 4 fixed points when the speed parameter is high 
\ZAB{(Figures~\ref{phase differences model}-\ref{phase differences model_alpha_13} (right))}.
 We saw how 4 fixed points (located on the corners of
a square) together with 2 saddle points (near the corners of the square),
merged to one fixed point (located on the center of the square). We would
like to show that in fact 7 fixed points merge and one fixed point bifurcates.
To this end, in Section \ref{H_app}, we approximate the coupling function 
$H_{BN}$ by a low order Fourier series.

\section{\ZAA{A class of coupling functions producing gait transitions }}
\label{application_general_H}

\ZA{In this section, we first  characterize a class of functions
satisfying Assumption \ref{eta_assumption} and then provide an example based on the bursting neuron model.}

\begin{proposition} \label{H-properties}
Let $H(\theta; \x)$ be $C^2$ and $T$-periodic on $\theta \in [0, T]$
and $C^1$ on $\x \in [\x_1, \x_2]$ and let $G(\theta; \x) = H(\theta; \x)
- H(-\theta; \x)$. Assume that 
\begin{description}
\item[(1)] $\exists \, \bar\x\in [\x_1,\x_2)$ such that $G(T/3; \bar\x) = 0$;
\item[(2)] $\forall \, \x > \bar\x$ and $T/3 \le \theta < T/2, \dis\frac{dG}{d \x}(\theta; \x) < 0$;
\item[(3)] $\exists \, \x_*\in (\bar\x,\x_1]$ such that $\forall \, \theta \in (T/3, T/2)$, and $\bar\x<\x<\x_*$, 
$\dis\frac{d^2 G}{d \theta^2} (\theta; \x) = G''(\theta; \x) < 0$.
\end{description}
Then, $\forall \, \x \in (\bar\x, \x_*), \, G(\theta; \x)$ has a unique solution
in $[T/3, T/2]$ denoted by $\hat{\theta}(\x)$ such that  $\hat{\theta}(\bar\x) = T/3$,
$\hat{\theta}(\x_*) = T/2$ and $\hat{\theta}(\x)$ is a continuous and increasing
function on $[\bar\x, \x_*]$.

{Moreover,  Equations~(\ref{torus:equation}), with the balance equation (\ref{balance2}) and
$\et(\x)= \hat\theta(\x) -T/3$, admit a fixed point at $(T-\hat\theta(\x), \hat\theta(\x)) = (2T/3-\et, T/3+\et)$, 
which corresponds to a forward tetrapod gait at $\x=\bar\x$,
 a tripod gait at $\x=\x_*$, and a transition gait for $\x\in(\bar\x,\x_*)$. }
\end{proposition} 

\bp
Since $H$ is $T$-periodic, $G(T/2; \x) = 0, \forall \, \x$, and because
$G''(\theta; \x) < 0$ for $\theta \in (T/3, T/2)$ and $\forall \, \x < \x_*$,
\be{eqG1}
G(T^-/2; \x) > 0 \;\; \mbox{where $\; T^-/2 < T/2\;$ is sufficiently close to $T/2$}.
\ee
Also, since $G(T/3; \bar\x) = 0$ and $\frac{dG}{d \x}(T/3; \x) < 0$,
\be{eqG2}
G(T/3; \x) < 0, \quad \forall \, \x > \bar\x .
\ee
Equations~(\ref{eqG1}) and (\ref{eqG2}) and Bolzano's intermediate value
theorem imply that for any $\x \in (\bar\x, \x_*)$, $G(\theta; \x)$ has a zero
$\hat{\theta}(\x) \in (T/3, T/2)$. $G''(\theta; \x) < 0$ for $\theta \in (T/3, T/2)$
guarantees uniqueness of $\hat{\theta}(\x)$. 

Next we show that $\hat{\theta}(\x)$ is increasing; i.e.,  for any $x_2 > x_1
\Rightarrow \hat{\theta}(x_2) > \hat{\theta}(x_1)$. Fix $x_1>\bar\x$. 
By definition of $\hat\theta(\x)$, $G(\hat{\theta}(x_1); x_1) = 0$, and because 
$\frac{dG}{d \x}(\hat{\theta}(x_1); \x) < 0, \;\, \forall \, \x > x_1$, 
\be{eqG3}
G(\hat{\theta}(x_1); x_2) < 0 .
\ee
Equations~(\ref{eqG1}) and (\ref{eqG3}) and Bolzano's theorem imply that
$G(\theta; x_2)$ has a zero in $(\hat{\theta}(x_1), T/2)$. Since the zero is
unique, it lies at $\hat{\theta}(x_2)$ and so $\hat{\theta}(x_2) > \hat{\theta}(x_1)$.
Moreover, $\hat{\theta}(x)$ is continuous: $\forall \, \epsilon > 0, \exists \, \delta > 0$
such that 
\be{delta_epsilon}
|x_1 - x_2| < \delta \, \Rightarrow \, |\hat{\theta}(x_1) - 
\hat{\theta}(x_2)| < \epsilon.
\ee 

We now prove inequality (\ref{delta_epsilon}). Fix $x_1 \in (\bar\x, \x_*)$ 
and choose $\bar{\x} < x_1$ small enough such that
\[
0 < b := G\lt(\hat{\theta}(\bar{\x}) + \frac{\epsilon}{2}; x_1\rt) < a := 
G\lt(\hat{\theta}(\bar{\x}) + \frac{\epsilon}{2}; \bar{\x}\rt) .
\]
Now $G(\theta; \x)$ is continuous, decreasing with $\x$, and $\bar{\x} < x_1$,
therefore $G(\hat{\theta}(\bar{\x}); x_1) < 0$. Since  $G(\hat{\theta}(\x) + 
{\epsilon}/{2}; x_1) > 0$ and $G(\hat{\theta}(\bar{\x}); x_1) < 0$
we find that $\hat{\theta}(x_1) \in (\hat{\theta}(\bar{\x}), \hat{\theta}(\bar{\x})
 + \epsilon/2)$, and hence that
\be{eqG4}
 |\hat{\theta}(\bar{\x}) - \hat{\theta}(x_1)| < \frac{\epsilon}{2}.
\ee

Since $G(\theta, \x)$ is continuous on $\x$, for $\epsilon_1 = (a - b)/4 > 0, \, 
\exists \, \delta_1 < (x_1 - \bar{\x})/2$ such that $|x_2 - x_1| < \delta_1$ implies that
\[
\left| G \left(\hat{\theta}(\bar{\x}) + \frac{\epsilon}{2}; x_2 \right) -  
G \left(\hat{\theta}(\bar{\x}) + \frac{\epsilon}{2}; x_1) \right)  \right| < \epsilon_1 ,
\]
and this in turn implies that $G(\hat{\theta}(\bar{\x}) + {\epsilon}/{2}; x_2) > 0$.
Since $\delta_1 < (x_1 - \bar{\x})/2$, $x_2 > \bar{\x}$ and so 
$G(\hat{\theta}(\bar{\x}); x_2) < 0$.
Therefore if $\hat{\theta}(x_2) \in (\hat{\theta}(\bar{\x}), \hat{\theta}(\bar{\x}) + \epsilon/2)$ 
then 
\be{eqG5}
|\hat{\theta}(x_2) - \hat{\theta}(\bar{\x})| < \frac{\epsilon}{2} .
\ee

Finally, Equations~(\ref{eqG4}) and (\ref{eqG5}) imply that for $\delta = \delta_1$,
if $|x_1 - x_2| < \delta$ then $|\hat{\theta}(x_1) - \hat{\theta}(x_2)| < \epsilon$.
\ep

\ZA{As an example, we next show that  $H_{app}(\theta; \x)$, an explicit function which  approximates $H_{BN}(\theta; \x)$, 
satisfies assumptions (1), (2) and (3) in Proposition~\ref{H-properties}.}

\subsection{Example of an explicit coupling function}\label{H_app}

In this section, we approximate  $H_{app}$ by its Fourier series and derive an explicit function $H_{app}$ as follows.   
To derive $H_{app}$, we first computed the coefficients of the Fourier series of $H_{BN}$, and then, using {\tt polyfit} in Matlab, fitted an
appropriate quadratic function for each coefficient, obtaining
\\
\begin{wrapfigure}[7]{r}{0.4\textwidth}
\centering \includegraphics[width=.4\textwidth]{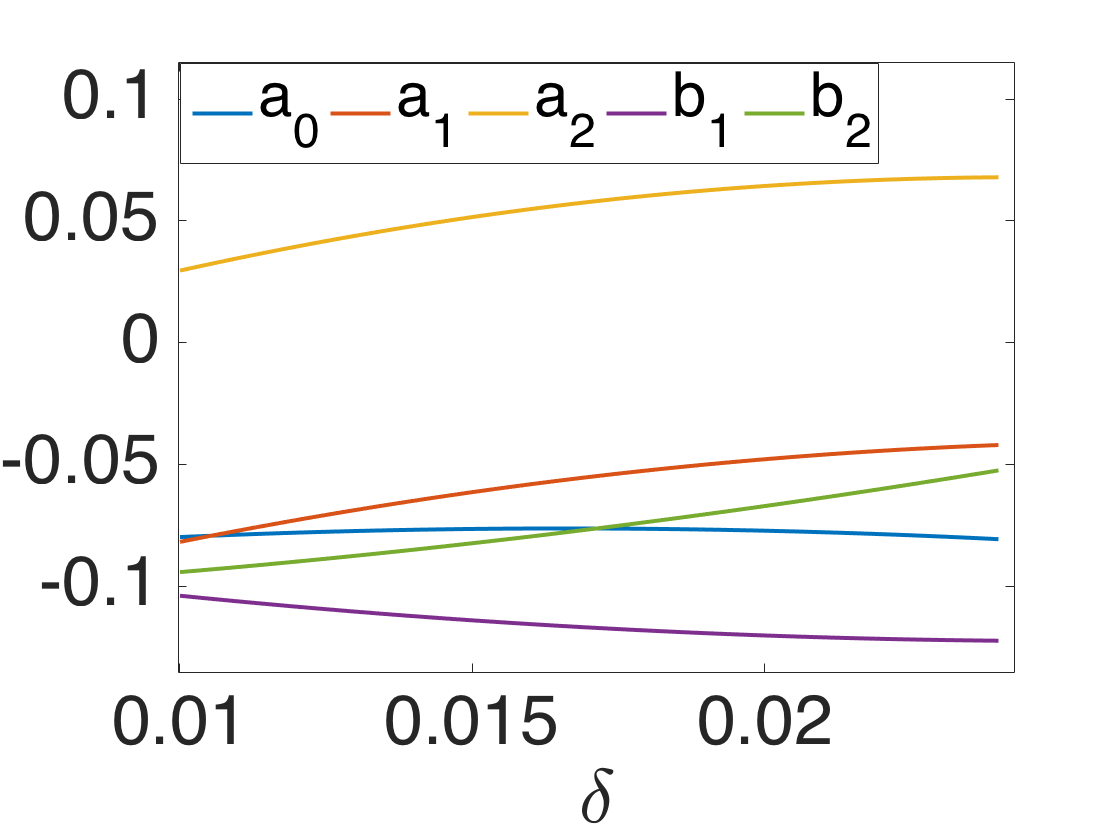}
\caption{Fourier coefficients of $H_{app}$ }
\label{fourier_coef_delta}
\vspace{-100pt}  
\end{wrapfigure}
\begin{subequations}\label{coefficient:Happ}
\begin{align}
a_0(\delta) &= -80.8384\delta^2+2.6862\delta-0.0986, \\
a_1(\delta) &=-137.9839\delta^2+7.5308\delta-0.1433, \\
b_1(\delta) &=77.9417 \delta^2-3.9694 \delta -0.0720, \\
a_2(\delta) &=-184.2374\delta^2+8.9996 \delta -0.0420, \\ 
b_2(\delta) &=68.0350 \delta^2+0.6692 \delta -0.1077 . 
\end{align}
\end{subequations}
By definition, $H_{{app}}(\theta;\delta)$ on $[0\; 1]\times[0.008\;  0.024]$ is
\beqn
H_{{app}}(\theta;\delta)\;:=\;\dis\sum_{k=0}^2 a_k(\delta)\cos(2\pi k\theta) + \dis\sum_{k=1}^2 b_k(\delta)\sin(2\pi k\theta).
\eeqn

In Figure~\ref{H_Happ_delta}, we compare the approximate coupling
function $H_{app}$ with $H_{BN}$ for the values of $\delta$ at the endpoints
of the interval of interest.
\begin{figure}[h!]
\begin{center}
\includegraphics[scale=.22]{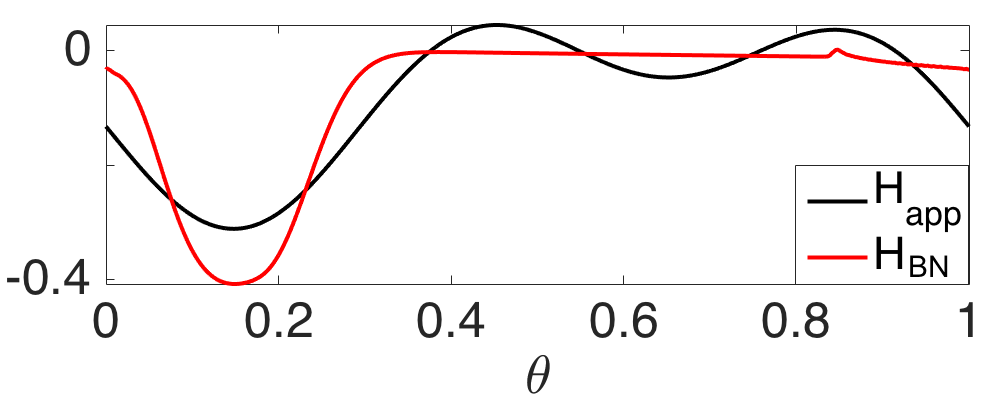}\quad
\includegraphics[scale=.22]{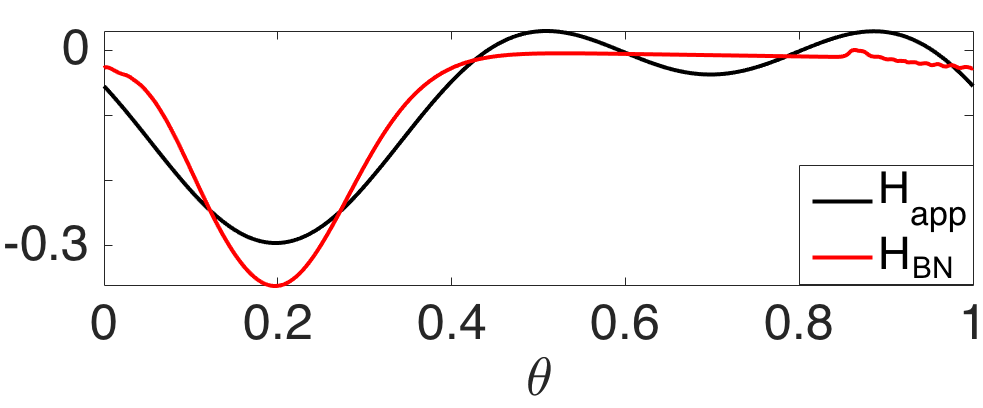}
\end{center}
\caption{The coupling function $H_{BN}$ and its approximation $H_{app}$ are 
shown for $\delta= 0.01$ (left) and $\delta= 0.024$ (right). }
\label{H_Happ_delta}
\end{figure}

We next verify that 
\be{Gapp}
G_{app} (\theta; \delta):= H_{app} (\theta; \delta)- H_{app} (-\theta; \delta)= 2b_1(\delta)\sin(2\pi\theta) + 2b_2(\delta)\sin(4\pi\theta),
\ee
 satisfies conditions  (1), (2), and (3) of Proposition~\ref{H-properties}. 

\begin{description}[leftmargin=*]
\item [Conditions of Proposition~\ref{H-properties}] 

Figure~\ref{Gapp_vs_delta} shows the graphs of $G_{app}$ for different values of $\delta$. 
Since we are only interested in the interval $[1/3,1/2]$, we only show the $G_{app}$'s in this interval.
\begin{figure}[h!]
\begin{center}
\includegraphics[scale=.2]{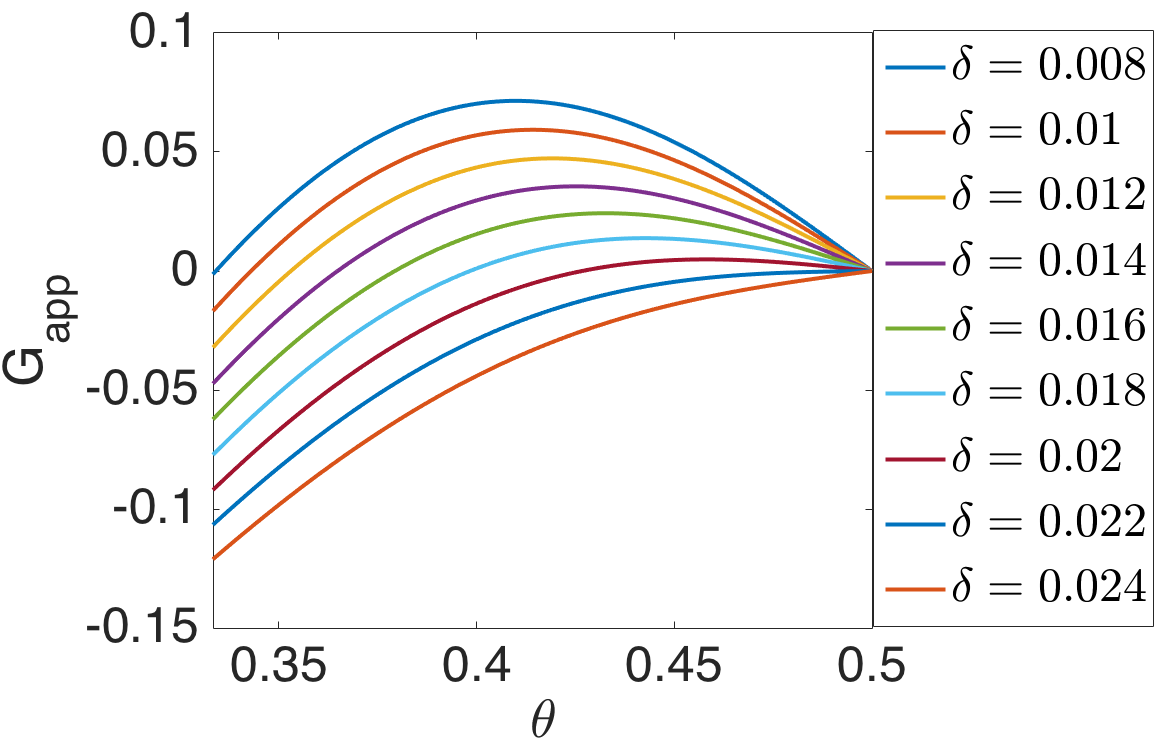}
\end{center}
\caption{The graphs of $G_{app}$ on $[1/3,1/2]$ and  for different values of $\delta$ are shown.}
\label{Gapp_vs_delta}
\end{figure}
As Figure~\ref{Gapp_vs_delta} shows, 
 for $\bar\delta=0.008$, $G_{app}$ equals to zero at $1/3$: $G_{app}(1/3;0.008)=0$. 
In the interval $[1/3,1/2]$,
as $\delta$ increases, at each point $\theta$, $G_{app}$ decreases: $dG_{app}/d\delta<0$. 
For $\delta<\delta_*=0.0218$, the graph of $G_{app}$ is concave down: $G''_{app}<0$. 
One can compute the zero of $G_{app}(1/3;\delta)$, $dG_{app}/d\delta$, and $G''_{app}$ 
explicitly and verify the above conditions. 
 
  \textbf{Computing $\et$.} 
We show that 
\be{eta_explicit_Happ}
\et(\delta) = \frac{1}{2\pi}\arccos\lt(\frac{-b_1(\delta)}{2b_2(\delta)}\rt) - \frac{1}{3},
\ee
is a unique non-constant and non-decreasing  solution of
$H_{app}\lt({2}/{3} - \et; \delta\rt) = H_{app}\lt({1}/{3}+\et, \delta\rt)$. 
Note that $\et$ is defined only where $\abs{{-b_1(\delta)}/{2b_2(\delta)}}
\leq 1$. Figure~\ref{eta_Happ} (left) shows that $\exists\delta_*\approx
0.0218$ such that for $\delta\in [0.008, \delta_*]$,
$-1\leq {-b_1(\delta)}/{2b_2(\delta)} <0$. Therefore, we let $[0.008,
\delta_*]$  be the domain of $\et$, where $\delta_*$ satisfies 
\be{delta_star_Happ}
\frac{-b_1(\delta_*)}{2b_2(\delta_*)} =- 1.
\ee
Figure~\ref{eta_Happ} (right) shows the graph of $\et$. Note that 
the range of $\et$ is approximately $[0, 1/6]$, as desired.  
\begin{figure}[h!]
\begin{center}
\includegraphics[width=0.35\textwidth]{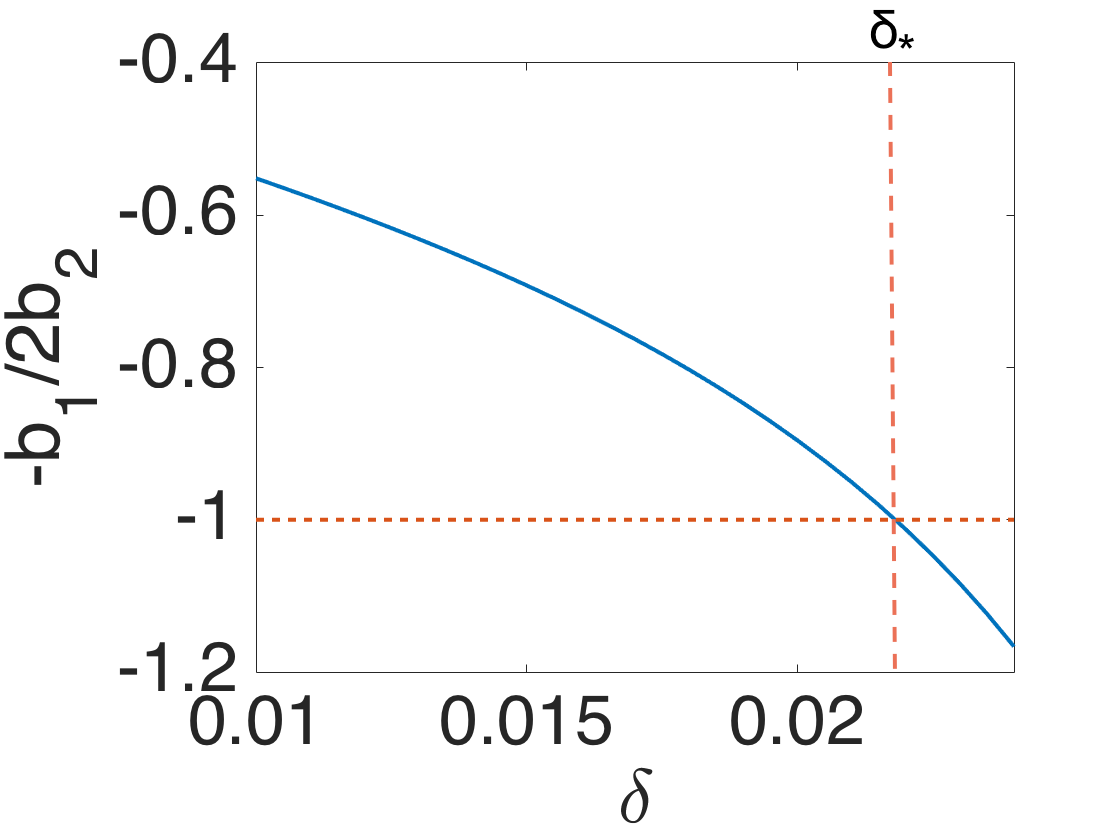}\quad
\includegraphics[width=0.35\textwidth]{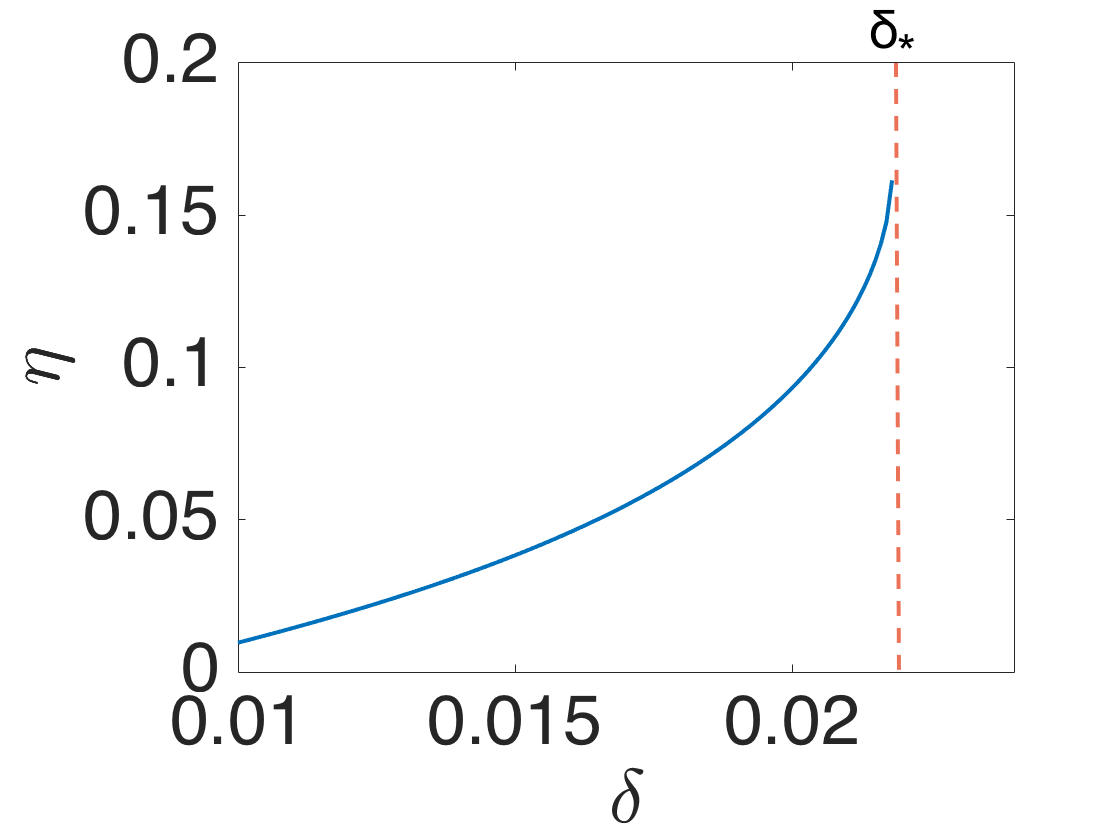}
\caption{(Left) the graph of ${-b_1(\delta)}/{2b_2(\delta)}$ which 
determines the domain of $\et$ defined in 
Equations~(\ref{eta_explicit_Happ}); (right) the graph of $\et$. }
\label{eta_Happ}
\end{center}
\end{figure}
A simple calculation shows that because $\cos(2\pi-x) = \cos x$,  
\[
\cos \lt(2\pi k\lt({2}/{3}- \et\rt)\rt) =\cos \lt(2\pi k-\lt(2\pi k\lt({2}/{3} - \et\rt)\rt)\rt) = \cos \lt(2\pi k\lt({1}/{3} +\et\rt)\rt),
\]
and therefore the cosine terms in the Fourier series cancel, resulting in 
\be{finding:eta:for:Happ}
\bal
  H_{app}\lt(\frac{2}{3} - \et; \delta\rt) = H_{app}\lt(\frac{1}{3}+\et; \delta\rt) &\Longleftrightarrow\\
  \sum_{k=1}^2 b_k(\delta)\sin \lt(2\pi k\lt({2}/{3}- \et\rt)\rt)  
  &=\sum_{k=1}^2 b_k(\delta)\sin \lt(2\pi k\lt({1}/{3}+ \et\rt)\rt).
\eal
\ee
Using the fact that $\sin(2\pi-x) = -\sin x$, we have 
\[\sin \lt(2\pi k\lt({2}/{3}- \et\rt)\rt) = - \sin \lt(2\pi k-\lt(2\pi k\lt({2}/{3}- \et\rt)\rt)\rt)= -\sin \lt(2\pi k\lt({1}/{3}+ \et\rt)\rt),  \]
and so the right hand equality of Equation (\ref{finding:eta:for:Happ}) 
can be written as follows. 
\begin{subequations}\label{OMEGA2}
\begin{align}
&\quad\quad \sum_{k=1}^2 b_k(\delta)\sin \lt(2\pi k\lt({2}/{3}- \et\rt)\rt)  =\sum_{k=1}^2 b_k(\delta)\sin \lt(2\pi k\lt({1}/{3}+ \et\rt)\rt) \\
\Longleftrightarrow& \quad - \sum_{k=1}^2 b_k(\delta)\sin \lt(2\pi k\lt({1}/{3}+ \et\rt)\rt)  =   \sum_{k=1}^2 b_k(\delta)\sin \lt(2\pi k\lt({1}/{3}+ \et\rt)\rt)\\
\Longleftrightarrow& \quad b_1(\delta)\sin \lt(2\pi \lt({1}/{3} +\et\rt)\rt) + b_2(\delta)\sin \lt(4\pi \lt({1}/{3}+ \et\rt)\rt) = 0. 
\end{align}
\end{subequations}
Now using the double-angle identity,  $\sin (2x) = 2 \sin x \cos x$, we get
\[H_{app}\lt(\frac{2}{3} - \et; \delta\rt) = H_{app}\lt(\frac{1}{3}+\et;  \delta\rt) 
 \Longleftrightarrow \sin \lt(2\pi(1/3+\et)\rt) \lt[ b_1(\delta) +2b_2(\delta) \cos \lt(2\pi(1/3+\et)\rt) \rt] = 0. \]
Since we are looking for a non-constant and non-decreasing solution, we solve 
\[b_1(\delta) +2b_2(\delta) \cos \lt(2\pi(1/3+\et)\rt)=0,\]
 for $\et$, which gives $\et$ as in Equation~(\ref{eta_explicit_Happ}). 
 
 Therefore, by Proposition~\ref{H-properties}, 
 Equations~(\ref{torus:equation}), with $H=H_{app}$ and the balance equation (\ref{balance2}) 
 admit a fixed point at $(T-\hat\theta(\x), \hat\theta(\x)) = (2T/3-\et, T/3+\et)$, 
which corresponds to a forward tetrapod gait at $\x=\bar\x$,
 a tripod gait $\x=\x_*$, and a transition gait for $\x\in(\bar\x,\x_*)$.

In what follows we assume Equations~(\ref{eq.osc.simplified}) with $H=H_{app}$.
 We compute $H_{app}'$, and show that it satisfies conditions of Propositions~\ref{special_coupling}, 
 \ref{special_coupling_corollary},  and \ref{fixedpoints_zero_T2}. 
   \be{der_Happ}
   H_{app}' = - 2\pi \dis\sum_{k=1}^2 k\;a_k(\delta)\sin(2\pi k\theta) + 2\pi \dis\sum_{k=1}^2 k\; b_k(\delta)\cos(2\pi k\theta).
   \ee

\item  [Conditions of Proposition~\ref{special_coupling}] First, we verify the stability of $(2/3-\et, 1/3+\et)$. 
\[
H_{app}'\lt( {1}/{3} +\eta, \delta\rt) \pm H_{app}'\lt( {2}/{3} -\eta, \delta\rt) >0, \quad\forall\delta\in [0.01, \delta_*].
\]
Substituting Equation~(\ref{eta_explicit_Happ}) in  the derivative of $H_{app}$,  
Equation~(\ref{der_Happ}), and using trigonometrical identities yields
\[
H_{app}'\lt( \frac{1}{3} +\et; \delta\rt) + H_{app}'\lt( \frac{2}{3} -\et;  \delta\rt) =
- 2\pi \;\dis\frac{4b_2^2(\delta) - b_1^2(\delta)}{b_2(\delta)}>0,
\]
and 
\[
H_{app}'\lt( \frac{1}{3} -\et; \delta\rt) + H_{app}'\lt( \frac{2}{3} -\et;  \delta\rt) =
 \pi\lt(a_1(\delta)-2\frac{a_2(\delta)b_1(\delta)}{b_2(\delta)}\rt) \;\dis\frac{\sqrt{4b_2^2(\delta) - b_1^2(\delta)}}{b_2(\delta)}>0,
\]
which are positive because $4b_2^2(\delta) - b_1^2(\delta)>0$  on  $[0.008, \delta_*]$, $a_1(\delta), b_1(\delta), b_2(\delta)<0$, and $a_2(\delta)>0$
(see Figures~\ref{eta_Happ}(left) and \ref{fourier_coef_delta}). Therefore, by  Proposition~\ref{special_coupling_corollary}, 
for $\a<\a_{max}$ (resp. $\a>\a_{max}$), 
$(2/3-\et, 1/3+\et)$ is a sink (resp. saddle point).

\item  [Conditions of Proposition~\ref{special_coupling_corollary}]  Next, we verify the stability of $(1/3+\et, 1/3+\et)$, 
$(1/3+\et, 2/3-\et)$, and $(2/3-\et, 2/3-\et)$. 

 $H_{app}'\lt( {2}/{3} -\eta, \delta\rt)$ changes sign, on the domain of $\et$, i.e., $[0.01,\delta_*]$. 
Substituting Equation~(\ref{eta_explicit_Happ}) in  the derivative of $H_{app}$,
Equation~(\ref{der_Happ}), and using trigonometrical identities yields
\[
H_{app}'\lt( \frac{2}{3} -\et; \delta\rt) =
 -\dis\frac{\pi}{b_2(\delta)} \sqrt{4b_2^2(\delta) - b_1^2(\delta)}\lt( a_1(\delta) -2 \frac{a_2(\delta)b_1(\delta)}{b_2(\delta)} +\sqrt{4b_2^2(\delta) - b_1^2(\delta)}\rt).
\]
Figure~\ref{der_Happ_23_13_12_0} (left) shows that $H_{app}'\lt( {2}/{3} -\et; \delta\rt)$ 
changes sign from positive to negative on $\delta\in[0.01, \delta_*]$, 
{at some $\delta$ near $0.01$.} 
We will see that through a transcritical bifurcation, $(1/3+\et, 1/3+\et)$ becomes a saddle point from a sink. 
The reason is that by Proposition \ref{special_coupling_corollary}, as $H'_{app} (2/3-\et; \delta)$ changes sign, 
one of the eigenvalues of $(1/3+\et,1/3+\et)$ becomes positive while the other one remains negative. 
For $\a>\a_{min}$, the fixed points $(1/3+\et, 2/3-\et)$ and  $(2/3-\et, 2/3-\et)$ are always sinks. 
\begin{figure}[h!]
\begin{center}
        \includegraphics[width=0.36\textwidth]{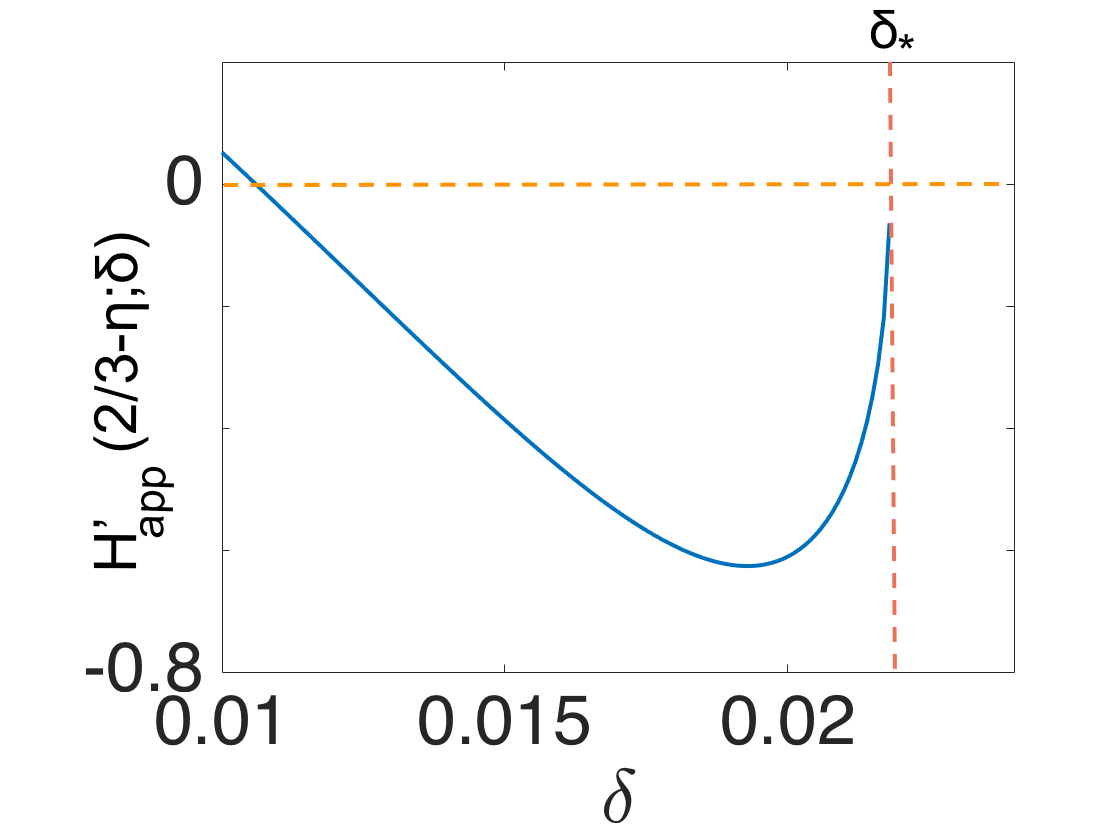}\quad
        \includegraphics[width=0.36\textwidth]{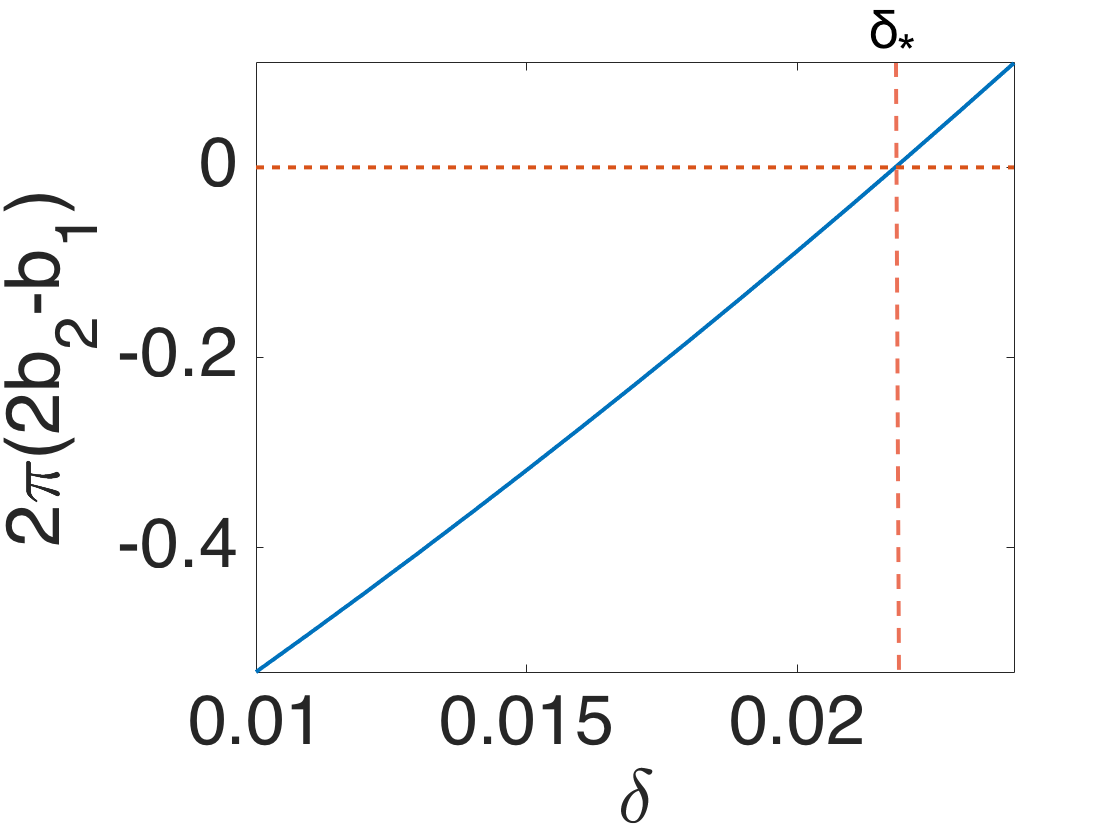}
\end{center}
\caption{(Left) $H_{app}'\lt( \frac{2}{3} -\eta; \delta\rt) <0$,
(right) $H'_{app} \lt(\frac{1}{2}; \delta\rt)$.}
 \label{der_Happ_23_13_12_0}
\end{figure}
\item  [Conditions of Proposition~\ref{fixedpoints_zero_T2}] 
Finally, we verify the stability types of $(1/2,1/2)$ and $(0,0)$. 
\bi
\item  For $\delta< \delta_*$, $H_{app}'\lt({1}/{2}; \delta\rt)<0$, 
while for  $\delta>\delta_*$, $H_{app}'\lt({1}/{2}; \delta\rt)>0$. 
Setting $\theta=1/2$ in Equation~(\ref{der_Happ}), we get 
\[H_{app}'(1/2;\delta) = 2\pi(2 b_2(\delta) - b_1(\delta)).\]
By the definition of $\delta_*$, for $\delta< \delta_*$, $-b_1/2b_2 < -1$.
Figure~\ref{der_Happ_23_13_12_0} (right)  shows that 
$H_{app}'(1/2,\delta)$ changes sign from negative to positive
at $\delta= \delta_*$. This guarantees that the fixed point $(1/2,1/2)$
becomes stable as $\delta$ passes $\delta_*$.

\item $H_{app}'(0;\delta)<0$. Setting $\theta=0$ in Equation~(\ref{der_Happ}),
we obtain $H_{app}'(0;\delta) = 2\pi(b_1(\delta) + 2 b_2(\delta))$, which is
negative because for $\delta\in[0.008 0.024]$ both  $b_1(\delta)$ and $b_2(\delta)$ are negative
(see Figure~\ref{fourier_coef_delta}). This guarantees that $(0,0)$ is always a source.  
\ei
\end{description}

\subsection{Bifurcation diagrams: 
balance conditions and equal contralateral couplings}
In this section, we consider  Equations (\ref{eq.osc.simplified})
for $H=H_{app}$ and study the bifurcations as $\delta$ increases. 
\ZAB{We draw the bifurcation  diagrams (Figure~\ref{Bifurcation_alpha_13_12_matcont}) using Matcont, 
a Matlab numerical continuation packages for the interactive bifurcation analysis of dynamical systems 
 \cite{Matcont}.}
We first consider the system with $\a=1/3$. When $\delta$ is small,
$\delta=0.01$, as Figures~\ref{Bifurcation_Happ_del} (first row, left) 
shows, there exist 12 fixed points: 6 saddle points, 2 sources, and 4 sinks. 
In this case, $(1/3+\et, 1/3+\et)$ is a sink (shown by a green dot in 
Figure~\ref{Bifurcation_Happ_del}). 
As $\delta$ increases and reaches $\delta^{(0)}$ (Figure~\ref{Bifurcation_alpha_13_12_matcont} (left)), 
through a transcritical bifurcation, $(1/3+\et, 1/3+\et)$
becomes a saddle. 
Further, as $\delta$ reaches $\delta^{(1)}$,  through a saddle node bifurcation, a sink (green dot)
and a saddle (orange star) annihilate each other and 10 fixed points remain: 
5 saddle points, 2 sources, and 3 sinks (see Figures~\ref{Bifurcation_Happ_del} (first row, right) and 
\ref{Bifurcation_alpha_13_12_matcont} (left)). 
Note that the two extra fixed points were not observed in the case of the numerically computed $H$ and 
 the transcritical and saddle node bifurcations did not occur.

As $\delta$ increases further to  $\delta^{(2)}$, through a degenerate
bifurcation, 4 fixed points disappear and only 6 fixed points remain (see 
Figures~\ref{Bifurcation_Happ_del} (second row, left) and 
\ref{Bifurcation_alpha_13_12_matcont} (left)).   

When $\delta$ reaches $\delta^{(3)}$, 2 fixed points vanish 
in a saddle node bifurcation and 4 fixed points remain: 2 saddle points, 
a source, and a sink  (see Figures~\ref{Bifurcation_Happ_del} 
(second row, right) and \ref{Bifurcation_alpha_13_12_matcont} (left)).
Note that 2 saddle points and 1 source near the edges of the square 
remain unchanged while $\delta$ varies. 
Figure~\ref{Bifurcation_alpha_13_12_matcont} (left) shows the 
bifurcation diagram when $\alpha=1/3$. 

\bremark
Figure~\ref{Bifurcation_alpha_13_12_matcont} (right) shows the bifurcation
diagram when $\alpha=1/2$. In this case, due to reflection symmetry
about $\theta^{(1)} = \theta^{(2)}$, there is no saddle node bifurcation
 at $\delta=\delta^{(3)}$ (as in the case of $\a=1/3$), and 7 
fixed points merge to $(1/2, 1/2)$ in a very degenerate bifurcation. 
\eremark
\begin{figure}[h!]
\begin{center}
        \includegraphics[scale=.12]{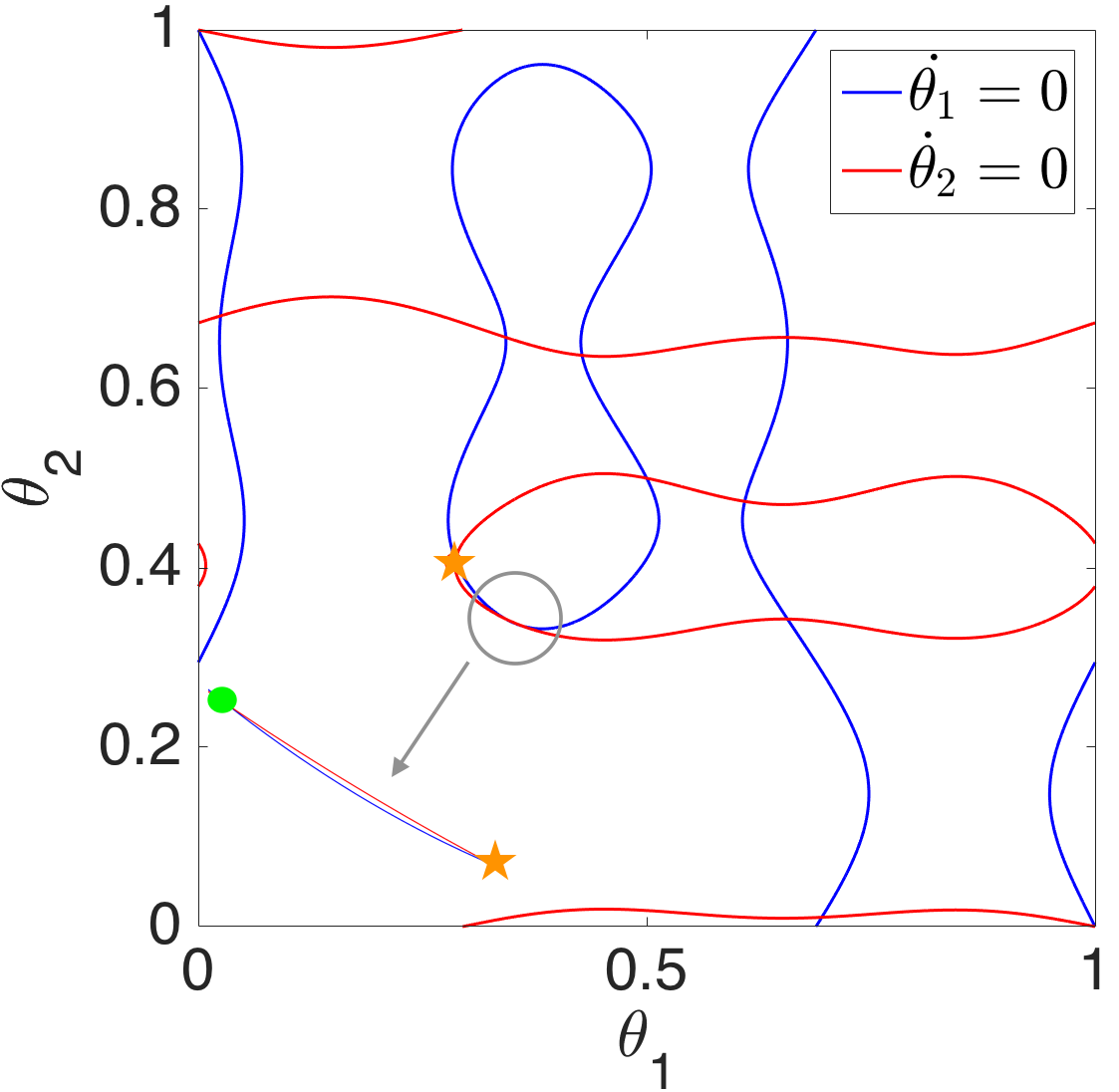}\quad
        \includegraphics[scale=.12]{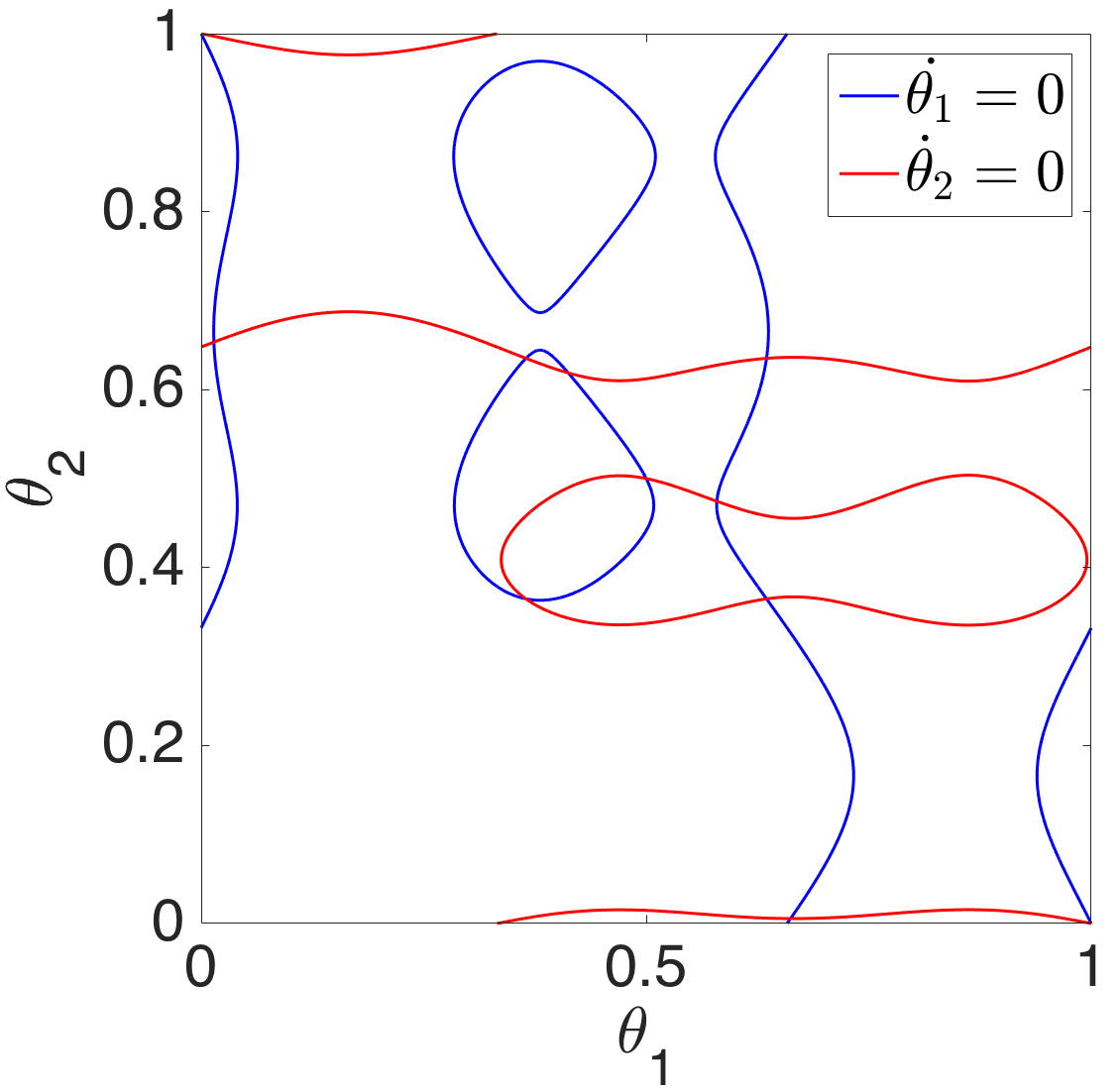}
        
        \includegraphics[scale=.12]{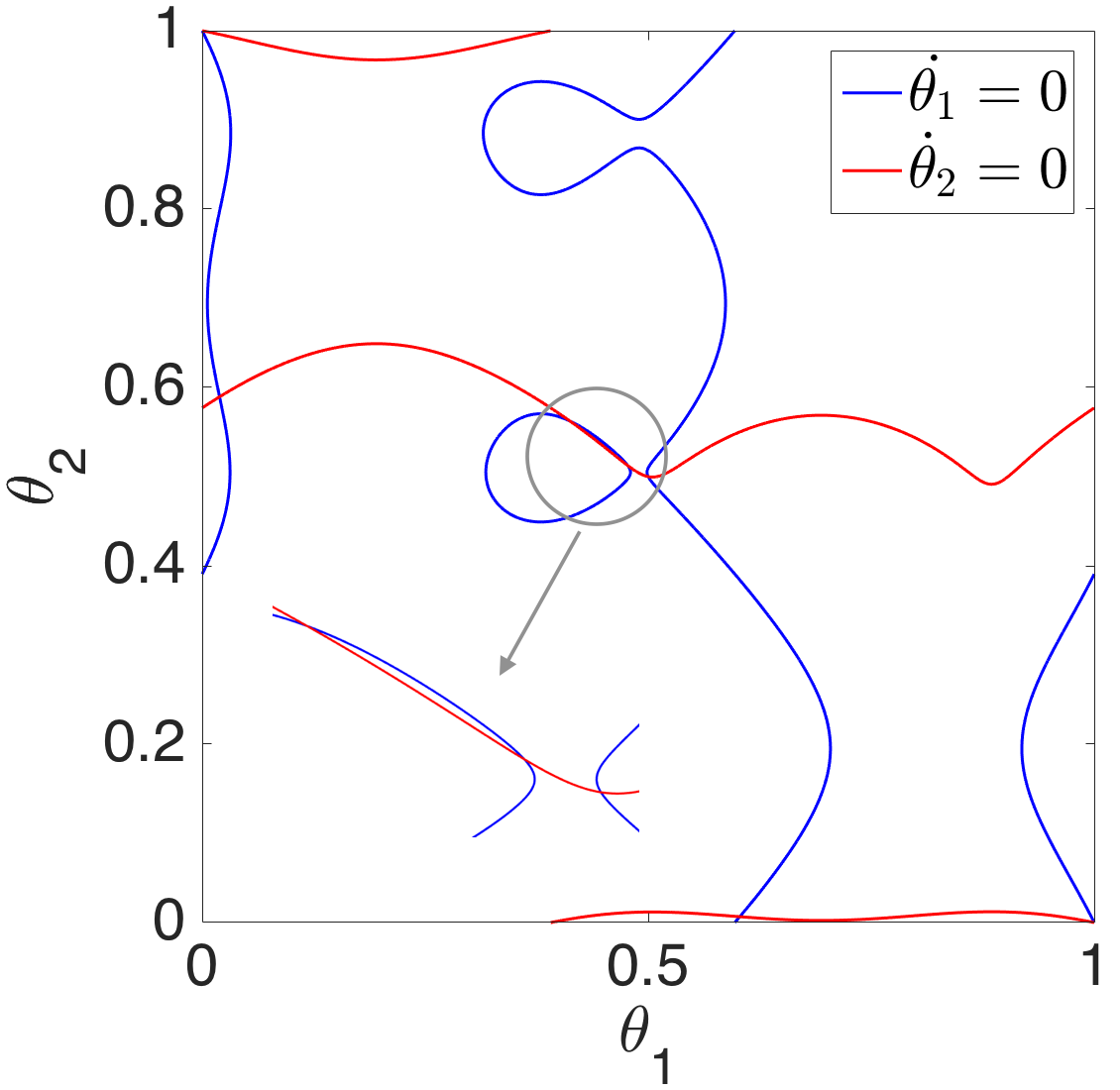}\quad
        \includegraphics[scale=.12]{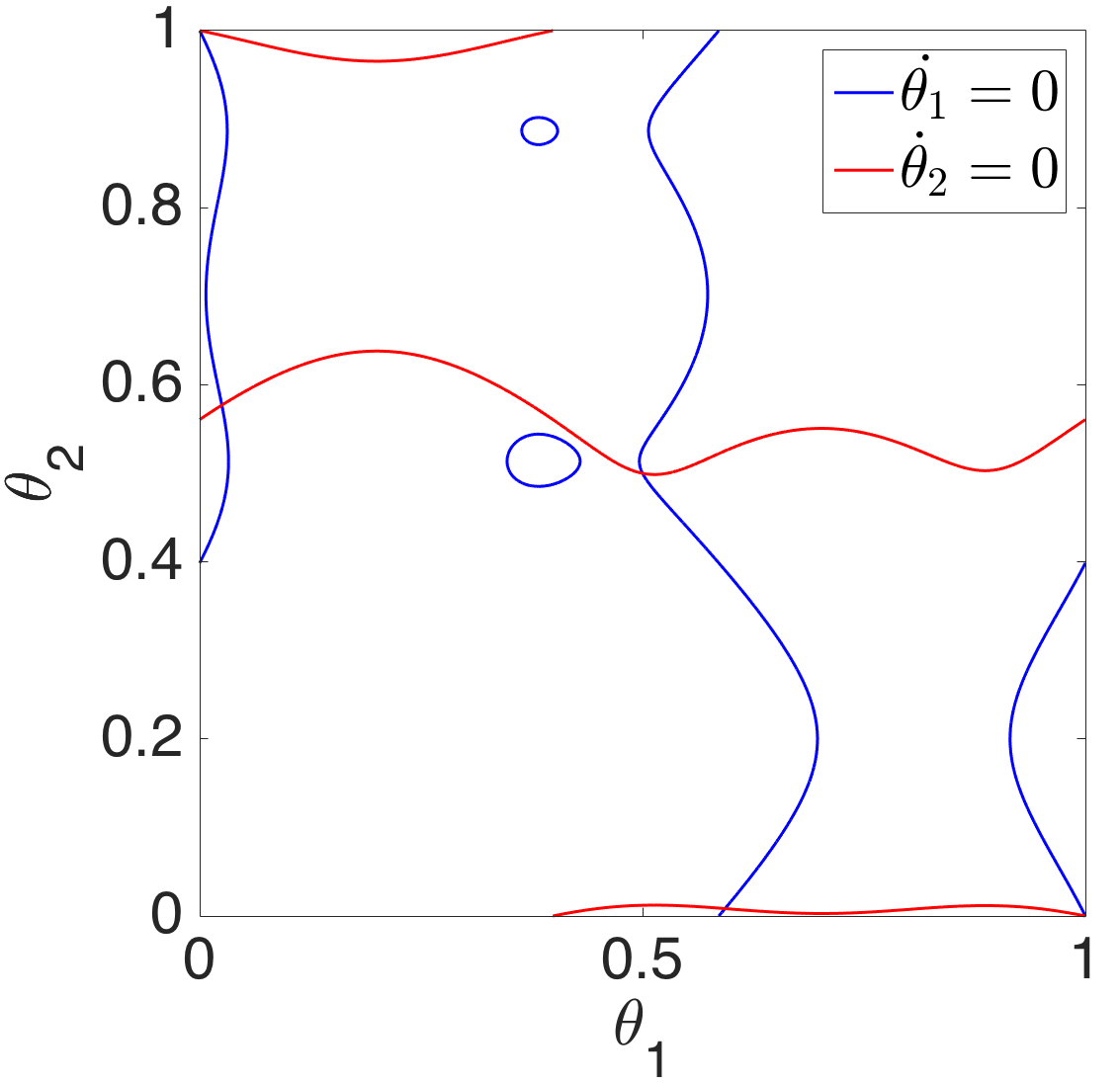}
\end{center}
\caption{ Nullclines of Equations~(\ref{eq.osc.simplified}) with $H=H_{app}$, $\a=1/3$ and 4 values of $\delta$
are shown. Note enlargements of nullcline intersections in left column.
First row: $\delta=0.01$ (left), $\delta=0.014$ (right); as $\delta$ increases,
a transcritical bifurcation at $\delta=\delta^{(0)}>0.01$ 
and a saddle node bifurcation at $\delta=\delta^{(1)}<0.014$ occur. 
Second row: 
$\delta=0.023$ (left), $\delta=0.025$ (right); as $\delta$ increases, a 
degenerate bifurcation at $\delta=\delta^{(2)}>0.023$ and a saddle node bifurcation at $\delta=\delta^{(3)}>0.025$ occur. 
The corresponding bifurcation diagram is shown in Figure~\ref{Bifurcation_alpha_13_12_matcont} (left)). 
Note that the green dot indicates a sink and the orange star indicates a saddle point. 
See text for further explanation. 
}
\label{Bifurcation_Happ_del}
\end{figure}
\begin{figure}[h!]
\begin{center}
\includegraphics[scale=.2]{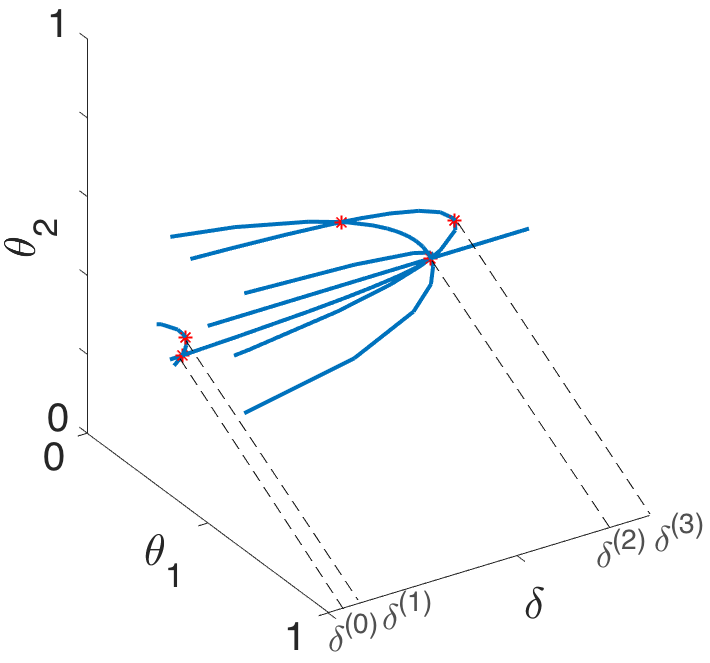}\qquad\qquad
\includegraphics[scale=.18]{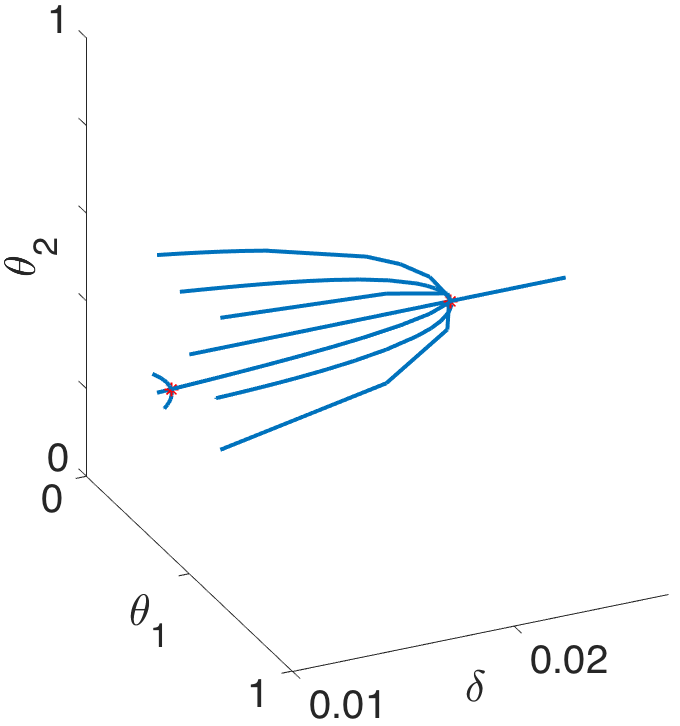}
\caption{Partial bifurcation diagrams 
of Equations~(\ref{eq.osc.simplified}) with $H=H_{app}$
for $\a=1/3$ (left) and $\a=1/2$ (right). 
In both cases the coupling strengths are balanced, but the $\a=1/3$
case is not rostrocaudally symmetric.  The source $(0,0)$ and two saddle points near 
$\theta_1=0$ and $\theta_2=0$ are omitted.}
\label{Bifurcation_alpha_13_12_matcont}
\end{center}
\end{figure}

\section{Gaits deduced from fruit fly  data fitting}
\label{DATA}

In this section, we use two sets of  coupling strengths which were estimated 
for slow, medium, and fast wild-type fruit flies in our reduced model on the 
torus and show the existence of stable tetrapod gaits at low frequency and 
stable tripod gaits at higher frequency. To vary frequency, we change $I_{ext}$
in the first set of estimates in Section \ref{data_vaibhav}, and we change 
$\delta$ in the second set of estimates in Section \ref{data_Enet}. Unlike
the gait transitions of Section~\ref{existence_and_stability}, the fitted data
predict different coupling strengths across the speed range. 
\ZAB{As in previous sections, we display both results from the bursting neuron model and the nullclines and phase planes
from the reduction to the $(\theta_1,\theta_2)$ plane.}

\subsection{Dataset 1}
\label{data_vaibhav}

We first exhibit a gait transition from tetrapod  to tripod 
as $I_{ext}$ increases.
\begin{table}[ht]
\begin{center}
 \begin{tabular}{|c |  c c c  c c c c c  |} 
 \hline
 & $\hat\omega$ 
 &$\c_1$ 
 & $\c_2$
 & $\c_3$
 & $\c_4$ 
 & $\c_5$ 
 & $\c_6$
 & $\c_7$
 \\ 
  [0.5ex] 
 \hline\hline
 slow 
 &9.92
 & 0.3614
 & 0.1478 
 &0.1780  
 &0.1837 
 &0.2509
 &0.3409
 &0.1495
  \\ 
 \hline
medium 
&12.48
&  0.2225 
&  0.6255
& 0.4715
& 0.1436 
&0.3895
&0.7921
&0.2964  
\\ 
 \hline
 fast 
 &15.52
 & 0.0580
 & 0.8608 
 &0.6726  
 &0.0470 
 &0.4294
 &1.1498
 &0.8500  
 \\ 
 \hline
 \end{tabular}
\caption{Values of estimated frequency and coupling strengths for slow, medium, and fast wild-type fruit flies.}
\label{vaibhav_data_table}
\end{center}
\end{table}
Table~\ref{vaibhav_data_table} shows the coupling strengths $\c_i$ which 
were estimated for slow (represented by coupled frequency $\hat\omega = 9.92$),
medium ($\hat\omega = 12.48$), and fast ($\hat\omega = 15.52$) wild-type
fruit flies. 
These fits were obtained after linearizing Equations~\eqref{eq.osc1} and adding i.i.d.  zero mean Gaussian noise to each equation. The touchdown times of every leg are treated as measurements of the phase of its associated oscillator in Equations~\eqref{eq.osc1}, additionally corrupted by a zero mean Gaussian measurement noise. 
To incorporate the circular nature of phase measurements, the initial condition distribution for Equations~\eqref{eq.osc1} is modeled by a mixture Gaussian distribution. 
For each sequence of leg touchdowns, a Gaussian sum filter~\cite{alspach1972nonlinear} is used to compute the distribution and the log-likelihood of leg touchdown times. The aggregate log-likelihood for pooled sequences of leg touchdowns for different flies is maximized to compute the maximum likelihood estimates (MLEs) of coupling strengths, phase differences, and 
variance of the i.i.d. measurement noises.

We choose 3 different values of $I_{ext}$:  $I_{ext}= 35.95$ for slow
 (represented by  uncoupled frequency $\omega=  8.76$), $I_{ext}= 36.85$ for 
 medium ($\omega= 12.64$), and $I_{ext}= 37.65$ for fast ($\omega=14.85$) speeds.
 Note that in general $\hat\omega<\omega$, because we assume that all the 
 couplings are inhibitory, $\c_i H<0$, although the coupled frequency corresponding
 to the slow and fast speed are not less than the uncoupled frequency in our simulations below.
Also note that the medium and fast speed coupling parameters
(Table~\ref{vaibhav_data_table}, second and third rows) are far from balanced. 

Figure~\ref{vaibhav_figure_gait} shows solutions of the 24 ODEs 
for the following initial conditions:
\be{initial_BN_Vaibhav}
v_1=-40,\;
v_2=10,\;
v_3=-10,\;
v_4=30,\;
v_5=15,\;
v_6=-30,\;
\ee
and for $i=1,\cdots,6$, the $m_i$'s, $\w_i$'s, and $s_i$'s take their steady 
state values as in Equation~(\ref{IC:numerical:tetra:tri:m,w,s:delta}). In
Figure~\ref{vaibhav_figure_gait} (left), $I_{ext} = 35.95$ and the coupling
strengths $\c_i$ are as in Table~\ref{vaibhav_data_table}, first row. 
In Figure~\ref{vaibhav_figure_gait} (middle), $I_{ext} = 36.85$ and the
coupling strengths $\c_i$ are as in Table~\ref{vaibhav_data_table}, second row. 
In Figure~\ref{vaibhav_figure_gait} (right), $I_{ext} = 37.65$ and the
coupling strengths $\c_i$ are as in Table~\ref{vaibhav_data_table}, third row. 
As we expect, these respectively depict  tetrapod, transition, and tripod gaits.  
We computed the solutions up to time $t=5000$ ms but only show the  time
windows $[4800, 5000],$ after transients have died out.  
\begin{figure}[h!]
\begin{center}
\includegraphics[scale=.2]{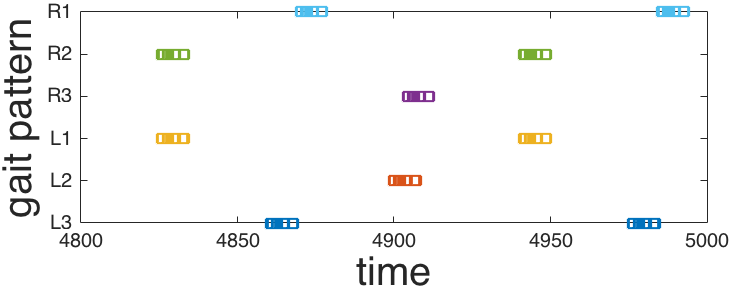}
\includegraphics[scale=.2]{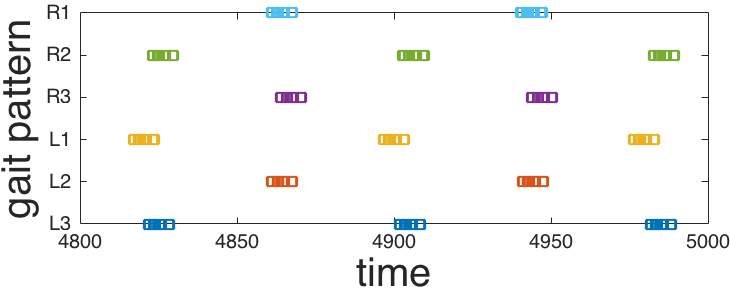}
\includegraphics[scale=.2]{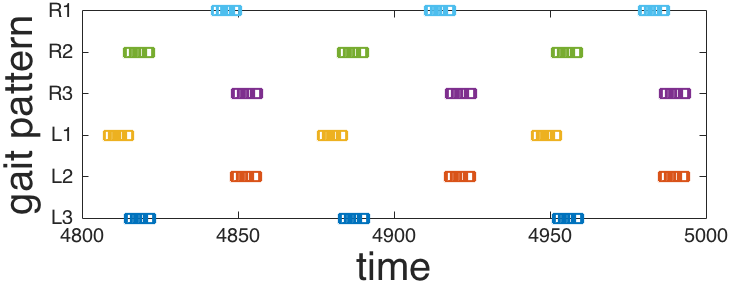}
\caption{
(Left to right) A solution of 24 ODEs  for 
$I_{ext} = 35.95$ and $\c_i$'s as in the first row of Table~\ref{vaibhav_data_table};  
$I_{ext} = 36.85$ and $\c_i$'s as in the second row of Table~\ref{vaibhav_data_table}; and for 
$I_{ext} = 37.65$  and $\c_i$'s as in the third row of Table~\ref{vaibhav_data_table}.}
 \label{vaibhav_figure_gait}
\end{center}
\end{figure}

Figure~\ref{vaibhav_figure_nullcline_phase_plane} shows the nullclines (first row) 
and the corresponding phase planes (second row) of 
Equation~(\ref{torus:equation}) for the three different values of $I_{ext}$.
 As Figure~\ref{vaibhav_figure_nullcline_phase_plane} 
 (left) depicts, when the speed parameter is small, there exist 6 fixed points: 2 sinks
which correspond to the forward and backward tetrapod gaits, a source,  and 3 saddle points. 
 As Figure~\ref{vaibhav_figure_nullcline_phase_plane} 
  (middle) depicts, when the speed parameter  increases, there exist 4 fixed points:
 a sink which corresponds to the transition  gait, a source,  and 2 saddle points. 
As Figure~\ref{vaibhav_figure_nullcline_phase_plane} 
(right) depicts, when the speed parameter is large, there exist only 2 fixed points:
a sink  corresponding to the tripod gait and a saddle point. 
\begin{figure}[h!]
\begin{center}
\includegraphics[scale=.1]{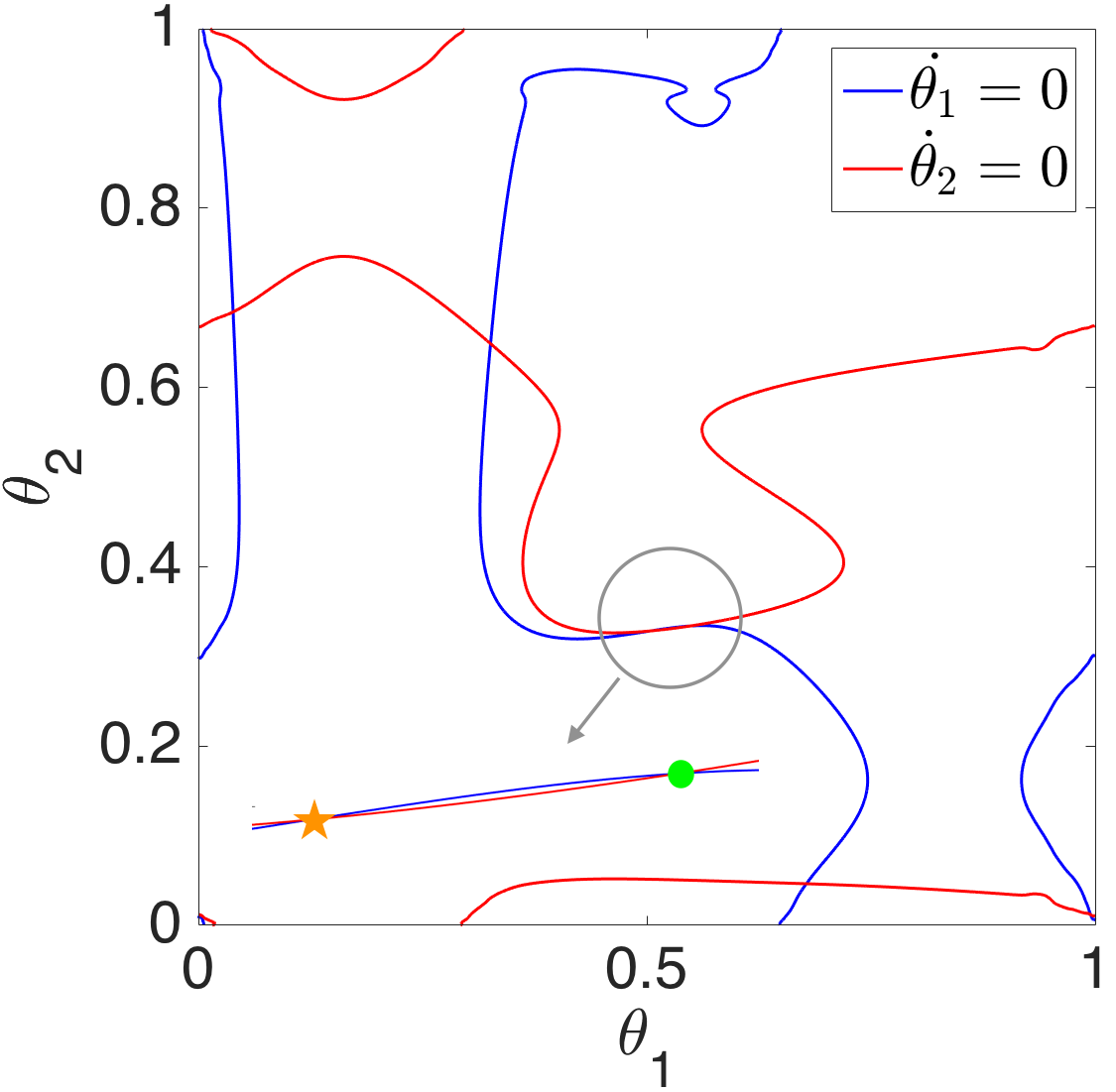}\quad
\includegraphics[scale=.1]{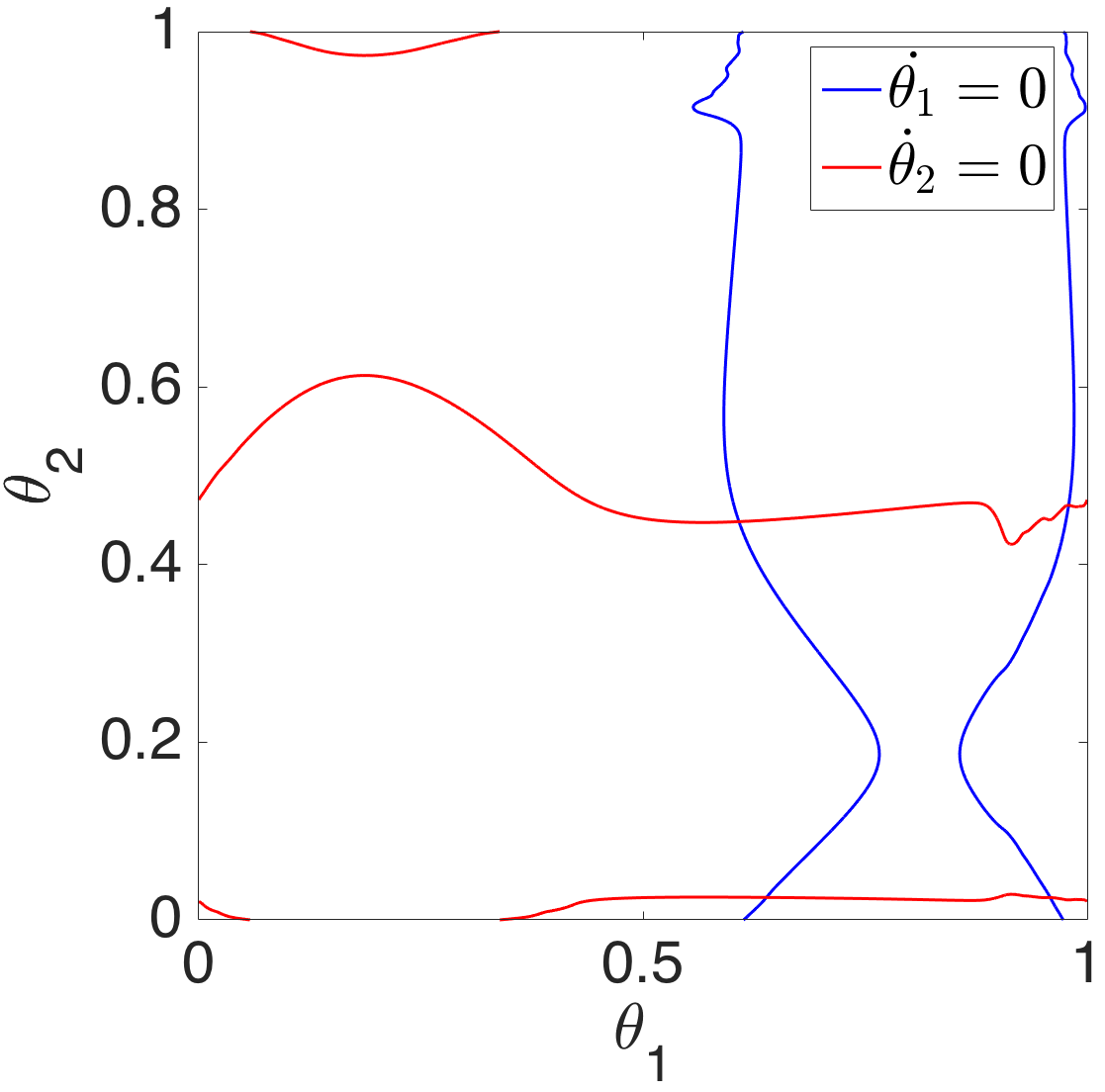}\quad
\includegraphics[scale=.1]{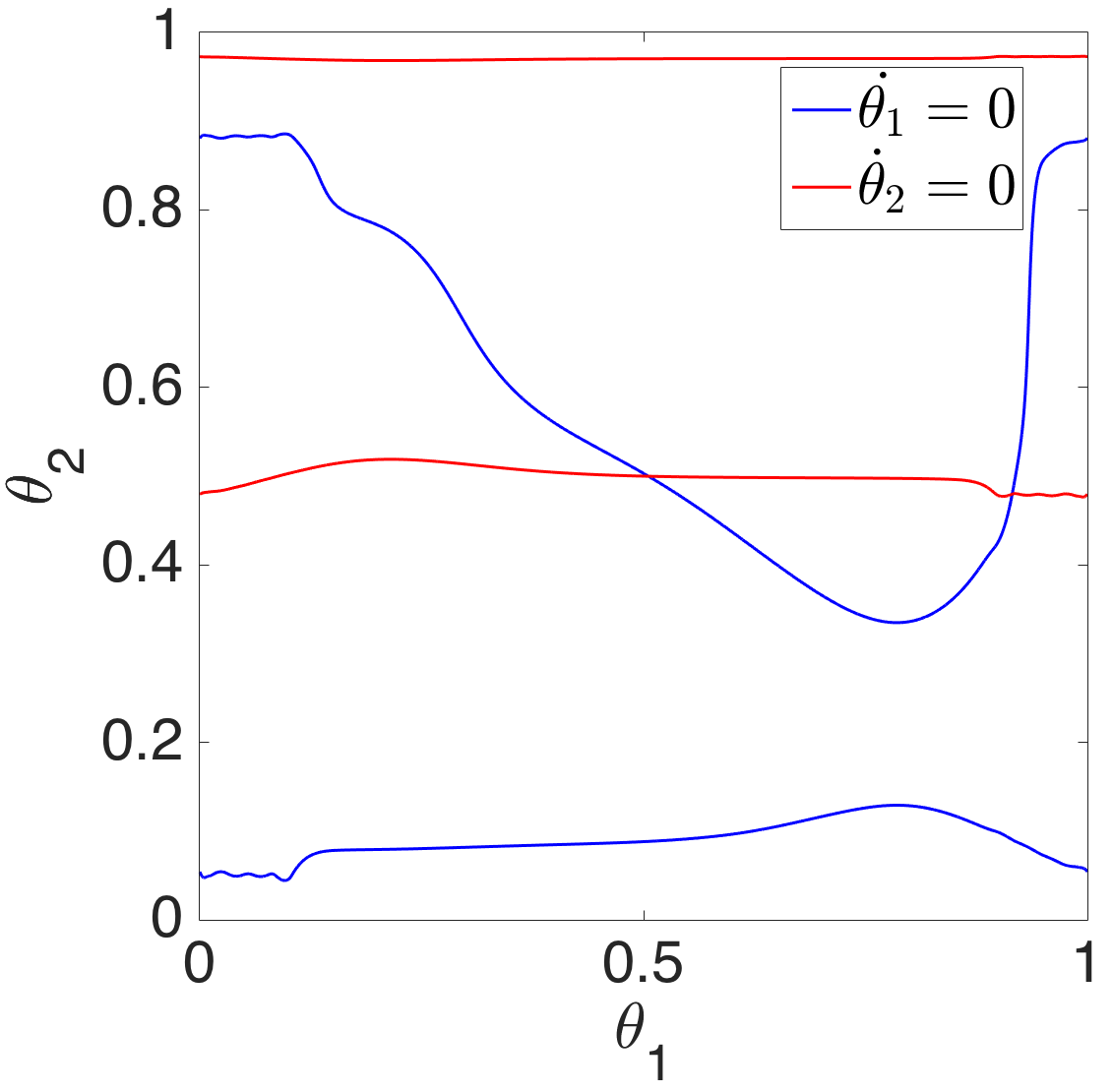}

\includegraphics[scale=.1]{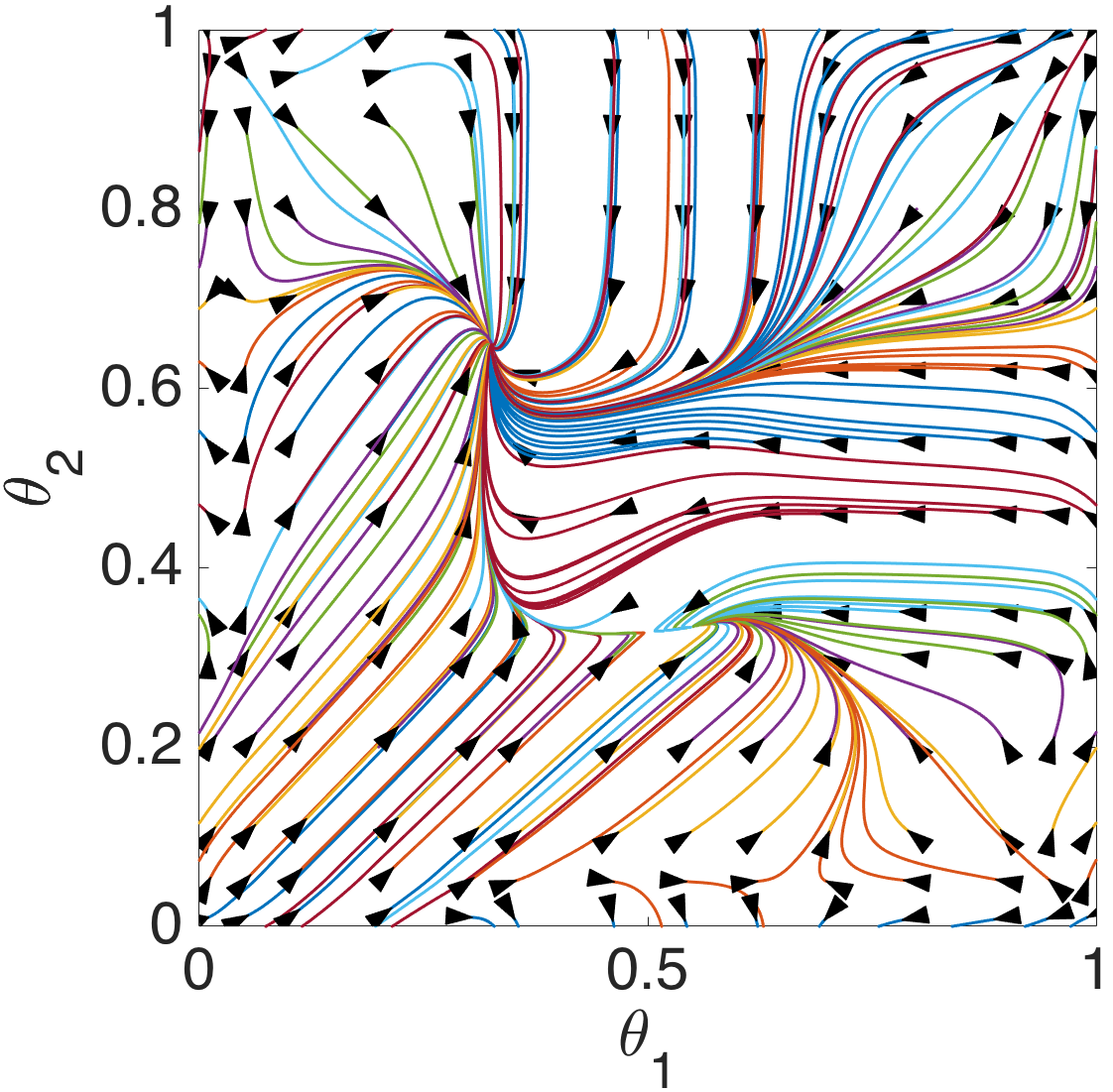}\quad
\includegraphics[scale=.1]{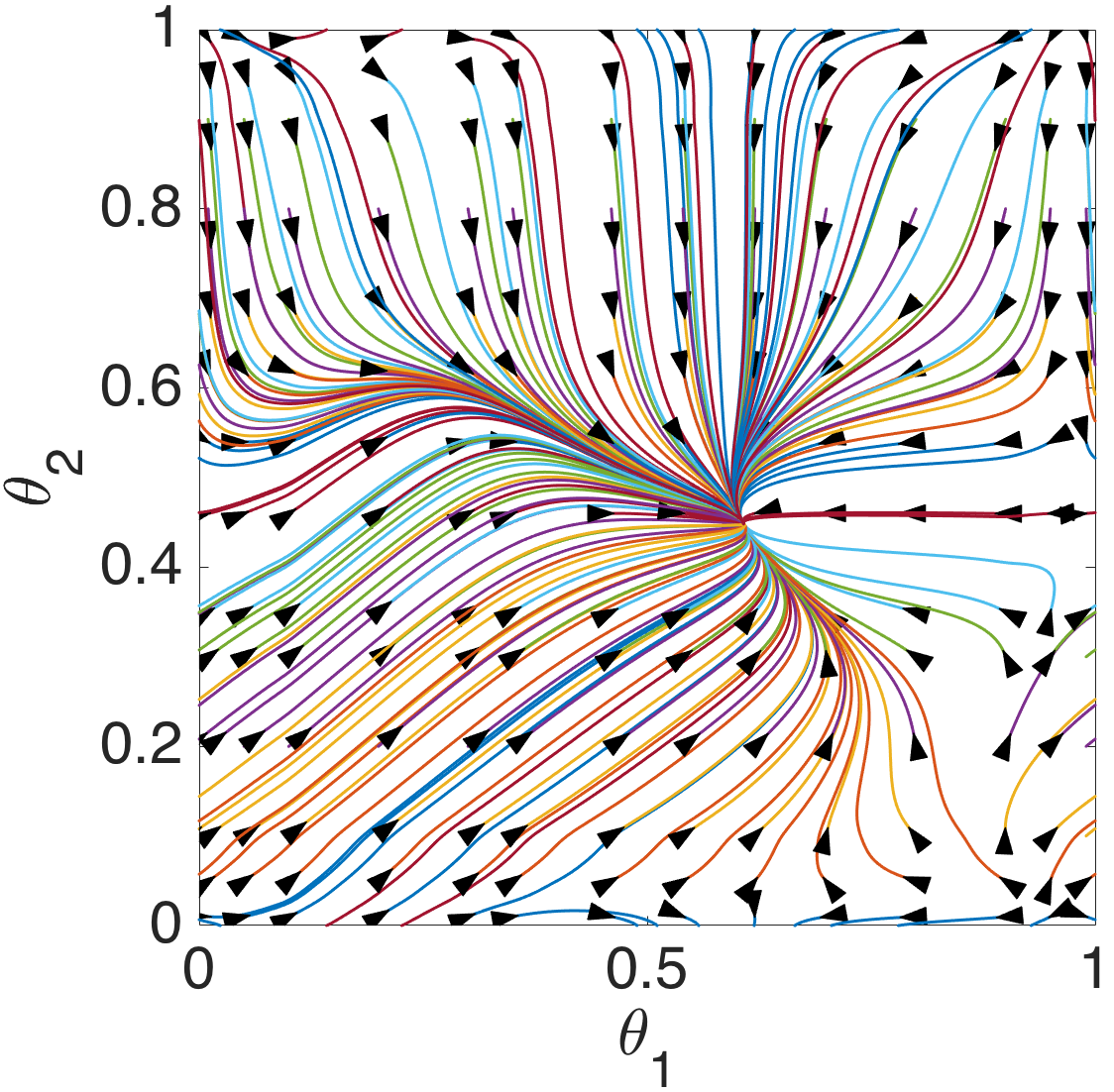}\quad
\includegraphics[scale=.1]{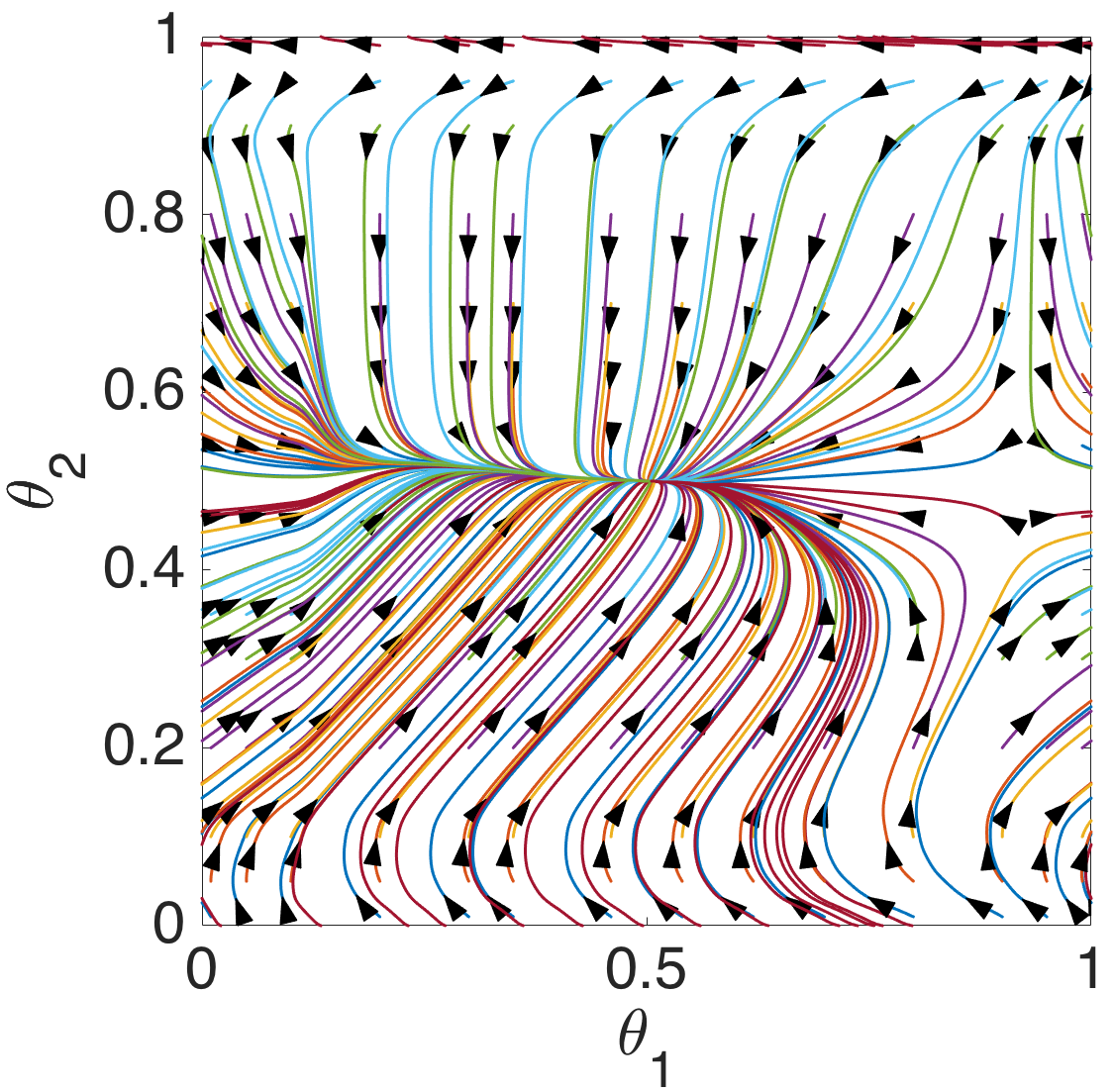}
\caption{
(First row: left to right) Nullclines  of Equations~(\ref{torus:equation}) for  
$I_{ext} = 35.95$ and $\c_i$'s  as in the first row of Table~\ref{vaibhav_data_table};  
$I_{ext} = 36.85$ and $\c_i$'s  as in the second row of Table~\ref{vaibhav_data_table}; and for 
$I_{ext} = 37.65$  and $\c_i$'s  as in the third row of Table~\ref{vaibhav_data_table}.
(Second row: left to right) Corresponding phase planes. 
Note that the green dot indicates a sink and the orange star indicates a saddle point. 
See text for further explanation. }
\label{vaibhav_figure_nullcline_phase_plane}
\end{center}
\end{figure}

\subsection{Dataset 2}
\label{data_Enet}

In this section, we show a gait transition from tetrapod  to tripod, as $\delta$ increases.
\begin{table}[ht]
\begin{center}
 \begin{tabular}{|c |  c c c  c c c c c  |} 
 \hline
 & $\hat\omega$ 
 &$\c_1$ 
 & $\c_2$
 & $\c_3$
 & $\c_4$ 
 & $\c_5$ 
 & $\c_6$
 & $\c_7$
 \\ 
  [0.5ex] 
 \hline\hline
medium 
&12.23
&0.2635 
&1.2860
& 2.9480
& 1.3185 
&1.3885 
&2.5025
&1.2265
  \\ 
 \hline
 fast 
 &15.65
 &2.9145
 & 2.5610
 & 2.6160 
 &2.9135  
 &5.1800 
 &5.4770
 & 2.6165 
  \\ 
 \hline
 \end{tabular}
\caption{Values of estimated frequency and coupling strengths for  medium, and fast free-walking wild-type fruit flies.}
\label{Enet_data_table}
\end{center}
\end{table}
Table~\ref{Enet_data_table} shows the coupling strengths $\c_i$ which were
estimated for medium (represented by coupled frequency $\hat\omega = 12.23$)
and fast ($\hat\omega = 15.65$) wild-type fruit flies. 
These fits are obtained using linearized ODEs similar to~Section~\ref{data_vaibhav}. However, to obtain these fits, touchdown sequences for different flies are concatenated to obtain a single large sequence and a Kalman filter is used to compute the distribution and the log-likelihood of leg touchdown times. The MLEs for coupling strengths are obtained by maximizing the aggregate likelihood for the concatenated touchdown sequence. 

We choose 2 different
values of $\delta$,  $\delta=0.014$ for medium (represented by  uncoupled
frequency $\omega= 3.57$), and $\delta=0.03$ for fast ($\omega=6.91$)
speeds \cite{Couzin_notes_16}. As noted earlier in Section~\ref{range_frequency},
as $\delta$ varies in the bursting neuron model, the range of frequency
does not match the range of frequency estimated from data. In spite of
this, we show that the estimated coupling strengths in the low speed range
(small $\delta$) give a tetrapod gait and in the high speed range 
(large $\delta$) give a tripod gait. 

Figure~\ref{Einet_figure_gait} shows solutions of the 24  ODEs  for  the following
initial conditions. 
\be{initial_BN_Einet}
v_1=-10,\;
v_2=-40,\;
v_3=-30,\;
v_4=-40,\;
v_5=5,\;
v_6=20,\;
\ee
and for $i=1,\cdots,6$, $m_i$'s, $\w_i$'s, and $s_i$'s take their steady
state values as in Equation~(\ref{IC:numerical:tetra:tri:m,w,s:delta}). 
In Figure~\ref{Einet_figure_gait} (left), $\delta=0.014$ and the coupling
strengths $\c_i$ are as in Table~\ref{Enet_data_table}, first row. 
In Figure~\ref{Einet_figure_gait} (right), $\delta=0.03$ and the coupling 
strengths $\c_i$ are as in Table~\ref{Enet_data_table}, second row. 
As we expect, Figure~\ref{Einet_figure_gait} (left to right) depicts transition
(still very close to a tetrapod gait) and tripod gaits, respectively.  
We computed the solutions up to time $t=5000$ ms but only show the
time window $[4000, 5000],$ after transients have died out.  
\begin{figure}[h!]
\begin{center}
\includegraphics[scale=.2]{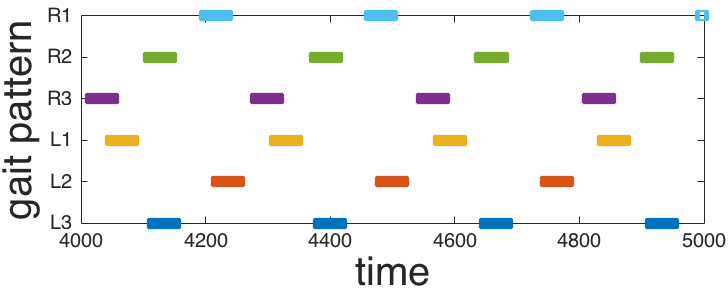}\quad
\includegraphics[scale=.2]{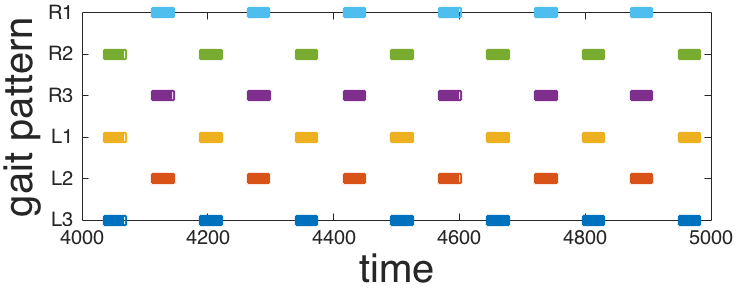}
\caption{(Left to right) A solution of 24 ODEs  for
$\delta=0.014$ and $\c_i$'s  as in the first row of Table~\ref{Enet_data_table} and for
$\delta= 0.03$ and $\c_i$'s  as in the second row of Table~\ref{Enet_data_table}. 
Note the approximate tetrapod and almost perfect tripod gaits.}
\label{Einet_figure_gait}
\end{center}
\end{figure}

Figures~\ref{Einet_figure_nullcline_phase_plane} (left to right) show the
nullclines and corresponding phase planes of Equations~(\ref{torus:equation})
for the two different values of $\delta$. As 
Figure~\ref{Einet_figure_nullcline_phase_plane} (left) depicts, when the
speed parameter is relatively small, there exist 4 fixed points: a sink which
corresponds to a transition gait, a source and 2 saddle points. 
Figure~\ref{Einet_figure_nullcline_phase_plane} (right) shows that these
fixed points persist as the speed parameter increases, but the sink now
corresponds to a tripod gait. No bifurcation of fixed points occurs, although
the topology of the nullclines changes.
\begin{figure}[h!]
\begin{center}
\includegraphics[scale=.1]{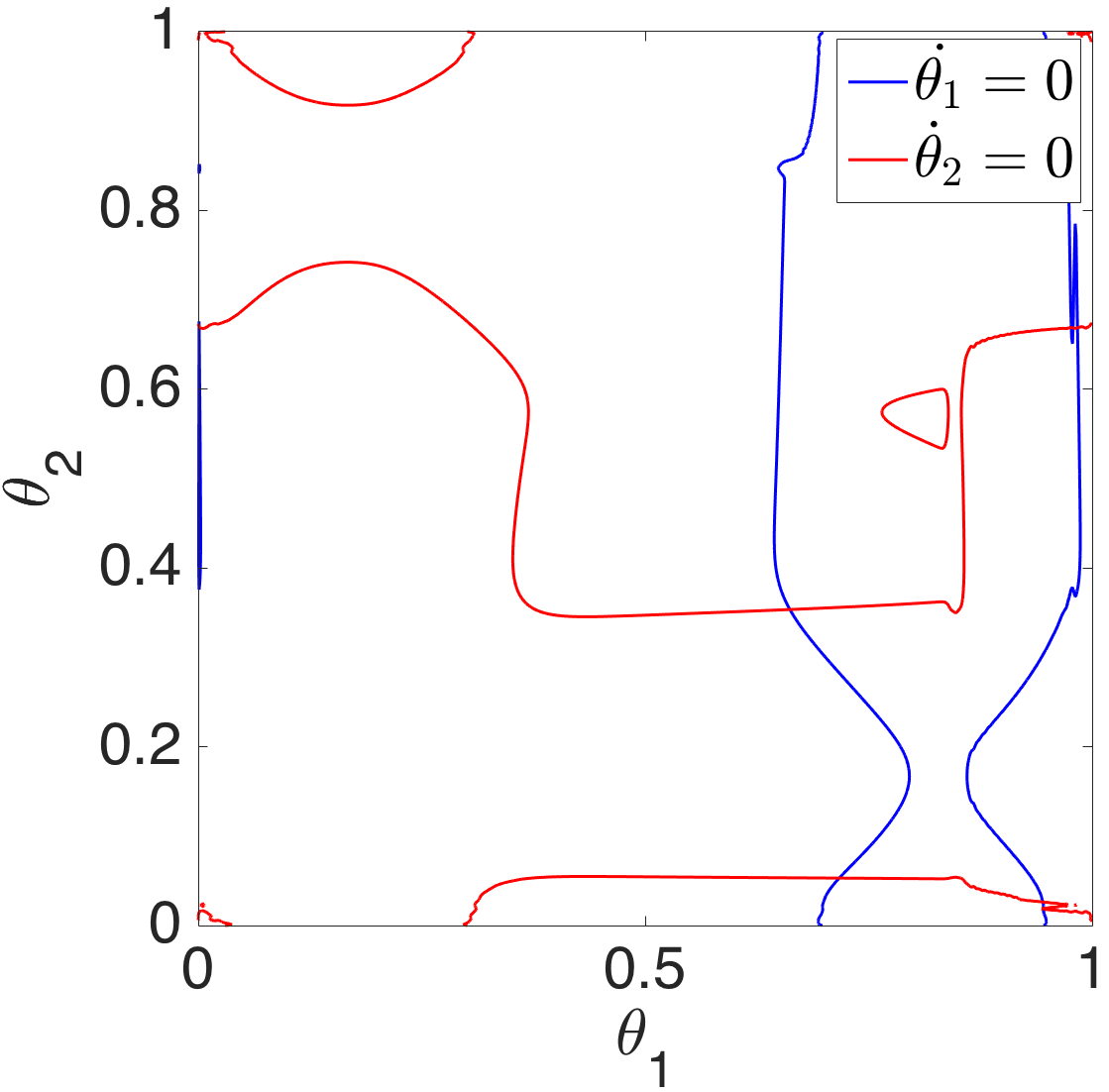}
\includegraphics[scale=.1]{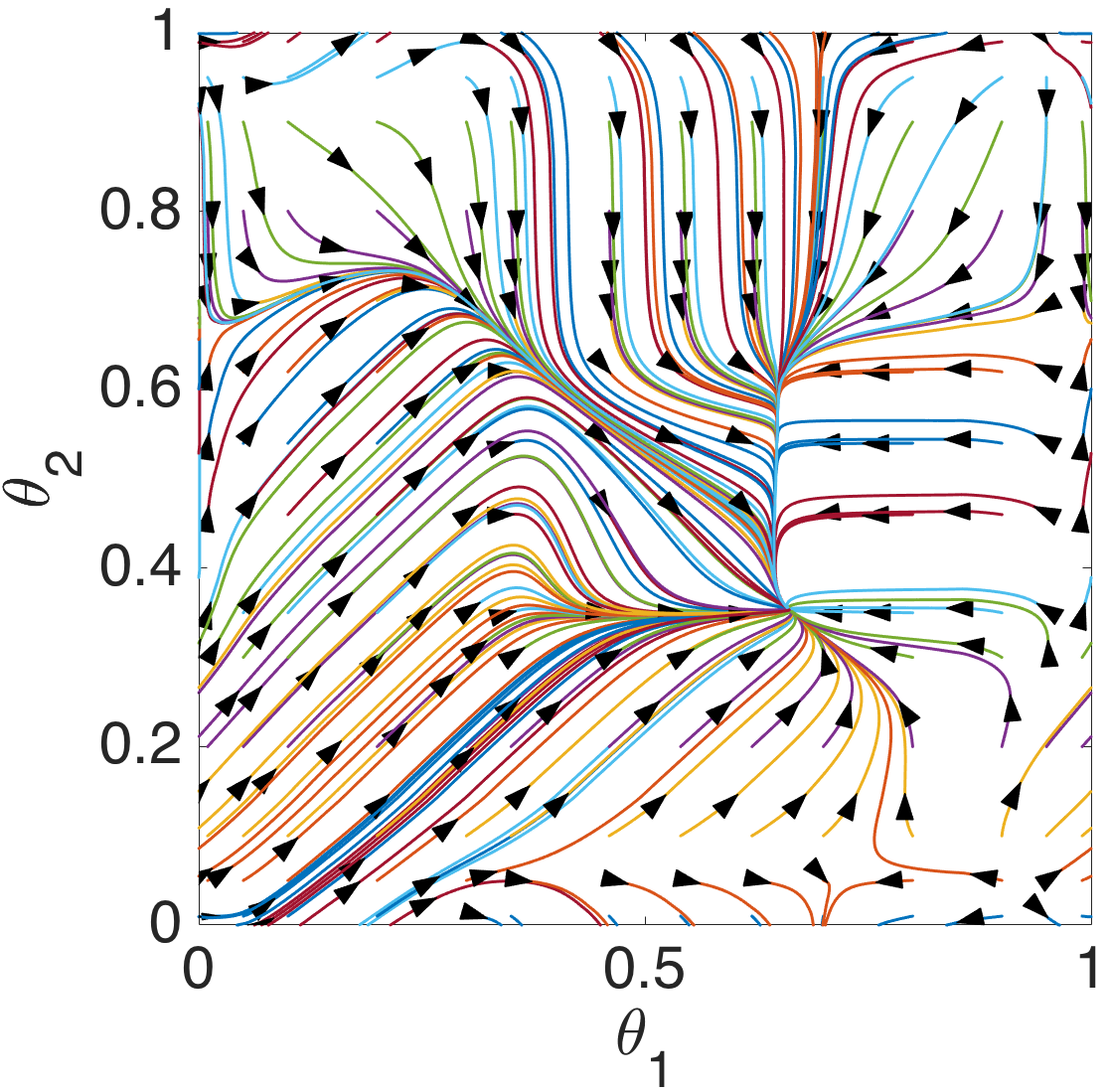}
\includegraphics[scale=.1]{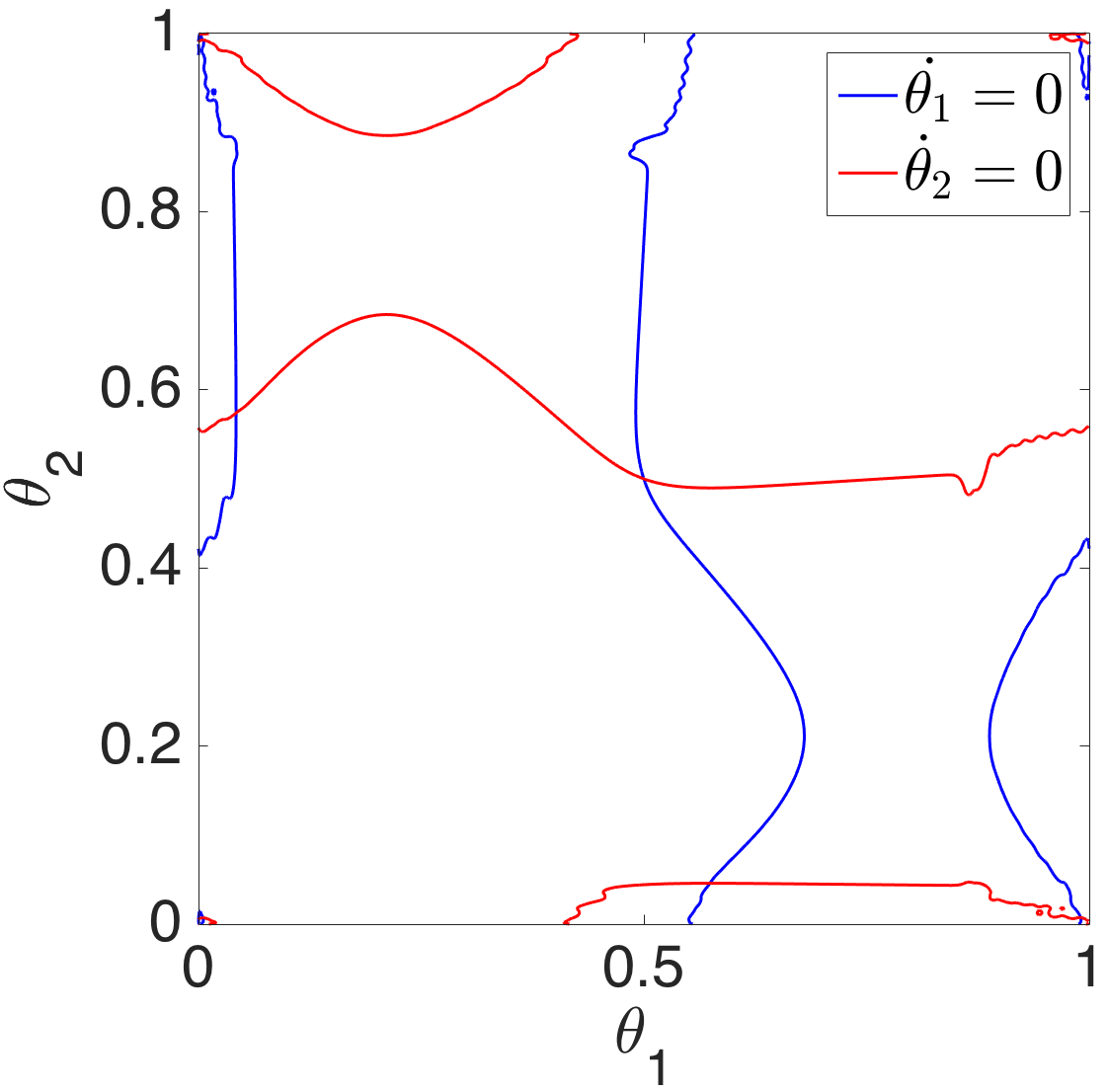}
\includegraphics[scale=.1]{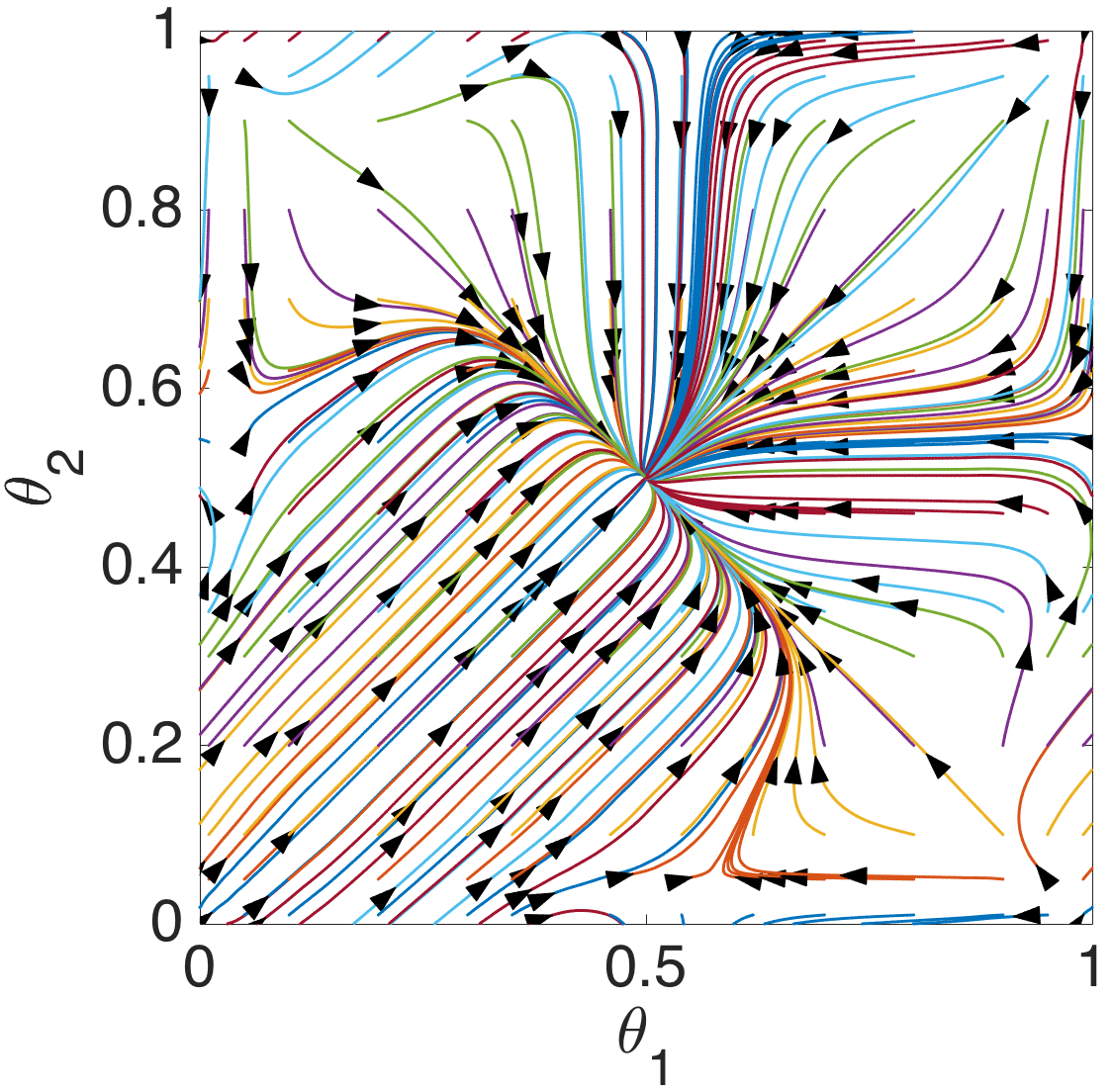}
\caption{(Left to right) Nullclines  and phase planes of Equations~(\ref{torus:equation}) for  
$\delta=0.014$ and $\c_i$'s  as in the first row of Table~\ref{Enet_data_table} (left pair); 
$\delta= 0.03$ and $\c_i$'s  as in the second row of Table~\ref{Enet_data_table} (right pair).
Note the close approximation to reflection symmetry at right due to almost perfect balance,
$\c_1\approx\c_2\approx \c_3$ and $\a = 0.5269 \approx 1/2$.}
\label{Einet_figure_nullcline_phase_plane}
\end{center}
\end{figure}

Note that the estimated coupling strengths in only the second row of
Table~\ref{Enet_data_table} approximately satisfy the balance equation
(\ref{balance2}) and also $\c_1 \approx \c_2 \approx \c_3$. Hence, as
our analysis predicts, the system has 4 fixed points: a sink corresponding
to a tripod gait, a source and 2 saddle points. Although the other estimated
coupling strengths do not satisfy the balance equation (\ref{balance2}),
we still observe the existence of one sink which corresponds to a tetrapod
gait (slow speed), a transition gait (medium speed), or a tripod gait (high
speed). As discussed earlier, the balance equation is a necessary
condition for the existence of tetrapod and tripod gaits but it is not sufficient.
The estimated coupling strengths in Tables~\ref{vaibhav_data_table}
and \ref{Enet_data_table} (first row), provide counterexamples. 

\bremark
The coupling strengths $\c_i$ in Tables~\ref{vaibhav_data_table} and 
\ref{Enet_data_table} are at most $O(1)$, the largest being 
$\approx 5.48$ in Table \ref{Enet_data_table}. From 
Figures~\ref{H_PRC_versus_delta} and \ref{H_PRC_versus_Iext}
(second rows), the maxima of $\abs{H}$ are $0.19$
(as $I_{ext}$ varies) and $0.4$ (as $\delta$ varies). Thus $\abs{c_i H}$
takes maximum values of $0.19 \times 1.15 \approx 0.219$ in 
Table~\ref{vaibhav_data_table} and $0.4 \times 5.48 \approx 2.19$
in Table~\ref{Enet_data_table}. 
\ZA{For both sets of data, we observe transition from a stable (forward) 
tetrapod gait to a stable tripod gait as the speed parameter
$\x$ increases. However,  the coupled frequency $\hat\omega$ should be 
less than the uncoupled frequency $\omega$, 
which does not hold  in some cases.}
\eremark 

\section{Discussion}
\label{conclusion}
In this paper we developed an ion-channel bursting-neuron 
model for an insect central pattern generator based
on that of \cite{SIAM2}. We used this to investigate tetrapod
to tripod gait transitions, at first numerically for a system of 24 ODEs
describing cell voltages, ionic gates and synapses, and then for a 
reduced system of six coupled phase oscillators. This still presents 
a challenging problem, but by fixing contralateral phase differences,
we further reduced to three ipsilaterally-coupled oscillators
and thence to a set of ODEs defined on the 2-torus that describes
phase differences between front and middle and hind and middle
legs. This allowed us to study different sets of inter-leg coupling 
strengths as stepping frequency increases, and to find constraints
on them that yield systems whose phase spaces are amenable
to analysis.

\ZAB{Recent studies of different 3-cell ion-channel bursting CPG networks 
\cite{Wojcik-PLoSone14, Barrio-EuroPhysLet15, Lozano-SciRepts16} 
share some common features with the current paper. Without explicitly addressing insect locomotion, 
or using phase reduction theory, the authors numerically extract Poincar\'{e} maps 
defined on 2-dimensional tori which have multiple stable fixed points corresponding to orbits with specific 
phase differences. In \cite{Lozano-SciRepts16} they discuss transient control inputs that can move
solutions from one stable state to another. A more abstract study of coupled cell systems with an emphasis on 
heteroclinic cycles that lie in ``synchronous subspaces" appears in \cite{Aguiar-JNS11}.}

In addition to Propositions \ref{prop:tetra:stability}, \ref{omega_hat}, 
 \ref{special_coupling},  \ref{special_coupling_corollary} and Corollary \ref{cor:tetra}, 
 which characterize particular tetrapod and tripod solutions of the
phase and phase-difference equations, our main results in
Sections~\ref{existence_and_stability} and \ref{DATA}
illustrate the existence of these solutions and their stability
types. Figures~\ref{NC_PP_balance_Tr_Det} and
\ref{phase differences model}-\ref{phase differences model_near_one_del}
display nullclines and phase portraits for systems
with balanced coupling strengths, showing how a set of fixed points
arrayed around a square astride the main diagonal 
$\theta_1= \theta_2$ on the 2-torus collapses to a single
fixed point, corresponding to a stable tripod gait, as speed increases.
Figures~\ref{Bifurcation_Happ_del} and \ref{Bifurcation_alpha_13_12_matcont}
illustrate nullclines and bifurcation diagrams for a Fourier series
approximation of the coupling function. Finally, 
Figures~\ref{vaibhav_figure_gait}-\ref{Einet_figure_nullcline_phase_plane}
show gaits, nullclines and phase portraits for several cases
in which coupling strengths were fitted to data from free running animals.

While details vary depending upon the coupling strengths, the
results of Section~\ref{existence_and_stability} reveal a
robust phenomenon in which a group of fixed points that
include stable forward and backward tetrapod gaits converge
upon and stabilize a tripod gait.
This occurs even for coupling strengths that are far from balanced. 
 For the coupling strengths
derived from data in Section~\ref{DATA}
(Figures~\ref{vaibhav_figure_gait}-\ref{Einet_figure_nullcline_phase_plane}),
as stepping frequency increases and coupling strengths change
there is still a shift from an approximate forward tetrapod
to an approximate tripod gait, {in which the tetrapod gaits disappear in  saddle node bifurcations.}
In the final example 
(Figures~\ref{Einet_figure_gait} and \ref{Einet_figure_nullcline_phase_plane}
(right panels)) the tripod gait is almost ideal.

In Definition~\ref{defn1} we introduced 4 tetrapod gaits, two of which
feature a wave traveling from front to hind legs. Such backward waves 
are not normally seen in insects and we excluded them from the
gaits illustrated thus far. They do, however, appear as fixed points
in the region $(\theta_1, \theta_2) = (T/3 + \et, 2T/3 - \et)$ on the
torus, which as shown in Proposition~\ref{other_fixedpoints_HBN}, 
are stable for some values of coupling strengths. 
We note that this backward wave in leg touchdowns 
does not imply backward walking,  the study of which 
demands a more detailed model with motoneurons and muscles,
to characterize different legs and leg joint angle sequences, as in e.g.
\cite{TothKnopsDaun-G-JNPhys12}. 

For completeness, see Figure~\ref{backward_gait} for a
 backward tetrapod gait of the interconnected bursting neuron model,  
 when $\delta = 0.01$. The initial conditions are as follows:
 \be{IC:backward}
v_1(0)= -40, \;
v_2(0)=-40,  \;
v_3(0)=-30, \;
v_4(0)=10, \; 
v_5(0)=5, \;
v_6(0)=-20, 
  \ee
  and for $i=1,\ldots,6$, $m_i$, $\w_i$, and $s_i$ are as in Equation~(\ref{IC:numerical:tetra:tri:m,w,s:delta}). The coupling strengths $\c_i$ are as in Equation~(\ref{coupling:numerical:tetra:tri}). 
\begin{figure}[h!]
\begin{center}
\includegraphics[scale=.25]{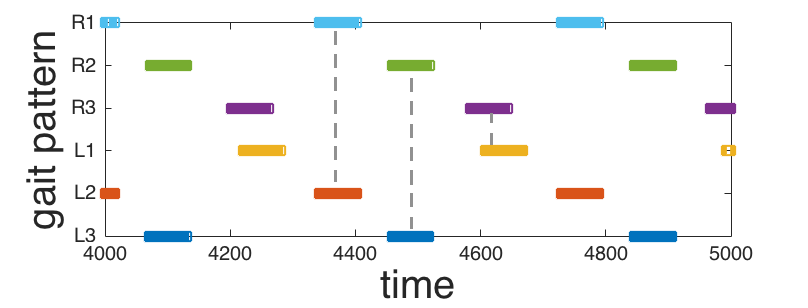}
\caption{Interconnected bursting neuron model: 
backward tetrapod gait for  $\delta =0.01$. 
}
\label{backward_gait}
\end{center}
\end{figure}

\ZA{Recall from Section \ref{application_BN} 
(Figures~\ref{phase differences model_near_one_del} and \ref{phase differences model_near_0_del}),
when $\a=0.95\approx1$, a stable 
backward tetrapod gait exists, but a stable forward tetrapod exists for $\a=0.032\ll1$. 
Since $\a=\frac{\c_4}{\c_4+\c_7}$, and $\c_5\approx\c_6$ if $\c_1\approx\c_2\approx \c_3$, 
this suggests that when couplings from front to hind legs are strong ($\c_4,\c_6\gg\c_7$), 
we expect to see backward tetrapod gaits, but when couplings from hind to front legs are strong 
($\c_5,\c_7\gg\c_4$), forward tetrapod gaits would be observed. 
Similarly, in \cite{Ermentrout_Kopell_SIAM}, a lamprey model suggested that the tail-to-head neural connections along the spinal cord would be stronger than those running from head to tail, despite the fact that the wave associated with swimming travels from head to tail. That prediction was later confirmed experimentally in \cite{SIGVARDT199237}. 
See Figure~\ref{vaibhav_figure_nullcline_phase_plane} (left) for examples of 
coexisting stable  backward and forward tetrapod gaits in a 
phase plane plot obtained from fitted fruit fly data. Backward tetrapod gaits have been 
observed in backward-walking flies, but have not been seen in forward-walking flies 
\cite[Supplementary Materials, Figure S1]{Bidaye-Science14}. 
}

{In the introduction we mentioned related work of Yeldesbay et. al.
\cite{YeldHolTothDaun_2016,YeldTothDaun_2017} in which a
non-bursting half center oscillator model for the CPG contained
in three ipsilateral segments is reduced to a set of
ipsilateral phase oscillators with unidirectional coupling running
from front to middle to hind and returning to front leg units. 
Tetrapod, tripod and transition gaits were also found in their work, although 
the cyclic architecture is strikingly different from our nearest neighbor coupling and
it involves excitatory and inhibitory proprioceptive feedback. 
It is therefore interesting to see that
similar gaits appear in both reduced models, although the bifurcations
exhibited in \cite{YeldHolTothDaun_2016,YeldTothDaun_2017}
appear quite different from those illustrated here in
Figures~(\ref{Bifurcation_Happ_del}) and (\ref{Bifurcation_alpha_13_12_matcont}).
Moreover, gait transitions occur in response to changes in feedback
 as well as to changes in stepping frequency.}

In summary, we have shown that multiple tetrapod gaits exist and can be stable, and described the transitions 
in which they approach tripod gaits as speed increases. 
In studying the phase reduced system on the 2-torus, we move from the special cases of Section \ref{existence_and_stability}, 
in which coupling strengths are balanced and other constraints apply, to the experimentally estimated data sets of
Section \ref{DATA} in which the detailed dynamics differ but tetrapod to tripod transitions still occur.

\section*{Acknowledgements}

This work was jointly supported by NSF-CRCNS grant DMS-1430077
and the National Institute of Neurological Disorders and Stroke of the
National Institutes of Health under Award U01-NS090514-01. The
content is solely the responsibility of the authors and does not 
necessarily represent the official views of the National Institutes of Health.
We thank Michael Schwemmer for sharing his Matlab code for adjoint iPRC computations,
Cesar Mendes and Richard Mann for providing fruit
fly locomotion data and Einat Couzin for sharing her values
of coupling strengths fitted to that data. \ZA{We also thank the anonymous  reviewers for 
their insightful comments and suggestions.}

\section{\ZAA{Appendix}}\label{Appendix}

Here, we  review the theory of weakly coupled oscillators which can 
reduce the dynamics of each neuron to a single first order
ODE describing the phase of the neuron. 
In Section~\ref{phase_reduction}, we  applied this method to the coupled bursting neuron models 
to reduce the 24 ODEs to 6 phase oscillator equations. 

Let the ODE
\be{single:neuron:appendix}
\dot X = f(X), \qquad X\in\r^n,
\ee
describe the dynamics of a single neuron. In our model, $X=(v,m,\w,s)^T$
and $f(X)$ is as the right hand side of Equations~(\ref{BN}).  Assume that
Equation~(\ref{single:neuron:appendix}) has an attracting hyperbolic limit cycle
{$\Gamma = \Gamma(t)$}, with period $T$ and frequency $\o=2\pi/T$. 

 The phase of a neuron is the time that has elapsed as its state
 moves around $\Gamma$, starting from an arbitrary reference point
 in the cycle. We define the phase of the periodically firing neuron at
 time $t$ to be
 \be{phase:def:appendix}
 \phi(t) = \omega t + \bar\phi \quad \mbox{mod T}.
\ee
The constant $\bar\phi$, which is called the relative phase, is determined
by the state of the neuron on $\Gamma$ at time $t=0$.
Note that by the definition of phase, Equation~(\ref{single:neuron:appendix}) 
for a single neuron is reduced to the scalar equation
\be{phase:equ:single:neuron:appendix}
\frac{d\phi}{dt} = \omega, 
\ee
while the dynamics of its relative phase are described by
\be{relative:phase:equ:single:neuron:appendix}
\frac{d\bar\phi}{dt} = 0.  
\ee 
Now consider the system of weakly coupled identical neurons 
\be{coupled:neurons:appendix}
\bal
&\dot X_1= f(X_1) + \epsilon g(X_1,X_2),\\
&\dot X_2= f(X_2) + \epsilon g(X_2,X_1),
\eal
\ee
where $0<\epsilon\ll1$ is the coupling strength and $g$ is the coupling function.
For future reference, recall that neurons are coupled only via their voltage
variables; see Equation~(\ref{cell1}). When a neuron is perturbed by
synaptic currents from other neurons or by other external stimuli,
its dynamics no longer remain on the limit cycle $\Gamma$, and the relative 
phase $\bar\phi$ is not constant. However, when perturbations are sufficiently weak, the 
intrinsic dynamics dominate, ensuring that the perturbed system
remains close to $\Gamma$ with frequency close to $\omega$. 
Therefore, we can approximate the solution of neuron $j$ by 
\be{Gamma:limit:cycle:appendix}
X_j(t) = \Gamma(\omega t+\bar\phi_j(t)),
\ee
where the relative phase  $\bar\phi_j(t)$ is now a function of time $t$.
Over each cycle of the oscillations, the weak perturbations to the
neurons produce only small changes in  $\bar\phi_j(t)$. These
changes are negligible over a single cycle, but they can slowly
accumulate over many cycles and produce substantial effects on
the relative firing times. The goal now is to understand how the
relative phases $\bar\phi_j(t)$ of the coupled neurons evolve. 
 
 To do this, we first review the concept of an 
infinitesimal phase response curve (iPRC), $Z(\phi)$,  and then 
we show how to derive the phase equation  given in 
Equation~(\ref{take:average:3}) from Equation~(\ref{coupled:neurons}). 
For details see \cite{SIAM2, Schwemmer2012}; specifically, we
borrow some material from \cite{Schwemmer2012}. 

Intuitively, an iPRC \cite{Winfree2001}
of an oscillating neuron measures the phase shifts in response to small
brief perturbations (Dirac $\delta$  function)  delivered at different times
in its limit cycle and acts like a Green's function for the oscillating neurons.
Below, we will give a precise mathematical definition of the iPRC and
explain how we compute it in our model. 

Suppose that a small brief rectangular current pulse of amplitude
$\epsilon I$ and duration $\Delta t$ is applied to a neuron
at phase $\phi$, i.e.,  the total charge applied to the cell by the stimulus 
is equal to $\epsilon I \Delta t$. Then the membrane potential $v$ changes by 
$\Delta v = \epsilon I \Delta t /C$.
Depending on the amplitude and duration of the stimulus and the phase
in the oscillation at which it is applied, the cell may fire sooner (phase advance) 
or later (phase delay) than it would
have fired without the perturbation. 
For sufficiently small and brief stimuli, the neuron will respond in an
approximately linear fashion, and the iPRC in the direction of $v$, denoted by
$Z_v$, scales linearly with the magnitude of the current stimulus in the limit 
$\Delta v\to 0$:
\be{iPRC:appendix}
 Z_v(\phi) :=\lim_{\Delta v\to 0}  \frac{\Delta \phi(\phi)}{\Delta v}. 
\ee
Note that $Z_v$ only captures the response to perturbations in the direction
of the membrane potential $v$. However, such responses can be
computed for perturbations in any direction in state space. 

There is a one to one correspondence between phase $\phi$ and each point
$x$ on the limit cycle $\Gamma$.  The phase map $\Phi$ on $\Gamma$ is
defined as follows. 
\be{phase:map:appendix}
\Phi(x(t)) :=  \phi(t) = \omega t+\bar\phi \quad \mbox{mod T},
\ee
which implies that
\be{phase:map:gradient:appendix}
\nabla_x\Phi\cdot \dot X =  \nabla_x\Phi\cdot f  = \omega. 
\ee
The phase map is well defined for all points  on $\Gamma$. For any
asymptotically stable limit cycle, we can extend the domain of the phase map
to points in the domain of attraction of the limit cycle. If $x$ is a point on
$\Gamma$ and $y$ is a point in a neighborhood of $\Gamma$, then we
say that $y$ has the same \textit{asymptotic phase} as $x$ if
\[\| X(t,x) - X(t,y)\| \to 0 \qquad \mbox{as $t\to\infty$},\] 
where $X(\cdot,x)$ is the unique solution of Equation~(\ref{single:neuron:appendix})
with initial condition $x$. Note that with $x\in\Gamma$, $X(t,x) = \Gamma 
(\omega t + \bar\phi)$, for some $\bar\phi$. 
This means that the solution starting at the initial point $y$ in a sufficiently
small neighborhood of $\Gamma$ converges to the solution starting at
the point $x \in \Gamma$ as $t \rightarrow \infty$, so that $\Phi(x)
= \Phi(y)$. The set of all points in the neighborhood of $\Gamma$ that
have the same asymptotic phase as the point $x \in \Gamma$ is called
the \textit{isochron} for phase  $\phi = \Phi(x)$ \cite{Winfree2001, Guck75}.

\begin{figure}[h!]
 \begin{center}
\includegraphics[scale=.3]{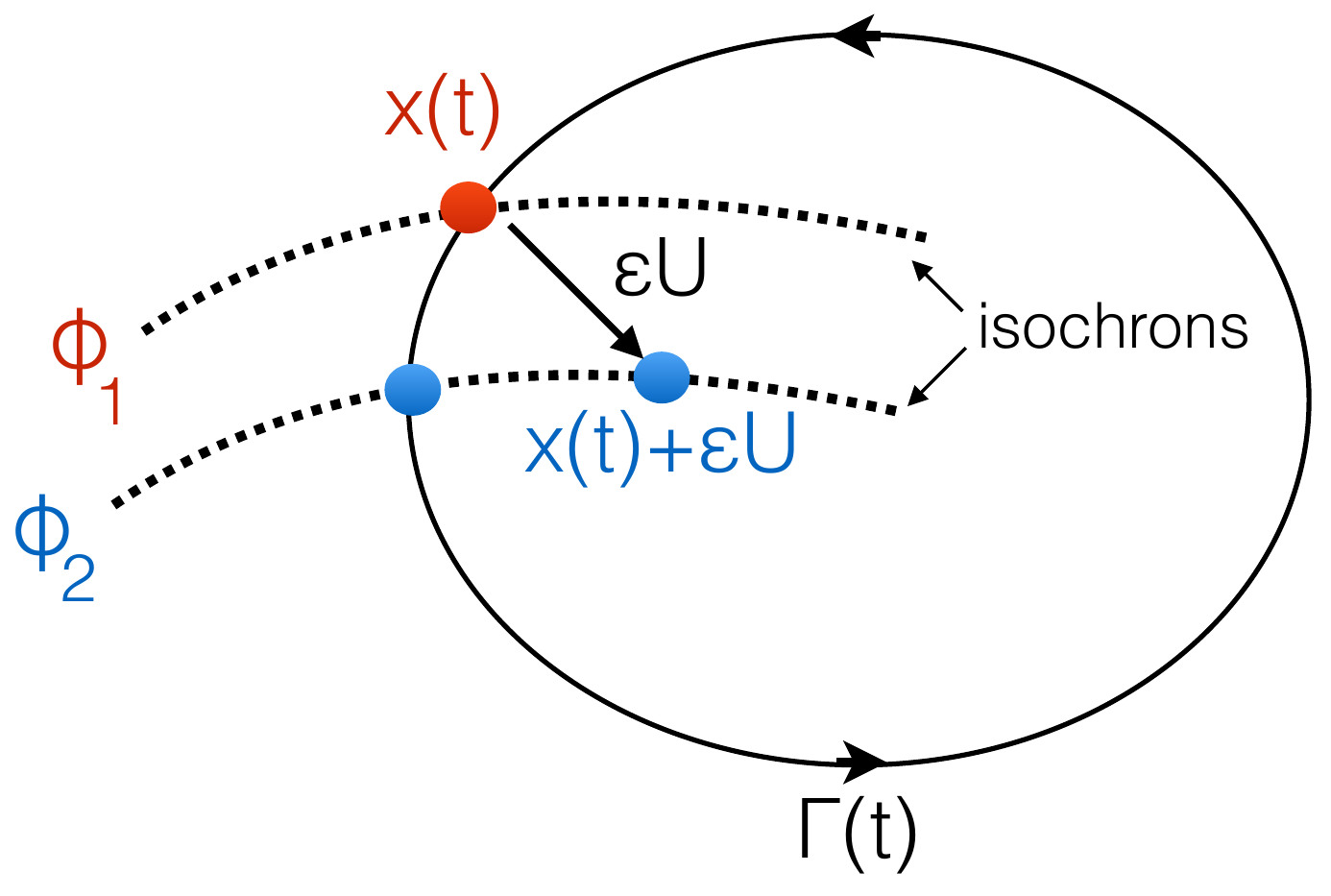}
\caption{ Isochrons and asymptotic phase.}
\label{isochron}
\end{center}
\end{figure}

Given the concepts of isochron and asymptotic phase,  we show that the
gradient of the phase map $\Phi$ is the vector iPRC, i.e., its components
are the iPRCs for every variable in Equation~(\ref{single:neuron:appendix}). 
Suppose that, at time $t$, the neuron is in  state $x(t)\in \Gamma(t)$ with
corresponding phase $\phi_1(t)$:
\[
 \Phi(x(t)) = \phi_1(t) = \omega t +\bar \phi_1(t).
 \]
At this time, it receives a small abrupt external perturbation $\epsilon U$ 
with magnitude $\epsilon$, where $U$ is the unit vector in the direction
of the perturbation in state space. Immediately after the perturbation, the
neuron is in the state $x(t)+\epsilon U$ and its new ``asymptotic phase" is  
\[ 
\Phi(x(t)+\epsilon U) = \phi_2(t) = \omega t +\bar \phi_2(t).
\]
See Figure \ref{isochron} for an illustration.  Using Taylor series,
\be{eq1PRC:appendix}
\phi_2(t) - \phi_1(t) =  \Phi(x(t)+\epsilon U) -  \Phi(x(t)) = \nabla_x\Phi(x(t))\cdot \epsilon U +O\lt(\epsilon^2\rt),
\ee
and dividing by $\epsilon,$ we obtain
\be{eq2PRC:appendix}
\dis\frac{\phi_2(t) - \phi_1(t)}{\epsilon}  = \nabla_x\Phi(x(t))\cdot U+O\lt(\epsilon\rt),
\ee
and therefore, by the definition of iPRC, as $\epsilon\to0$, the left hand 
side of Equation~(\ref{eq2PRC:appendix}) is the iPRC at $\phi_1(t)$ in the direction of $U$: 
\be{iPRC:gradient:appendix}
Z(\phi_1(t)) \cdot U  = \nabla_x \Phi(x(t))\cdot U. 
\ee
Hence, for any point on the limit cycle $\Gamma$, $Z= \nabla_x \Phi$. 

The iPRCs can also be computed from an adjoint formulation
 \cite{Schwemmer2012, Erment-Terman10}, which is the method adopted here. Specifically,
the iPRC $Z$ is a $T$-periodic solution of the adjoint equation of Equation~(\ref{single:neuron:appendix}), i.e., 
\be{adjoint:appendix}
\frac{dZ}{dt} = - [J_f(\Gamma)]^T \; Z, 
\ee
subject to the constraint that makes $Z(\phi_1(t))$ normal to the limit cycle $\Gamma(t)$ at $t=0$:
\be{normalization:appendix}
Z(0) \cdot \Gamma'(0) =0. 
\ee
In Equation~(\ref{adjoint:appendix}), $J_f(\Gamma)=D_f(\Gamma)$ is the linearization
of Equation~(\ref{single:neuron:appendix}) around the limit cycle $\Gamma$ and 
$\Gamma'(0)$ denotes the vector  tangent to the limit cycle at time $t=0$:
$\Gamma'(0) = f(x(0))\mid_{x \in \Gamma}$. 
Note that the adjoint system (\ref{adjoint:appendix}) has the
opposite stability of the original system~(\ref{single:neuron:appendix}), which has
an asymptotically stable solution $\Gamma$.
Thus, to obtain the unstable periodic
solution of Equation~(\ref{adjoint:appendix}), we integrate 
backwards in time from an arbitrary initial condition.
To obtain the iPRC, we normalize
the periodic solution using Equation~(\ref{normalization:appendix}).

There is a direct way to relate the gradient of the phase map  to the
solution of the adjoint equation~(\ref{adjoint:appendix}). In fact, $\nabla_x 
\Phi(\Gamma(t))$ satisfies the adjoint equation (\ref{adjoint:appendix}) and the
normalization condition Equation~(\ref{normalization:appendix}), \cite{Brown2004}.
Figure~\ref{H_PRC_versus_delta} and~\ref{H_PRC_versus_Iext}
(first rows) show $Z_v$, the first component of the vector iPRC $Z$ computed
by the adjoint method,  of the bursting neuron model for different values
of $\delta$,  and $I_{ext}$, respectively. 

Now consider the system of weakly coupled identical neurons introduced in Equation~(\ref{coupled:neurons:appendix}). As we discussed earlier, our goal is to  understand how the relative phase $\bar\phi_j(t)$ of the coupled neurons evolves slowly in time. 
For $i=1,2$, let  $X_i(t)$ be solutions of Equation~(\ref{coupled:neurons:appendix}) with corresponding phases
\[\phi_i(t) := \Phi(X_i(t)) = \omega t+\bar \phi_i(t).\] 
Then by taking the derivative of $\phi_i$ and using Equations~(\ref{coupled:neurons:appendix}), (\ref{Gamma:limit:cycle:appendix}), (\ref{phase:map:gradient:appendix}), and (\ref{iPRC:gradient:appendix}), we obtain: 
\begin{subequations}\label{deriv_coupling}
\begin{align}
\dis\frac{d\phi_i}{dt} (t)&= \nabla_x\Phi(X_i(t)) \cdot \dot{X}_i\\
&= \nabla_x\Phi(X_i(t)) \cdot \lt[f(X_i(t)) + \epsilon g(X_i,X_j)\rt]\\
&\approx  \nabla_x\Phi(\Gamma(\omega t+\bar\phi_i(t))) \cdot \lt[f(\Gamma(\omega t+\bar\phi_i(t))) + \epsilon g(\Gamma(\omega t+\bar\phi_i(t)),\Gamma(\omega t+\bar\phi_j(t)))\rt]\\
& = \omega + \epsilon Z (\Gamma(\omega t+\bar\phi_i(t))) \cdot g(\Gamma(\omega t+\bar\phi_i(t)),\Gamma(\omega t+\bar\phi_j(t))).
\end{align}
\end{subequations}
Using the change of variables $\phi_i(t) = \omega t + \bar \phi_i(t)$, 
we get the following dynamics for ${d\bar\phi_i}/{dt}$
\be{take:average:1}
\dis\frac{d\bar\phi_i}{dt}(t) = \epsilon Z (\Gamma(\omega t+\bar\phi_i(t))) \cdot g(\Gamma(\omega t+\bar\phi_i(t)),\Gamma(\omega t+\bar\phi_j(t))).
\ee
Now letting $\tilde t := \omega t+\bar\phi_i(t)$ and taking the average of
the right hand side of Equation~(\ref{take:average:1}) over one 
unperturbed period and using  the Averaging Theorem \cite[Section 4.1]{Holmes_book},
we obtain the following equation for the relative phase $\bar\phi_i$. 
\be{take:average:2}
\dis\frac{d\bar\phi_i}{dt} = \frac{\epsilon}{T}\int_0^T Z (\Gamma(\tilde t)) \cdot g(\Gamma(\tilde t),\Gamma(\tilde t +\bar\phi_j(t) -\bar\phi_i(t) )) \; d\tilde t=: \epsilon H(\bar\phi_j(t) -\bar\phi_i(t) ),
\ee
where 
\[H = H(\theta) =  \frac{1}{T}\int_0^T Z (\Gamma(\tilde t)) \cdot g(\Gamma(\tilde t),\Gamma(\tilde t +\theta )) \; d\tilde t,\]
 is the coupling function: the convolution  of the synaptic current  input to the neuron via coupling $g$ and the neuron's iPRC $Z$. 
Using $\phi_i(t) = \omega t+\bar \phi_i(t)$ and Equation~(\ref{take:average:2}), we can write the phase equation of each neuron instead of relative phase equations, 
\be{}
\dis\frac{d\phi_i}{dt}(t) =\omega + \epsilon H(\phi_j(t) -\phi_i(t)),
\ee
where $\epsilon$ denotes the coupling strength (cf. Equation~(\ref{take:average:2})).



\begin{thebibliography}{10}

\bibitem{Wilson66}
D.M. Wilson.
\newblock Insect walking.
\newblock {\em Ann. Rev. Entomol.}, pages 103--122, 1966.

\bibitem{Graham85}
D.~Graham.
\newblock Pattern and control of walking in insects.
\newblock {\em Adv. Insect Physiology}, 18:31--140, 1985.

\bibitem{Mendes_eLife_2013}
C.S. Mendes, I.~Bartos, T.~Akay, S.~M\'{a}rka, and R.S. Mann.
\newblock Quantification of gait parameters in freely walking wild type and
  sensory deprived \emph{Drosophila melanogaster}.
\newblock {\em eLife}, 2:e00231, 2013.

\bibitem{Fuchs14}
E.~Couzin-Fuchs, T.~Kiemel, O.~Gal, A.~Ayali, and P.~Holmes.
\newblock Intersegmental coupling and recovery from perturbations in freely
  running cockroaches.
\newblock {\em J. Exp. Biol.}, 218(2):285--297, 2015.

\bibitem{AYALI20151}
A.~Ayali, A.~Borgmann, A.~Buschges, E.~Couzin-Fuchs, S.~Daun-Gruhn, and
  P.~Holmes.
\newblock The comparative investigation of the stick insect and cockroach
  models in the study of insect locomotion.
\newblock {\em Current Opinion in Insect Science}, 12:1 -- 10, 2015.

\bibitem{Golubitsky1999}
M.~Golubitsky, I.~Stewart, P-L. Buono, and J.J. Collins.
\newblock Symmetry in locomotor central pattern generators and animal gaits.
\newblock {\em Nature}, 401:693--695, 1999.

\bibitem{Collins1993}
J.J. Collins and I.N. Stewart.
\newblock Coupled nonlinear oscillators and the symmetries of animal gaits.
\newblock {\em Journal of Nonlinear Science}, 3(1):349--392, Dec 1993.

\bibitem{SIAM2}
R.M. Ghigliazza and P.~Holmes.
\newblock A minimal model of a central pattern generator and motorneurons for
  insect locomotion.
\newblock {\em SIAM J. Appl. Dyn. Sys.}, 3(4):671--700, 2004.

\bibitem{Delcomyn_1971}
F.~Delcomyn.
\newblock {The locomotion of the cockroach {\it Periplaneta americana}}.
\newblock {\em J. Exp. Biol.}, 54 (2):725--744, 1971.

\bibitem{pearson_Iles_1973}
K.G. Pearson and J.F. Iles.
\newblock Nervous mechanisms underlying intersegmental co-ordination of leg
  movements during walking in the cockroach.
\newblock {\em J. Exp. Biol.}, 58:725--744, 1973.

\bibitem{YeldHolTothDaun_2016}
A.~Yeldesbay, P.~Holmes, T.~T{\'o}th, and S.~Daun.
\newblock Phase reduction of an inter-segmental network model of stick insect
  locomotion.
\newblock 2016.
\newblock Poster presented at Advances in the collective behaviour of complex
  systems, University of Potsdam, Sept. 1-3.

\bibitem{YeldTothDaun_2017}
A.~Yeldesbay, T.~T{\'o}th, and S.~Daun.
\newblock Phase reduction of an inter-segmental network model of multi-legged
  locomotion.
\newblock In preparation, 2017.

\bibitem{SIAM1}
R.M. Ghigliazza and P.~Holmes.
\newblock Minimal models of bursting neurons: The effects of multiple currents
  and timescales.
\newblock {\em SIAM J. Appl. Dyn. Sys.}, 3(4):636--670, 2004.

\bibitem{Marder_Bucher_CPG_review}
E.~Marder and D.~Bucher.
\newblock Central pattern generators and the control of rhythmic movements.
\newblock {\em Current Biology}, 11(23):R986 -- R996, 2001.

\bibitem{Ijspeert_CPG_review}
A.J. Ijspeert.
\newblock Central pattern generators for locomotion control in animals and
  robots: A review.
\newblock {\em Neural Networks}, 21(4):642 -- 653, 2008.

\bibitem{pearson_Iles_1970}
K.G. Pearson and J.F. Iles.
\newblock Discharge patterns of coxal levator and depressor motoneurones of the
  cockroach, \emph{Periplaneta americana}.
\newblock {\em J. Exp. Biol.}, 52:139--165, 1970.

\bibitem{pearson_1972}
K.G. Pearson.
\newblock Central programming and reflex control of walking in the cockroach.
\newblock {\em J. Exp. Biol.}, 56:173--193, 1972.

\bibitem{TothKnopsDaun-G-JNPhys12}
T.I. T{\'o}th, S.~Knops, and S.~Daun-Gruhn.
\newblock A neuromechanical model explaining forward and backward stepping in
  the stick insect.
\newblock {\em J. Neurophysiol.}, 107:3267--3280, 2013.

\bibitem{Ashwin-JMatNSci16}
P.~Ashwin, S.~Coombes, and R.~Nicks.
\newblock Mathematical frameworks for oscillatory network dynamics in
  neuroscience.
\newblock {\em J. Math. Neurosci.}, 6 (2):1--92, 2016.

\bibitem{Zhang2017}
C.~Zhang and T.J. Lewis.
\newblock Robust phase-waves in chains of half-center oscillators.
\newblock {\em J. Math. Biol.}, 74(7):1627--1656, 2017.

\bibitem{Holmes_book}
J.~Guckenheimer and P.~Holmes.
\newblock {\em Nonlinear Oscillations, Dynamical Systems, and Bifurcations of
  Vector Fields}.
\newblock Springer-Verlag, 6th edition, 2002.

\bibitem{Couzin_notes_16}
E.~Couzin.
\newblock Analysis of free-walking fly data.
\newblock Unpublished notes, 2016.

\bibitem{Matcont}
A.~Dhooge, W.~Govaerts, and Yu.A. Kuznetsov.
\newblock {M}atcont: A {M}atlab package for numerical bifurcation analysis of
  {ODE}s.
\newblock {\em ACM TOMS}, 29:141--164, 2003.

\bibitem{alspach1972nonlinear}
D.~Alspach and H.~Sorenson.
\newblock Nonlinear {B}ayesian estimation using {G}aussian sum approximations.
\newblock {\em IEEE Transactions on Automatic Control}, 17(4):439--448, 1972.

\bibitem{Wojcik-PLoSone14}
J.~Wojcik, J.~Schwabedal, R.~Clewley, and A.L. Silnikov.
\newblock Key bifurcations of bursting polyrhythms in 3-cell central pattern
  generators.
\newblock {\em PLoS ONE}, 9 (4):e92918, 2014.

\bibitem{Barrio-EuroPhysLet15}
R.~Barrio, M.~Rodr{\'{i}}guez, S.~Serrano, and A.~Silnikov.
\newblock Mechanism of quasi-periodic lag jitter in bursting rhythms by a
  neuronal network.
\newblock {\em European Phys. Let.}, 112:38002, 2015.

\bibitem{Lozano-SciRepts16}
A.~Lozano, M.~Rodr{\'{i}}guez, and R.~Barrio.
\newblock Control strategies of 3-cell central pattern generator via global
  stimuli.
\newblock {\em Scientific Reports}, 6:23622, 2016.

\bibitem{Aguiar-JNS11}
M.~Aguiar, P.~Ashwin, A.~Dias, and M.~Field.
\newblock Dynamics of coupled cell networks: {S}ynchrony, heteroclinic cycles
  and inflation.
\newblock {\em J. Nonlinear Sci.}, 21:271--323, 2011.

\bibitem{Ermentrout_Kopell_SIAM}
G.B. Ermentrout and N.~Kopell.
\newblock Frequency plateaus in a chain of weakly coupled oscillators.
\newblock {\em SIAM Journal on Mathematical Analysis}, 15(2):215--237, 1984.

\bibitem{SIGVARDT199237}
K.A. Sigvardt and T.L. Williams.
\newblock Models of central pattern generators as oscillators: {T}he lamprey
  locomotor {CPG}.
\newblock {\em Seminars in Neuroscience}, 4(1):37 -- 46, 1992.

\bibitem{Bidaye-Science14}
S.S. Bidaye, C.~Machacek, Y.~Wu, and B.J. Dickson.
\newblock Neuronal control of {\it drosophila} walking direction.
\newblock {\em Science}, 344:97--101, 2014.

\bibitem{Schwemmer2012}
M.A. Schwemmer and T.J. Lewis.
\newblock The theory of weakly coupled oscillators.
\newblock In N.W. Schultheiss, A.A. Prinz, and R.J. Butera, editors, {\em Phase
  Response Curves in Neuroscience: Theory, Experiment, and Analysis}, pages
  3--31. Springer, New York, 2012.

\bibitem{Winfree2001}
A.T. Winfree.
\newblock {\em The Geometry of Biological Time}, volume~12 of {\em
  Interdisciplinary Applied Mathematics}.
\newblock Springer-Verlag, New York, second edition, 2001.

\bibitem{Guck75}
J.~Guckenheimer.
\newblock Isochrons and phaseless sets.
\newblock {\em J. Math. Biol.}, 1:259--273, 1975.

\bibitem{Erment-Terman10}
G.B. Ermentrout and D.~Terman.
\newblock {\em Mathematical Foundations of Neuroscience}.
\newblock Springer, New York, 2010.

\bibitem{Brown2004}
E.~Brown, J.~Moehlis, and P.~Holmes.
\newblock On the phase reduction and response dynamics of neural oscillator
  populations.
\newblock {\em Neural Comput.}, 16(4):673--715, April 2004.

\end{thebibliography}
\end{document}